\documentclass[11pt]{article}%
\usepackage{amssymb,amsmath,amsfonts,amsthm,array,bm,bbm,color}%
\usepackage{amsmath}%
\usepackage{amsfonts}%
\usepackage{amssymb}%
\usepackage{graphicx}
\providecommand{\U}[1]{\protect\rule{.1in}{.1in}}

\numberwithin{equation}{section}

\newtheorem{theorem}{Theorem}[section]
\newtheorem{lemma}[theorem]{Lemma}
\newtheorem{corollary}[theorem]{Corollary}
\newtheorem{proposition}[theorem]{Proposition}
\newtheorem{remark}[theorem]{Remark}

\newtheorem{definition}[theorem]{Definition}

\def\<{\langle}
\def\>{\rangle}
\def\d{{\rm d}}
\def\L{\mathcal{L}}
\def\div{{\rm div}}
\def\E{\mathbb{E}}

\def\N{\mathbb{N}}
\def\P{\mathbb{P}}
\def\R{\mathbb{R}}
\def\T{\mathbb{T}}
\def\Z{\mathbb{Z}}

\def\eps{\varepsilon}

\begin{document}

\title{Stochastic mSQG equations with multiplicative transport noises: white noise solutions and scaling limit}

\author{Dejun Luo\footnote{Email: luodj@amss.ac.cn. Key Laboratory of RCSDS, Academy of Mathematics and Systems Science, Chinese Academy of Sciences, Beijing 100190, China, and School of Mathematical Sciences, University of the Chinese Academy of Sciences, Beijing 100049, China. }
\quad Rongchan Zhu\footnote{Email: zhurongchan@126.com. Department of Mathematics, Beijing Institute of Technology, Beijing 100081, China and Department of Mathematics, University of Bielefeld, D-33615 Bielefeld, Germany.} }

\maketitle

\begin{abstract}
We consider the modified Surface Quasi-Geostrophic (mSQG) equation on the 2D torus $\mathbb T^2$, perturbed by multiplicative transport noise. The equation admits the white noise measure on $\T^2$ as the invariant measure. We first prove the existence of white noise solutions to the stochastic equation via the method of point vortex approximation, then, under a suitable scaling limit of the noise, we show that the solutions converge weakly to the unique stationary solution of the dissipative mSQG equation driven by space-time white noise. The weak uniqueness of the latter equation is also proved by following Gubinelli and Perkowski's approach in \cite{GP-18}.
\end{abstract}


\section{Introduction}

Let $\T^2=\R^2/\Z^2$ be the 2D torus. The modified Surface Quasi-Geostrophic (mSQG) equation on $\T^2$ reads as
  \begin{equation}\label{mSQG}
  \left\{ \aligned
  & \partial_t \xi+ u\cdot \nabla\xi =0, \\
  & u= \nabla^\perp (-\Delta)^{-(1+\eps)/2} \xi,
  \endaligned
  \right.
  \end{equation}
where $\nabla^\perp=(\partial_{x_2}, -\partial_{x_1})$, $\Delta$ is the usual Laplacian operator and $\eps\in [0,1]$. The mSQG equation links the SQG equation ($\eps=0$) with the vorticity form of 2D Euler equation ($\eps=1$). The SQG equation is a well known model in geophysics, used to describe the temperature in a rapidly rotating stratified fluid, see for instance \cite{CMT, HPGS, SBHMTHV} for the geophysical background. We refer to the introduction of \cite{FS} for a detailed list of well posedness results on the equation \eqref{mSQG}. In the smooth setting, since the velocity field $u$ is divergence free, the $L^2$-norm of solutions to \eqref{mSQG} is formally preserved by the dynamics. This implies that the white noise measure $\mu$ (also called the enstrophy measure) on $\T^2$ is invariant under the dynamics of \eqref{mSQG}. The measure $\mu$ has the heuristic expression
  $$\mu(\d\xi)= \frac1{Z} \exp\bigg(-\frac12 \int_{\T^2} |\xi(x)|^2\,\d x\bigg) \d\xi,$$
and is supported on $H^{-1-}(\T^2) = \cap_{s<-1} H^s(\T^2)$, where $H^s(\T^2)$ is the usual Sobolev space on $\T^2$. In this paper, the elements in $H^s(\T^2)$ will be assumed to have zero average on $\T^2$. We say that a stochastic process $\xi:\Omega\to C\big( [0,T], H^{-1-}\big)$, defined on a probability space $(\Omega, \mathcal F,\P)$, is a white noise solution to \eqref{mSQG} if for any $t\in [0,T]$, $\xi_t$ has the law $\mu$ and for all $\phi \in C^\infty(\T^2)$, $\P$-a.s. for all $t\in [0,T]$, it holds
  \begin{equation}\label{mSQG-solution}
  \<\xi_t, \phi\>= \<\xi_0, \phi\> +\int_0^t \big\<\xi_s\otimes \xi_s, H^\phi_\eps \big\>\,\d s,
  \end{equation}
where $\<\cdot, \cdot\>$ is the duality between distributions and smooth functions and for all $x,y\in \T^2$,
  \begin{equation}\label{H-eps-phi}
  H_\eps^\phi(x,y)= {\bf 1}_{\{x\neq y\}}\, \frac12 K_\eps(x-y)\cdot (\nabla\phi(x)- \nabla\phi(y)),
  \end{equation}
in which $K_\eps$ is the kernel on $\T^2$ corresponding to the operator $\nabla^\perp (-\Delta)^{-(1+\eps)/2}$. Near the origin, the kernel $K_\eps$ has the approximative expression
  \begin{equation}\label{K-eps}
  K_\eps(x)\sim \frac{x^\perp}{|x|^{3-\eps}}.
  \end{equation}
This easily implies that there is a constant $C>0$ such that for all $\phi \in C^\infty(\T^2)$,
  \begin{equation}\label{H-eps-phi-estimate}
  \big| H_\eps^\phi(x,y) \big|\leq \frac{C \|\nabla^2\phi \|_\infty}{|x-y|^{1-\eps}}.
  \end{equation}
Due to the singularity of $H_\eps^\phi$, some careful analysis have to be done in order to give a meaning to the nonlinear term $\big\<\xi_s\otimes \xi_s, H^\phi_\eps \big\>$ in \eqref{mSQG-solution} for $\eps\in (0,1]$, see \cite[Section 2.2]{FS} or Section \ref{subsec-definition} below. For $\eps=0$ we are unable to make sense of the nonlinear term and the difficulty comes from the singularity of the kernel $K_\eps$ near the origin.

In the case $\eps=1$, namely, for the 2D Euler equation on the torus, the existence of white noise solutions has been known for a long time, thanks to Albeverio and Cruzeiro's work \cite{AC} which is based on the Galerkin approximation. The same result  was proved by Flandoli in the recent paper \cite{Flandoli18} via the method of point vortex approximation. This powerful method also allows him to establish other results, for instance, it was shown in \cite[Theorem 1]{Flandoli18} that the white noise solutions can be approximated by $L^\infty$-solutions. Later on, following the method of \cite{Flandoli18}, Flandoli and Saal proved in \cite{FS} the existence of white noise solutions to the mSQG equation \eqref{mSQG}, see also \cite{NPGT} for relevant results via the method of Galerkin approximation.

Motivated by the recent paper \cite{FlaLuo-1}, we consider in the white noise regime the mSQG equation perturbed by transport noise:
  \begin{equation}\label{stoch-mSQG}
  \left\{ \aligned
  &\d\xi + u\cdot\nabla \xi \,\d t+ \sum_{k\in \Z^2_0} \theta_k\, \sigma_k\cdot\nabla \xi \circ\d W^k_t=0, \\
  & u= \nabla^\perp (-\Delta)^{-(1+\eps)/2} \xi,
  \endaligned
  \right.
  \end{equation}
where $\Z^2_0= \Z^2\setminus \{0\}$ and $\theta\in \ell^2:= \ell^2(\Z^2_0)$ is the collection of square summable sequences indexed by $\Z^2_0$; moreover, $\{\sigma_k \}_{k\in \Z^2_0}$ is a CONS of the space
  $$H:=\bigg\{f\in L^2(\T^2;\R^2);\div f=0, \int f\, \d x=0 \bigg\}$$
and $\{W^k \}_{k\in \Z^2_0}$ a family of independent complex Brownian motions, see the next section for explicit definitions. We always assume that $\theta\in \ell^2$ is radially symmetric, namely,
  \begin{equation}\label{theta-symmetry}
  \theta_k =\theta_l \quad \mbox{whenever } |k|=|l|.
  \end{equation}
The first equation in \eqref{stoch-mSQG} has the following It\^o form:
  \begin{equation}\label{stoch-mSQG-Ito}
  \d\xi + u\cdot\nabla \xi \,\d t= \frac12 \|\theta \|_{\ell^2}^2\, \Delta \xi\,\d t - \sum_{k\in \Z^2_0} \theta_k\, \sigma_k\cdot\nabla \xi \,\d W^k_t,
  \end{equation}
which can be easily deduced using the equalities \eqref{key-identity} and \eqref{BMs} below. The nonlinear part is understood in the same way as \eqref{mSQG-solution}, with suitable interpretations of other terms. Following the method of \cite{Flandoli18, FlaLuo-1}, we shall first consider the stochastic point vortex dynamics associated to \eqref{stoch-mSQG}, and prove the existence of stationary white noise solutions to the equation \eqref{stoch-mSQG}, by letting the number of vortices go to infinity.

\begin{theorem}[Existence] \label{thm-existence}
Let $\eps\in (0,1]$. There exist a filtered probability space $(\Omega, \mathcal F, (\mathcal F_t),\P)$, a family of complex $(\mathcal F_t)$-Brownian motions $\{W^k \}_{k\in \Z^2_0}$ and an $(\mathcal F_t)$-progressively measurable stationary process $\xi:\Omega\to C\big( [0,T], H^{-1-}\big)$ such that, for all $t\in [0,T]$, $\xi_t$ is a white noise on $\T^2$; and for all $\phi\in C^\infty(\T^2)$, $\P$-a.s. for all  $t\in [0,T]$, one has
  $$\aligned
  \big\<\xi_t,\phi \big\> =&\ \big\<\xi_t,\phi \big\> + \int_0^t \big\<\xi_s\otimes \xi_s, H^\phi_\eps \big\>\,\d s + \frac12 \|\theta \|_{\ell^2}^2 \int_0^t \big\<\xi_s,\Delta \phi \big\> \,\d s \\
  &\, + \sum_{k\in \Z^2_0} \theta_k\int_0^t \big\<\xi_s,\sigma_k\cdot \nabla\phi \big\>\,\d W^k_s.
  \endaligned $$
\end{theorem}

In fact, as in \cite{FlaLuo-1}, one can prove more general results: given a density function $\rho_0\in C_b \big(H^{-1-},\R_+\big)$ (i.e. $\int_{H^{-1-}} \rho_0 \,\d\mu =1$), there exists a stochastic process $\xi:\Omega\to C\big( [0,T], H^{-1-}\big)$ and a time-dependent density function $\rho\in L^\infty\big( [0,T]\times H^{-1-}, \R_+\big)$, such that for all $t\in [0,T]$, $\xi_t$ has the law $\rho_t\,\d\mu$ and the process $\xi_t$ solves the equation \eqref{stoch-mSQG} in the same sense as above. We do not investigate such generalizations in this work. For the moment, we are unable to deal with more general initial densities via the method of point vortex approximation.

Next, following \cite{FlaLuo-2}, we introduce a suitable scaling of the noise in the stochastic mSQG equations \eqref{stoch-mSQG}. For this purpose, we take a special sequence $\{\theta^N\}_{N\geq 1} \subset \ell^2$:
  $$\theta^N_k = \frac1{|k|^\gamma} {\bf 1}_{\{|k|\leq N\}}, \quad k\in \Z^2_0, $$
where $\gamma\in (1/2,1]$ is a fixed parameter. It is clear that $\|\theta^N \|_{\ell^2} \to \infty$ as $N\to \infty$. For every $N\geq 1$, we consider the equations:
  \begin{equation}\label{stoch-mSQG-N}
  \left\{ \aligned
  & \d\xi^N + u^N\cdot\nabla \xi^N \,\d t+ \frac{\sqrt 2}{\|\theta^N \|_{\ell^2}} \sum_{k\in \Z^2_0} \theta^N_k \sigma_k\cdot\nabla \xi^N \circ\d W^k_t =0, \\
  & u^N= \nabla^\perp (-\Delta)^{-(1+\eps)/2} \xi^N .
  \endaligned \right.
  \end{equation}
Note that we have performed a scaling of the noise, thus, in the It\^o formulation, the coefficient in front of the Laplacian operator is simply  1; namely, we have (comparing with \eqref{stoch-mSQG-Ito}):
  \begin{equation}\label{stoch-mSQG-Ito-N}
  \d\xi^N + u^N\cdot\nabla \xi^N \,\d t= \Delta \xi^N\,\d t - \frac{\sqrt 2}{\|\theta^N \|_{\ell^2}} \sum_{k\in \Z^2_0} \theta^N_k \sigma_k\cdot\nabla \xi^N \,\d W^k_t.
  \end{equation}
By Theorem \ref{thm-existence}, for any $N\geq 1$, the above equation admits a stationary white noise solution $\xi^N$ with trajectories in $C\big( [0,T], H^{-1-}\big)$. The main result of this paper is

\begin{theorem}[Scaling limit] \label{thm-scaling-limit}
Suppose that $\eps\in (0,1]$. As $N\to \infty$, the sequence of white noise solutions $\xi^N$ to \eqref{stoch-mSQG-Ito-N} converge in law to the unique stationary solution of the stochastic dissipative mSQG equation:
  \begin{equation}\label{mSQG-space-time-white-noise}
  \left\{ \aligned
  &\d\xi + u\cdot\nabla \xi \,\d t= \Delta \xi\,\d t + \nabla^\perp \cdot \d \zeta_t, \\
  & u= \nabla^\perp (-\Delta)^{-(1+\eps)/2} \xi,
  \endaligned \right.
  \end{equation}
where $\zeta_t$ is a cylindrical Brownian motion on $H$.
\end{theorem}

Here the stationary solution to \eqref{mSQG-space-time-white-noise} is understood in the sense of Definition  \ref{def-cylinder-martingale} with initial law $\mu$ and the weak uniqueness of the stationary solution will be proved in Theorem \ref{thm-uniqueness} below (the case $\eps=1$ was studied in \cite{DaPD, AF}). It is possible to slightly extend the scaling limit to non-stationary case, but, due to the remark below Theorem \ref{thm-existence}, the initial density $\rho_0$ will be required to be a bounded continuous function on $H^{-1-}$. We omit such extension here.

Since the driven noise of \eqref{mSQG-space-time-white-noise} is given by the derivative of the space-time white noise, equation \eqref{mSQG-space-time-white-noise} is a singular SPDE and the nonlinear term is not well-defined in the classical sense. For $\eps=1$ equation \eqref{mSQG-space-time-white-noise} is the vorticity form of stochastic Navier-Stokes equation driven by space-time white noise studied in \cite{DaPD}. When $\eps$ becomes smaller, $u$ and the nonlinear term are more singular.  Recently, this kind of singular SPDEs have been studied a lot due to the new solution theories such as regularity structures proposed by Hairer in \cite{Hai14} and paracontrolled distributions introduced by Gubinelli, Imkeller and Perkowski in \cite{GIP15}.

Equation \eqref{mSQG-space-time-white-noise}  might be treated  by using the above two theories. According to \cite{Hai14}, for $\eps\in (0,1)$, equation \eqref{mSQG-space-time-white-noise} is in the subcritical regime while for $\eps=0$ it corresponds to the critical regime. Although there are powerful black box theories in regularity structure to obtain local well-posedness of singular SPDEs (see \cite{BHZ19, CH16}), it is  unclear whether the solution obtained by these theories is the same as the limiting  solution obtained in Theorem \ref{thm-scaling-limit}. (See \cite{RZZ17} for the proof of the dynamical $\Phi^4_2$ model and for the problem of the dynamical $\Phi^4_3$ model.)

On the other hand, equation \eqref{mSQG-space-time-white-noise} has Gaussian measure as an invariant measure. This suggests us to use a powerful probabilistic approach recently developed by Gubinelli and Perkowski  in \cite{GP-18}, by constructing a domain for the corresponding infinitesimal generators and  solving the Kolmogorov equation in this domain. We follow the method developed in \cite{GP-18} and prove the existence and uniqueness of cylinder martingale solutions to equation \eqref{mSQG-space-time-white-noise} as follows. This also helps us to identify the limit in Theorem \ref{thm-scaling-limit}. Let $\L$ be the generator corresponding to \eqref{mSQG-space-time-white-noise}, see \eqref{generator} for its definition.

\begin{theorem}\label{thm-uniqueness}
Assume $\eps\in (0,1]$. Let $\eta\in L^2(H^{-1-}, \mu)$ be a density function. There exists a unique solution $\xi$ to the cylinder martingale problem for $\L$ in the sense of Definition \ref{def-cylinder-martingale}, with initial distribution $\xi_0\sim \eta\,\d\mu$. Moreover, $\xi$ is a homogeneous Markov process with transition kernel $(T_t)_{t\geq 0}$ and invariant measure $\mu$.
\end{theorem}

This paper is organized as follows. In the next section, we follow the arguments in \cite[Sections 2 and 3]{FlaLuo-1} to prove the existence of white noise solutions to the stochastic mSQG equation \eqref{stoch-mSQG}. Then we prove the main result (Theorem \ref{thm-scaling-limit}) in Section 3. Finally,  we present in Section 4 a proof of the uniqueness of the cylinder martingale solutions (Theorem \ref{thm-uniqueness}) to \eqref{mSQG-space-time-white-noise} for $\eps\in (0,1]$. Section 5 (the appendix) is devoted to the proofs of some technical lemmas used in Section 4.

\section{Existence of white noise solutions to \eqref{stoch-mSQG}} \label{sec-white-noise-solution}

In this section we prove that the stochastic mSQG equation \eqref{stoch-mSQG} has a white noise solution, by using the method of point vortex approximation. The main ideas are similar to those in \cite[Sections 2 and 3]{FlaLuo-1}, thus we will only give some sketched proofs. We fix $\eps\in (0,1]$.

\subsection{Preliminaries} \label{subsec-prelim}

First of all, we introduce some notations. Let $\Z^2_+$ and $ \Z^2_-$ be a partition of $\Z^2_0$ such that $\Z^2_+ \cap \Z^2_-= \emptyset$ and $\Z^2_-= -\Z^2_+$. Let
  $$a_k= \begin{cases}
  \frac{k^\perp}{|k|}, & k\in \Z^2_+, \\
  -\frac{k^\perp}{|k|},& k\in \Z^2_-;
  \end{cases}  $$
then $a_k=a_{-k}$ for all $k\in \Z^2_0$. We define the CONS $\{\sigma_k \}_{k\in \Z^2_0}$ of the space $H:=\{f\in L^2(\T^2;\R^2);\div f=0, \int f\,\d x=0\}$ as follows:
  $$\sigma_k(x)= a_k \,e_k(x) = a_k\, e^{2\pi {\rm i} k\cdot x} ,\quad x\in \T^2, k\in \Z^2_0. $$
Note that
  \begin{equation}\label{basis}
  a_k\cdot k^\perp = |k| \big({\bf 1}_{\Z^2_+}(k) - {\bf 1}_{\Z^2_-}(k)\big).
  \end{equation}
Recall that $\theta\in \ell^2$ is symmetric, i.e. it satisfies \eqref{theta-symmetry}; then we have the following useful identity (see for instance \cite[Lemma 2.6]{FlaLuo-2}):
  \begin{equation}\label{key-identity}
  \sum_{k\in \Z^2_0} \theta_k^2\, \sigma_k \otimes \sigma_{-k}= \frac12 \|\theta \|_{\ell^2}^2 I_2,
  \end{equation}
where $I_2$ is the $2\times 2$ unit matrix.

Next, the family $\{W^k_t\}_{k\in \Z^2_0}$ consists of independent complex Brownian motions on a stochastic basis $(\Omega, \mathcal F, \P)$, satisfying
  \begin{equation}\label{BMs}
  \overline{W^k_t} = W^{-k}_t, \quad \big[W^k, W^l\big]_t = 2\, t \delta_{k,-l},  \quad k,l \in \Z^2_0,
  \end{equation}
where $[\cdot, \cdot]$ is the joint quadratic covariance operator. Now the cylindrical Brownian motion has the expansion:
  $$W_t(x)= \sum_{k\in  \Z^2_0} \sigma_k(x) W^k_t.$$

\subsection{Stochastic point vortex systems}

Throughout this section we fix a $\theta\in \ell^2$ fulfilling the symmetry condition \eqref{theta-symmetry}. To study the point vortex dynamics corresponding to the stochastic mSQG equation \eqref{stoch-mSQG}, we need some more notations. For $N\in \N$, define the generalized diagonal of $(\T^2)^N$ as
  $$D_N= \big\{(x_1,\cdots, x_N)\in (\T^2)^N: \exists\, i\neq j \mbox{ such that } x_i=x_j \big\} $$
and the complement $D_N^c:= (\T^2)^N \setminus D_N$. Let $(\xi_1,\cdots, \xi_N)\in (\R\setminus \{0\})^N$ be vortex intensities and $(x_1,\cdots, x_N)\in D_N^c$ be the initial positions of vortices; we consider the stochastic system of $N$-point vortices:
  \begin{equation}\label{stoch-point-vortices}
  \left\{ \aligned
  \d X^{i,N}_t &= \frac1{\sqrt N} \sum_{j=1,j\neq i}^N \xi_j K_\eps \big( X^{i,N}_t- X^{j,N}_t\big)\,\d t + \sum_{|k|\leq N} \theta_k\, \sigma_k\big( X^{i,N}_t \big) \circ \d W^k_t,\\
  X^{i,N}_0&= x_i, \quad 1\leq i\leq N.
  \endaligned \right.
  \end{equation}
It is known that (see \cite[Proposition 3.4]{LS}), $\P$-a.s., the above system admits a unique strong solution $\big(X^{1,N}_t, \cdots, X^{N,N}_t \big)\in D_N^c$ for all $t>0$. Moreover, since the fields $K_\eps$ and $\sigma_k$ are divergence free,  one can prove

\begin{lemma}\label{2-lem-1}
If the initial positions $(x_1,\cdots, x_N)$ are distributed as the Lebesgue measure ${\rm Leb}_{\T^2}^{\otimes N}$ on $(\T^2)^N$ and are independent of the Brownian motions $\{W^k_t\}_{k\in \Z^2_0}$, then the solution $\big(X^{1,N}_t, \cdots, X^{N,N}_t \big)$ is a stationary process with the invariant law ${\rm Leb}_{\T^2}^{\otimes N}$.
\end{lemma}

\begin{proof}
The proof is similar to that of \cite[Theorem 2.1]{FlaLuo-1}; we omit it here.
\end{proof}

Next, on a probability space $(\Omega, \mathcal F, \P)$, we take an i.i.d. sequence $\{\xi_i \}_{i\geq 1}$ of standard  Gaussian random variables, and  an i.i.d. sequence $\{X^i_0\}_{i\geq 1}$ of uniformly distributed random variables on $\T^2$; the two sequences are independent, and they are also independent of the Brownian motions $\{W^k_t\}_{k\in \Z^2_0}$. For any $N\in \N$, we define the random point vortices
  \begin{equation}\label{random-point-vortices}
  \xi^N_0= \frac1{\sqrt N} \sum_{i=1}^N \xi_i \delta_{X^i_0},
  \end{equation}
where $\delta_x,\, x\in \T^2,$ is the delta Dirac mass. It was shown in \cite[Proposition 21]{Flandoli18} that, as $N\to\infty$, $\xi^N_0$ converges weakly to the white noise on $\T^2$. In other words, if we denote by $\mu^N$ the law of $\xi^N_0$ on $H^{-1-}$, then $\mu^N$ converges weakly to the white noise measure $\mu$ on the torus $\T^2$.

For any $N\geq 1$, consider the stochastic point vortex system \eqref{stoch-point-vortices} with random vortex intensities $(\xi_1,\cdots, \xi_N)$ and random initial positions $(X^1_0, \cdots , X^N_0)$; we have the following result which is similar to \cite[Proposition 2.3]{FlaLuo-1}.

\begin{proposition}\label{2-prop-1}
For $\P$-a.s. value of $\big((\xi_1, X^1_0), \cdots, (\xi_N, X^N_0)\big)$, the system of equations \eqref{stoch-point-vortices} admits a unique solution $\big(X^{1,N}_t, \cdots, X^{N,N}_t \big)$ taking values in $D_N^c$ for all $t>0$, and the measure-valued process
  $$\xi^N_t= \frac1{\sqrt N} \sum_{i=1}^N \xi_i \delta_{X^{i,N}_t}$$
satisfies the following equation: for any $\phi\in C^\infty(\T^2)$, $\P$-a.s. for all $t>0$, it holds
  \begin{equation}\label{2-prop-1.1}
  \aligned
  \big\<\xi^N_t,\phi \big\>=&\ \big\<\xi^N_0,\phi \big\> + \int_0^t \big\<\xi^N_s\otimes \xi^N_s,H_\eps^\phi \big\>\,\d s + c_N \int_0^t \big\<\xi^N_s,\Delta \phi \big\>\,\d s\\
  &\, + \sum_{|k|\leq N} \theta_k\int_0^t \big\<\xi^N_s,\sigma_k\cdot \nabla\phi \big\>\,\d W^k_s,
  \endaligned
  \end{equation}
where $c_N=\frac12 \sum_{|k|\leq N} \theta_k^2$ and
  $$\big\<\xi^N_t \otimes \xi^N_t, H_\eps^\phi \big\>= \frac1{2N} \sum_{1\leq i \neq j\leq N} \xi_i\xi_j K_\eps \big( X^{i,N}_t- X^{j,N}_t\big) \cdot \big(\nabla\phi \big( X^{i,N}_t \big) - \nabla\phi \big( X^{j,N}_t \big) \big). $$
Moreover, the process $\xi^N_t$ is stationary with invariant law $\mu^N$.
\end{proposition}

\begin{proof}
Let $\phi\in C^\infty(\T^2)$; by the It\^o formula,
  $$\aligned
  \d\phi \big( X^{i,N}_t \big)=&\, \frac1{\sqrt N} \sum_{j=1,j\neq i}^N \xi_j K_\eps \big( X^{i,N}_t- X^{j,N}_t\big) \cdot \nabla\phi\big( X^{i,N}_t \big) \,\d t \\
  &+ \sum_{|k|\leq N} \theta_k \sigma_k\big( X^{i,N}_t \big)\cdot \nabla\phi\big( X^{i,N}_t \big) \circ \d W^k_t.
  \endaligned $$
Using the skew-symmetry of the kernel $K_\eps$ we can deduce
  $$\d\big\<\xi^N_t,\phi \big\> = \big\<\xi^N_t \otimes \xi^N_t, H_\eps^\phi \big\>\,\d t + \sum_{|k|\leq N} \theta_k \big\<\xi^N_t, \sigma_k\cdot \nabla\phi\big\> \circ \d W^k_t. $$
In the It\^o form, the above equation reads as
  $$\aligned
  \d\big\<\xi^N_t,\phi \big\> =&\ \big\<\xi^N_t \otimes \xi^N_t, H_\eps^\phi \big\>\,\d t + \sum_{|k|\leq N} \theta_k \big\<\xi^N_t, \sigma_k\cdot \nabla\phi\big\> \, \d W^k_t \\
  &+ \sum_{|k|\leq N} \theta_k^2 \big\<\xi^N_t, \sigma_{-k}\cdot \nabla(\sigma_k\cdot \nabla \phi)\big\> \,\d t,
  \endaligned  $$
where we have used \eqref{theta-symmetry} and \eqref{BMs}. By the fact that $\sigma_{-k}\cdot \nabla\sigma_k\equiv 0$ and \eqref{key-identity}, we have
  $$\sum_{|k|\leq N} \theta_k^2 \big[ \sigma_{-k}\cdot \nabla(\sigma_k\cdot \nabla \phi)\big] = c_N \Delta \phi.$$
This leads to the first assertion. The second one can be deduced from Lemma \ref{2-lem-1}, cf. \cite[Proposition 22]{Flandoli18}.
\end{proof}

We provide the following estimates on the point vortices $\xi^N_t$. Compared to \cite[Theorem 16]{FS} where only the second moment of $\big\<\xi^N_t\otimes \xi^N_t,H_\eps^\phi \big\>$ is considered, in the current paper, we need higher moment estimates since we are dealing with stochastic processes which has only H\"older continuous trajectories; see Section \ref{subsec-limit} where we will make use of time-fractional Sobolev spaces. In the last assertion we need the condition $\eps>0$.

\begin{proposition}\label{2-prop-estimates}
\begin{itemize}
\item[\rm (i)] For all $p>1$, there exists $C_p>0$ such that for any $t>0$ and $f\in L^\infty(\T^2)$,
  $$\E\big| \big\<\xi^N_t,f \big\> \big|^p \leq C_p \|f\|_{L^\infty(\T^2)}^p. $$
\item[\rm (ii)] For all $p>1$ and $\delta>0$, there is a $C_{p,\delta}>0$ independent of $N\geq 1$ and $t\in [0,T]$, such that
  $$\E \Big[\big\| \xi^N_t\big\|_{H^{-1-\delta}(\T^2)}^p \Big] \leq C_{p,\delta}. $$
\item[\rm (iii)] For all $p\in [1, 1/(1-\eps))$, there exists a constant $C>0$ such that for all $N\geq 1$, $t>0$ and $\phi\in C^\infty(\T^2)$,
  $$\E\Big[\big\<\xi^N_t\otimes \xi^N_t,H_\eps^\phi \big\>^{2p}\Big] \leq C\big\|H_\eps^\phi \big\|_{L^{2p}(\T^2\times \T^2)}^{2p}.   $$
\end{itemize}
\end{proposition}

\begin{proof}
The proofs of the first two estimates can be found in \cite[Lemma 23]{Flandoli18}; we only prove the last one, which is different from the one proved in \cite[Lemma 23]{Flandoli18} because the function $H_\eps^\phi$ is unbounded. First, by \eqref{H-eps-phi-estimate}, we see that $H_\eps^\phi\in L^{2p}(\T^2\times \T^2)$ for $p<1/(1-\eps)$. Next, since the process $\xi^N_t$ is stationary, we have
  $$\E\Big[\big\<\xi^N_t\otimes \xi^N_t,H_\eps^\phi \big\>^{2p}\Big]= \E\Big[\big\<\xi^N_0\otimes \xi^N_0,H_\eps^\phi \big\>^{2p} \Big];$$
thus it is enough to estimate the right-hand side. Note that
  $$  \big\<\xi^N_0\otimes \xi^N_0,H_\eps^\phi \big\>= \frac1N \sum_{i,j=1}^N \xi_i\xi_j H_\eps^\phi\big(X^{i}_0, X^{j}_0 \big);$$
one has
  $$\aligned
  \E\Big[\big\<\xi^N_0\otimes \xi^N_0,H_\eps^\phi \big\>^{2p} \Big] &= \E\Bigg[ \bigg(\frac1N \sum_{i,j=1}^N \xi_i\xi_j H_\eps^\phi\big(X^{i}_0, X^{j}_0 \big) \bigg)^{2p} \Bigg]\\
  &= \E\Bigg\{ \E\Bigg[ \bigg(\frac1N \sum_{i,j=1}^N \xi_i\xi_j H_\eps^\phi\big(X^{i}_0, X^{j}_0 \big) \bigg)^{2p} \bigg| \mathcal G \Bigg] \Bigg\},
  \endaligned $$
where $\mathcal G$ is the $\sigma$-algebra generated by $\{X^i_0 \}_{i\geq 1}$. Since the two families $\{\xi_i \}_{i\geq 1}$ and $\{X^i_0 \}_{i\geq 1}$ are independent, it holds that
  $$\E\Bigg[ \bigg(\frac1N \sum_{i,j=1}^N \xi_i\xi_j H_\eps^\phi\big(X^{i}_0, X^{j}_0 \big) \bigg)^{2p} \bigg| \mathcal G \Bigg]= \Bigg\{ \E\Bigg[\bigg(\sum_{i,j=1}^N a_{i,j} \xi_i\xi_j \bigg)^{2p} \Bigg] \Bigg\}_{a_{i,j}= \frac1N H_\eps^\phi (X^{i}_0, X^{j}_0 )}\, . $$
We remark that the coefficients $\{a_{i,j} \}_{1\leq i,j\leq N}$ satisfy $a_{i,j}=a_{j,i}$ and $a_{i,i}=0$; the latter is due to \eqref{H-eps-phi}. Fix such a family $\{a_{i,j} \}_{1\leq i,j\leq N}$, the random variable $\sum_{i,j=1}^N a_{i,j} \xi_i\xi_j$ belongs to the second Wiener chaos corresponding to the family $\{\xi_i \}_{i\geq 1}$ of  i.i.d. standard Gaussian random variables, thus
  $$\E\Bigg[\bigg(\sum_{i,j=1}^N a_{i,j} \xi_i\xi_j \bigg)^{2p} \Bigg] \leq C_p^{2p} \Bigg\{\E\Bigg[\bigg(\sum_{i,j=1}^N a_{i,j} \xi_i\xi_j \bigg)^2 \Bigg] \Bigg\}^p = C_p^{2p} \Bigg[ 2 \sum_{i,j=1}^N a_{i,j}^2 \Bigg]^p, $$
where the last step follows easily from the Isserlis-Wick theorem. Substituting this estimate into the above two equalities leads to
  $$\E\Big[\big\<\xi^N_0\otimes \xi^N_0,H_\eps^\phi \big\>^{2p} \Big] \leq C'_p\, \E\Bigg\{ \Bigg[\sum_{i,j=1}^N \frac1{N^2} H_\eps^\phi \big(X^{i}_0, X^{j}_0 \big)^2 \Bigg]^p \Bigg\} \leq C'_p\, \E\Bigg[ \sum_{i,j=1}^N \frac1{N^2} H_\eps^\phi \big(X^{i}_0, X^{j}_0 \big)^{2p} \Bigg], $$
where we have used Jensen's inequality. Recall that $\{X^i_0 \}_{i\geq 1}$ is an i.i.d. family of $\T^2$-valued uniformly distributed random variables; we have
  $$\E\Big[\big\<\xi^N_0\otimes \xi^N_0,H_\eps^\phi \big\>^{2p} \Big] \leq C'_p\int_{\T^2\times \T^2} H_\eps^\phi(x, y)^{2p} \,\d x\d y = C'_p \big\|H_\eps^\phi \big\|_{L^{2p}(\T^2\times \T^2)}^{2p} .$$
The proof is complete.
\end{proof}

For later purpose, we give the following remark.

\begin{remark}\label{rem-time-reversal}
Fix $T>0$; the time-reversal $\hat\xi^N_t:= \xi^N_{T-t}$ of the process $\xi^N_t$ enjoys similar properties as above. Indeed, if we define the reversed Brownian motions $\hat W^k_t= W^k_T- W^k_{T-t},\, 0\leq t\leq T,\, k\in \Z^2_0$ and consider the system
  \begin{equation*}
  \left\{ \aligned
  \d \hat X^{i,N}_t &= -\frac1{\sqrt N} \sum_{j=1,j\neq i}^N \xi_j K_\eps \big(\hat X^{i,N}_t- \hat X^{j,N}_t\big)\,\d t + \sum_{|k|\leq N} \theta_k\, \sigma_k\big(\hat X^{i,N}_t \big) \circ \d \hat W^k_t,\\
  \hat X^{i,N}_0&= X^{i,N}_T, \quad 1\leq i\leq N,
  \endaligned \right.
  \end{equation*}
  with $\xi_j, \theta_k$ as in \eqref{stoch-point-vortices},
then we have, $\P$-a.s., $\hat X^{i,N}_t= X^{i,N}_{T-t},\, 0\leq t\leq T,\, 1\leq i\leq N$ and
  $$\hat\xi^N_t = \frac1{\sqrt N} \sum_{i=1}^N \xi_i \hat X^{i,N}_t;$$
moreover, for any $\phi\in C^\infty(\T^2)$, repeating the proof of Proposition \ref{2-prop-1} yields the equation:
  $$  \aligned
  \big\<\hat\xi^N_t,\phi \big\>=&\ \big\<\hat \xi^N_0,\phi \big\> - \int_0^t \big\<\hat \xi^N_s\otimes \hat \xi^N_s,H_\eps^\phi \big\>\,\d s + c_N \int_0^t \big\<\hat \xi^N_s,\Delta \phi \big\>\,\d s\\
  &\, + \sum_{|k|\leq N} \theta_k\int_0^t \big\<\hat \xi^N_s,\sigma_k\cdot \nabla\phi \big\>\,\d \hat W^k_s.
  \endaligned $$
\end{remark}

\subsection{Definition of the nonlinear term for a white noise}\label{subsec-definition}

As mentioned above, the random point vortices \eqref{random-point-vortices} converges weakly to the white noise on $\T^2$. Before proceeding further, we have to give a definition of the nonlinear term in \eqref{2-prop-1.1} when $\xi^N_s$ is replaced by a white noise $\omega_{WN}$ on $\T^2$. Indeed, due to the singularity of the kernel $K_\eps $ ($\eps\in (0,1)$, see \eqref{K-eps}), the function $H_\eps^\phi$ is more singular than $H_\phi= H_1^\phi$ considered in the papers \cite{Flandoli18, FlaLuo-1}. Fortunately, $H_\eps^\phi(x,y)$ is still square integrable on $\T^2\times \T^2$ with respect to the Lebesgue measure; hence, for a white noise $\omega_{WN}$ on $\T^2$, one can follow the ideas in \cite[Section 2.4]{Flandoli18} to define the nonlinear term $\big\<\omega_{WN}\otimes \omega_{WN}, H_\eps^\phi\big\>$, see e.g. \cite[Section 2.2]{FS}. Here we recall the main steps.

Let $\omega_{WN}$ be a white noise on $\T^2$, defined on some probability space $(\Omega, \mathcal F, \P)$. By definition, $\omega_{WN}$ is a centered Gaussian random variable taking values in the space of distributions $C^\infty(\T^2)'$ such that for any $\phi,\psi\in C^\infty(\T^2)$, one has
  $$\E[\<\omega_{WN}, \phi\>\<\omega_{WN}, \psi\>] = \<\phi, \psi\>.$$
Here, $\<\cdot, \cdot\>$ denotes the duality between distributions and smooth functions, or the inner product in $L^2(\T^2)$ when both objects are functions. Using the Fourier basis on $\T^2$, it is not difficult to show that the law $\mu$ of $\omega$ is supported by $H^{-1-}(\T^2)$. The results below are proved in \cite[Corollary 6]{Flandoli18}; we omit the proofs here.

\begin{lemma}\label{2-lem-2}
Let $\omega_{WN}: \Omega \to H^{-1-}(\T^2)$ be a white noise and $f\in H^{2+}(\T^2\times \T^2)$.
\begin{itemize}
\item[\rm(i)] For every $p\geq 1$ there is constant $C_p>0$ such that
  $$\E[|\<\omega_{WN}\otimes \omega_{WN}, f\>|^p] \leq C_p \|f\|_{L^\infty(\T^2)}^p.$$

\item[\rm(ii)] We have $\E \<\omega_{WN}\otimes \omega_{WN}, f\> = \int_{\T^2} f(x,x)\,\d x$.

\item[\rm(iii)] If $f$ is symmetric, then
  $$\E\big[ |\<\omega_{WN}\otimes \omega_{WN}, f\> - \E\<\omega_{WN}\otimes \omega_{WN}, f\>|^2\big]= 2 \int_{\T^2\times \T^2} f(x,y)^2\,\d x\d y.$$
\end{itemize}
\end{lemma}

Based on these facts we can give a definition of $\big\<\omega_{WN}\otimes \omega_{WN}, H_\eps^\phi \big\>$ when $\omega_{WN}$ is a white noise on $\T^2$.

\begin{proposition}\label{2-prop-2}
Let $\omega_{WN}: \Omega \to H^{-1-}(\T^2)$ be a white noise. Assume that $H_n\in H^{2+}(\T^2\times \T^2)$ are symmetric and approximate $H_\eps^\phi$ in the following sense:
  $$\lim_{n\to\infty} \int_{\T^2\times \T^2} \big(H_n-H_\eps^\phi \big)^2(x,y)\,\d x\d y =0, \quad \lim_{n\to\infty} \int_{\T^2} H_n(x,x)\,\d x =0.$$
Then the sequence of random variables $\<\omega_{WN}\otimes \omega_{WN}, H_n\>$ is a Cauchy sequence in mean square. We denote its limit by $\big\<\omega_{WN}\otimes \omega_{WN}, H_\eps^\phi \big\>$.

Moreover, the limit is the same if $H_n$ is replaced by $\tilde H_n$ with the same properties and such that $\lim_{n\to\infty} \int_{\T^2\times \T^2} (H_n- \tilde H_n)^2(x,y)\,\d x\d y =0$.
\end{proposition}

\begin{proof}
The proofs are the same as those of \cite[Theorem 8]{Flandoli18}; we recall them here for completeness. Since $\lim_{n\to\infty} \int_{\T^2} H_n(x,x)\,\d x =0$, it is equivalent to show that $\<\omega_{WN}\otimes \omega_{WN}, H_n\>- \int_{\T^2} H_n(x,x)\,\d x$ is a Cauchy sequence in mean square. We have
  $$\aligned &\ \E\bigg[\Big|\<\omega_{WN}\otimes \omega_{WN}, H_n\>- \int_{\T^2} H_n(x,x)\,\d x- \<\omega_{WN}\otimes \omega_{WN}, H_m\>+ \int_{\T^2} H_m(x,x)\,\d x\Big|^2 \bigg]\\
  =&\ \E\bigg[\Big|\<\omega_{WN}\otimes \omega_{WN}, H_n-H_m\>- \int_{\T^2} (H_n-H_m)(x,x)\,\d x \Big|^2 \bigg]\\
  =&\ 2 \int_{\T^2\times \T^2} (H_n-H_m)^2(x,y)\,\d x\d y,
  \endaligned$$
where the last equality follows from (ii) and (iii) of Lemma \ref{2-lem-2}. This implies the Cauchy property, and thus $\big\<\omega_{WN}\otimes \omega_{WN}, H_\eps^\phi \big\>$ is well defined. The invariance property is proved in the same way.
\end{proof}

Here is an example of the approximating functions $H_n$. Let $\chi: \T^2=[-1/2, 1/2]^2 \to [0,1]$ be a smooth and symmetric function with support in a small ball $B(0,r)$, $r<1$, and equal to 1 in $B(0,r/2)$. For any $n\geq 1$, set $\chi_n(x) = \chi(nx),\, x\in \T^2$. Define
  $$H_n(x,y) = \begin{cases}
  H_\eps^\phi(x,y)(1-\chi_n(x-y)), & x\neq y;\\
  0, & x=y.
  \end{cases}$$
Since $H_n(x,x) \equiv 0$, we have the following estimate (cf. Lemma \ref{2-lem-2}(iii)):
  \begin{equation}\label{eq-1}
  \E \Big[\big(\<\omega_{WN} \otimes \omega_{WN}, H_n\>- \big\<\omega_{WN} \otimes \omega_{WN}, H_\eps^\phi \big\> \big)^2 \Big] \leq 2 \int_{\T^2\times \T^2} \big(H_n-H_\eps^\phi \big)^2(x,y)\,\d x\d y.
  \end{equation}
We remark that the random variables $\<\omega_{WN} \otimes \omega_{WN}, H_n\>$ belong to the second Wiener chaos defined in terms of $\omega_{WN}$; thus, as an $L^2(\Omega, \P)$-limit, $\big\<\omega_{WN} \otimes \omega_{WN}, H_\eps^\phi \big\>$ also belongs to the second Wiener chaos. As a result, for any $p\geq 2$,
  \begin{equation}\label{moment-estimate}
  \E \Big[ \big| \big\<\omega_{WN} \otimes \omega_{WN}, H_\eps^\phi \big\> \big|^p \Big] \leq C_p \big\| H_\eps^\phi \big\|_{L^2(\T^2\times \T^2)}^p.
  \end{equation}

The above results hold as well for a stochastic process of white noises, see \cite[Theorem 10]{Flandoli18} for the proof.

\begin{corollary}
Let $\omega:\Omega \to C\big([0,T], H^{-1-}\big)$ be such that $\omega_t$ is a white noise on $\T^2$ for all $t\in [0,T]$. For any $\phi\in C^\infty(\T^2)$, let $H_n$ be an approximation of $H_\eps^\phi$ as in Proposition \ref{2-prop-2}; then $\<\omega_\cdot \otimes \omega_\cdot, H_n \>$ is a Cauchy sequence in $L^2\big(\Omega, L^2(0,T)\big)$ and we denote the limit by $\big\<\omega_\cdot \otimes \omega_\cdot, H_\eps^\phi \big\>$.
\end{corollary}

\subsection{White noise solutions to \eqref{stoch-mSQG}} \label{subsec-limit}

Recall the measure-valued process $\xi^N_t$ defined in Proposition \ref{2-prop-1}. In this part we prove that the sequence of processes $\big\{\xi^N_t \big\}_{N\geq 1}$ converges weakly to a limit process $\xi_t$, which will be a white noise solution to the equation \eqref{stoch-mSQG}. The arguments are by now classical; we will proceed by following the ideas in \cite[Section 4.2]{Flandoli18} or \cite[Section 3]{FlaLuo-1}.

Let $Q^N$ be the law of $\xi^N_\cdot$; we want to prove that the family of laws $\big\{Q^N \big\}_{N\geq 1}$ is tight on $C\big([0,T], H^{-1-} \big)$. It is sufficient to show that, for any fixed $\delta>0$, $\big\{Q^N \big\}_{N\geq 1}$ is tight on $C\big([0,T], H^{-1-\delta} \big)$. To this end, we need the compactness result of Simon \cite[p. 90, Corollary 9]{Simon}, which involves the fractional Sobolev space. Given $\alpha\in (0,1)$, $p>1$ and a normed linear space $(Y,\|\cdot \|_Y)$, the space $W^{\alpha,p}(0,T; Y)$ is defined as those functions $f\in L^p(0,T; Y)$ such that
  $$\int_0^T\!\! \int_0^T \frac{\|f(t)- f(s) \|_Y^p}{|t-s|^{1+\alpha p}} \,\d t\d s <+\infty. $$
Fix $\delta>0$ small enough and $\kappa>5$ (this choice is due to the computations in Corollary \ref{cor-tightness} below); we have the compact inclusions
  $$H^{-1-\delta/2} \subset H^{-1-\delta} \subset H^{-\kappa}. $$
It is well known that there exists a constant $C>0$ such that
  $$\|\xi \|_{H^{-1-\delta}}\leq C\|\xi \|_{H^{-1-\delta/2}}^{1-\theta} \|\xi \|_{H^{-\kappa}}^\theta,\quad \xi\in H^{-1-\delta/2},$$
where $\theta= \delta/(2\kappa-2 -\delta) \in (0,1)$. In the sequel, for a fixed $p\in (1,1/(1-\eps))$ and $p\leq 2$, we shall choose $\alpha\in (0,1/2)$ such that $\alpha> 1/(2p)$. Obviously, such parameter $\alpha$ exists. Now by \cite[p. 90, Corollary 9]{Simon}, we have

\begin{lemma}\label{lem-simon}
For $q$ sufficiently big, the space
  $$L^q\big(0,T; H^{-1-\delta/2}\big) \cap W^{\alpha,2p}\big(0,T; H^{-\kappa}\big) $$
is compactly embedded into $C\big([0,T], H^{-1-\delta} \big)$.
\end{lemma}

Recall that $Q^N$ is the law of the stationary process $\big\{\xi^N_t \big\}_{t\in [0,T]}$ of random point vortices obtained in Proposition \ref{2-prop-1}. As a consequence of Lemma \ref{lem-simon}, it is easy to prove the next result.

\begin{proposition}\label{2-prop-3}
If for any $q>1$, there exists a $C_q>0$ such that for all $N\geq 1$, it holds
  \begin{equation}\label{2-prop-3.1}
  \E \int_0^T \big\|\xi^N_t \big\|_{H^{-1-\delta/2}}^q \,\d t + \E \int_0^T\!\! \int_0^T \frac{\|\xi^N_t- \xi^N_s \|_{H^{-\kappa}}^{2p}}{|t-s|^{1+2\alpha p}}\,\d t\d s <C_q,
  \end{equation}
then the family $\big\{Q^N \big\}_{N\geq 1}$ is tight on $C\big([0,T], H^{-1-\delta} \big)$.
\end{proposition}

Thanks to Proposition \ref{2-prop-estimates}(ii), we immediately obtain the boundedness of the first expectation in \eqref{2-prop-3.1}; therefore, it remains to check that the second one is also finite. For this purpose, we first prove the following estimate; recall that $e_k(x)= e^{2\pi{\rm i} k\cdot x},\, x\in \T^2,\, k\in \Z^2_0$.

\begin{lemma}\label{lem-difference}
There is a constant $C>0$ such that for all $k\in \Z^2_0$ and $0\leq s<t\leq T$, one has
  $$\E\Big[ \big|\big\<\xi^N_t- \xi^N_s, e_k\big\> \big|^{2p} \Big]\leq C|k|^{4p} |t-s|^p. $$
\end{lemma}

\begin{proof}
Replacing $\phi$ by $e_k$ in \eqref{2-prop-1.1}, one easily get
  \begin{equation}\label{lem-difference.1}
  \aligned
  \E\Big[ \big| \big\<\xi^N_t- \xi^N_s, e_k\big\> \big|^{2p} \Big]\leq &\ C_p \E\bigg[ \Big| \int_s^t \big\<\xi^N_r\otimes \xi^N_r,H_\eps^{e_k} \big\>\,\d r \Big|^{2p} \bigg] + C_p c_N^{2p} \E\bigg[ \Big| \int_s^t \big\<\xi^N_r,\Delta e_k \big\>\,\d r \Big|^{2p} \bigg]\\
  &\, + C_p \E\Bigg[ \bigg| \sum_{|l|\leq N} \theta_l \int_s^t \big\<\xi^N_r,\sigma_l\cdot \nabla e_k \big\>\,\d W^l_r \bigg|^{2p} \Bigg],
  \endaligned
  \end{equation}
  where $c_N=\frac12 \sum_{|k|\leq N} \theta_k^2$ and is uniformly bounded in $N$.
We have, by Proposition \ref{2-prop-estimates}(iii),
  $$\aligned
  \E\bigg[ \Big| \int_s^t \big\<\xi^N_r\otimes \xi^N_r,H_\eps^{e_k} \big\>\,\d r \Big|^{2p} \bigg] &\leq |t-s|^{2p-1} \int_s^t \E\Big[\big\<\xi^N_r\otimes \xi^N_r,H_\eps^{e_k} \big\>^{2p} \Big]\,\d r\\
  & \leq C_{1,p}|t-s|^{2p} \big\|H_\eps^{e_k} \big\|_{L^{2p}(\T^2\times \T^2)}^{2p}.
  \endaligned $$
Recalling the estimate \eqref{H-eps-phi-estimate}, we have
  $$\big\|H_\eps^{e_k} \big\|_{L^{2p}(\T^2\times \T^2)} \leq C'_{1,p} \|\nabla^2 {e_k}\|_\infty \leq C''_{1,p} |k|^2. $$
Substituting this result into the inequality above yields
  \begin{equation}\label{lem-difference.2}
  \E\bigg[ \Big| \int_s^t \big\<\xi^N_r\otimes \xi^N_r,H_\eps^{e_k} \big\>\,\d r \Big|^{2p} \bigg] \leq C_{2,p} |k|^{4p} |t-s|^{2p}.
  \end{equation}

Next, as $\Delta e_k= -4\pi^2 |k|^2 e_k$ and $c_N= \frac12 \sum_{|l|\leq N} \theta_l^2 \leq \frac12 \|\theta \|_{\ell^2}^2$,
  \begin{equation}\label{lem-difference.3}
  c_N^{2p} \E\bigg[ \Big| \int_s^t \big\<\xi^N_r,\Delta e_k \big\>\,\d r \Big|^{2p} \bigg] \leq \frac{\|\theta \|_{\ell^2}^{4p}}{2^{2p}} |t-s|^{2p-1} \int_s^t \E \big| \big\<\xi^N_r,\Delta e_k \big\> \big|^{2p} \,\d r \leq C_{3,p}|k|^{4p} |t-s|^{2p},
  \end{equation}
where the second step follows from the estimate (i) in Proposition \ref{2-prop-estimates}. Finally, by the Burkholder-Davis-Gundy inequality and the Cauchy inequality,
  $$\aligned
  \E\Bigg[ \bigg| \sum_{|l|\leq N} \theta_l \int_s^t \big\<\xi^N_r,\sigma_l\cdot \nabla e_k \big\>\,\d W^l_r \bigg|^{2p} \Bigg] &\leq C_{4,p} \E\Bigg[ \bigg( \sum_{|l|\leq N} \theta_l^2 \int_s^t \big| \big\<\xi^N_r,\sigma_l\cdot \nabla e_k \big\> \big|^2\,\d r \bigg)^p \Bigg] \\
  &\leq C_{4,p} |t-s|^{p-1} \int_s^t \E \Bigg[ \bigg(\sum_{|l|\leq N} \theta_l^2 \big| \big\<\xi^N_r,\sigma_l\cdot \nabla e_k \big\> \big|^2 \bigg)^p \Bigg] \,\d r.
  \endaligned $$
Recall that $2c_N= \sum_{|l|\leq N} \theta_l^2$; by Jensen's inequality,
  $$\aligned
  \bigg(\sum_{|l|\leq N} \theta_l^2 \big| \big\<\xi^N_r,\sigma_l\cdot \nabla e_k \big\> \big|^2 \bigg)^p &= (2c_N)^p \bigg(\sum_{|l|\leq N} \frac{\theta_l^2}{2c_N} \big| \big\<\xi^N_r,\sigma_l\cdot \nabla e_k \big\> \big|^2 \bigg)^p \\
  &\leq (2c_N)^{p-1} \sum_{|l|\leq N} \theta_l^2 \big| \big\<\xi^N_r,\sigma_l\cdot \nabla e_k \big\> \big|^{2p} .
  \endaligned $$
Now, using Proposition \ref{2-prop-estimates}(i),
  $$\aligned
  \E \Bigg[ \bigg(\sum_{|l|\leq N} \theta_l^2 \big| \big\<\xi^N_r,\sigma_l\cdot \nabla e_k \big\> \big|^2 \bigg)^p \Bigg] &\leq  (2c_N)^{p-1} \sum_{|l|\leq N} \theta_l^2\, \E \Big( \big| \big\<\xi^N_r,\sigma_l\cdot \nabla e_k \big\> \big|^{2p} \Big) \\
  &\leq (2c_N)^{p-1} \sum_{|l|\leq N} \theta_l^2 C_{5,p} \|\sigma_l\cdot \nabla e_k \|_\infty^{2p} \leq C'_{5p} \|\theta \|_{\ell^2}^{2p} |k|^{2p}.
  \endaligned $$
Summarizing these arguments yields
  $$\E\Bigg[ \bigg| \sum_{|l|\leq N} \theta_l \int_s^t \big\<\xi^N_r,\sigma_l\cdot \nabla e_k \big\>\,\d W^l_r \bigg|^{2p} \Bigg] \leq \tilde C_{5,p} |k|^{2p}|t-s|^p. $$
Combining this estimate with \eqref{lem-difference.1}--\eqref{lem-difference.3}, we obtain the desired result.
\end{proof}

Now we can prove

\begin{corollary}\label{cor-tightness}
The family of laws $\big\{Q^N \big\}_{N\geq 1}$ is tight on $C\big([0,T], H^{-1-}\big)$.
\end{corollary}

\begin{proof}
As mentioned above, it is sufficient to check that the second expectation in \eqref{2-prop-3.1} is bounded in $N\geq 1$. Indeed, by the H\"older inequality and Lemma \ref{lem-difference},
  $$\aligned
  \E\Big( \big\|\xi^N_t- \xi^N_s \big\|_{H^{-\kappa}}^{2p} \Big) &= \E\Bigg[ \bigg(\sum_{k\in \Z^2_0} \frac1{|k|^{2\kappa}}\big| \big\<\xi^N_t- \xi^N_s, e_k\big\> \big|^2 \bigg)^p \Bigg] \\
  & \leq \bigg(\sum_{k\in \Z^2_0} \frac1{|k|^{2\kappa}} \bigg)^{p-1} \sum_{k\in \Z^2_0} \frac1{|k|^{2\kappa}} \E\Big[ \big| \big\<\xi^N_t- \xi^N_s, e_k\big\> \big|^{2p} \Big] \\
  &\leq C_{\kappa,p} \sum_{k\in \Z^2_0} \frac1{|k|^{2\kappa}} C|k|^{4p} |t-s|^p \leq C'_{\kappa,p} |t-s|^p
  \endaligned $$
since $p\leq 2$ and $\kappa>5$. From this estimate we obtain
  $$\E \int_0^T\!\! \int_0^T \frac{\|\xi^N_t- \xi^N_s \|_{H^{-\kappa}}^{2p}}{|t-s|^{1+2\alpha p}}\,\d t\d s \leq  \int_0^T\!\! \int_0^T \frac{C |t-s|^{p}}{|t-s|^{1+2\alpha p}}\,\d t\d s <\infty,$$
where the last step is due to $\alpha<1/2$.
\end{proof}

At this step, we can use the Prohorov theorem (see \cite[p. 59, Theorem 5.1]{Billingsley}) to conclude that there is a subsequence  $Q^{N_i}$ converging weakly to some probability $Q$ on $C\big([0,T], H^{-1-}\big)$. Next, the Skorohod representation theorem (see \cite[p. 70, Theorem 6.7]{Billingsley}) implies that there exist a sequence of stochastic processes $\big\{\tilde \xi^i_\cdot \big\}_{i\geq 1}$ and a limit process $\tilde\xi_\cdot$ defined on a new probability space $\big(\tilde \Omega, \tilde{\mathcal F}, \tilde\P \big)$, such that
\begin{itemize}
\item $\tilde\P$-a.s., $\tilde \xi^i_\cdot$ converges in the topology of $C\big([0,T], H^{-1-}\big)$ to $\tilde\xi_\cdot\, $;
\item $\tilde\xi_\cdot$ has the law $Q$ and $\tilde \xi^i_\cdot$ has the law $Q^{N_i}$ for all $i\geq 1$.
\end{itemize}
With these preparations, it is easy to show that the limit process $\tilde \xi_\cdot$ is stationary and, for any $t\in [0,T]$, $\tilde \xi_t$ is a white noise on $\T^2$; moreover, by  classical arguments, one can prove that it solves the equation \eqref{stoch-mSQG} in the following sense: for any $\phi\in C^\infty(\T^2)$, $\tilde\P$-a.s. for all $t\in [0,T]$, it holds
  \begin{equation}\label{eq-limit}
  \aligned
  \big\<\tilde \xi_t,\phi \big\> =&\ \big\<\tilde \xi_0,\phi \big\> + \int_0^t \big\<\tilde\xi_s\otimes \tilde\xi_s, H^\phi_\eps \big\>\,\d s + \frac12 \|\theta \|_{\ell^2}^2 \int_0^t \big\<\tilde \xi_s,\Delta \phi \big\> \,\d s \\
  &\, + \sum_{k\in \Z^2_0} \theta_k\int_0^t \big\<\tilde\xi^N_s,\sigma_k\cdot \nabla\phi \big\>\,\d \tilde W^k_s,
  \endaligned
  \end{equation}
where $\big\{\tilde W^k_\cdot \big\}_{k\in \Z^2_0}$ is a family of complex Brownian motions defined on $\big(\tilde \Omega, \tilde{\mathcal F}, \tilde\P \big)$. We omit the proofs here, see pp. 800--807 of \cite{FlaLuo-1} for details.

Finally, thanks to Remark \ref{rem-time-reversal}, we give the following

\begin{remark}\label{rem-time-reversal-1}
Given $T>0$; if we consider the approximating sequence $\big\{\xi^N_t \big\}_{N\geq 1}$ together with the reversed processes $\big\{\hat\xi^N_t= \xi^N_{T-t} \big\}_{N\geq 1}$, then the above arguments work as well and yield that the time-reversal $\hat\xi_t=\tilde \xi_{T-t}$ of the limit process $\tilde \xi_t$ solves the equation below:
  $$\aligned  \big\<\hat \xi_t,\phi \big\> =&\ \big\<\hat \xi_0,\phi \big\> - \int_0^t \big\<\hat\xi_s\otimes \hat\xi_s, H^\phi_\eps \big\>\,\d s + \frac12 \|\theta \|_{\ell^2}^2 \int_0^t \big\<\hat \xi_s,\Delta \phi \big\> \,\d s \\
  &\, + \sum_{k\in \Z^2_0} \theta_k\int_0^t \big\< \hat\xi^N_s,\sigma_k\cdot \nabla\phi \big\>\,\d \hat W^k_s,
  \endaligned $$
where $\hat W^k_s= \tilde W^k_T- \tilde W^k_{T-s},\, 0\leq s\leq T,\, k\in \Z^2_0$. Note that the sign in front of the nonlinear part is different from the one in \eqref{eq-limit}.
\end{remark}

\section{The scaling limit}

In this part we take the following special sequence $\theta^N \in \ell^2, N\geq 1$ such that
  $$\theta^N_k = \frac1{|k|^\gamma} {\bf 1}_{\{|k|\leq N\}}, \quad k\in \Z^2_0, $$
where $\gamma\in (1/2, 1]$. Consider the stochastic mSQG equations \eqref{stoch-mSQG-Ito-N} that we recall here:
  \begin{equation}\label{3-stoch-mSQG-Ito}
  \d\xi^N + u^N\cdot\nabla \xi^N \,\d t= \Delta \xi^N\,\d t - \frac{\sqrt 2}{\|\theta^N \|_{\ell^2}} \sum_{|k|\leq N} \frac1{|k|^\gamma} \sigma_k \cdot \nabla \xi^N \,\d W^k_t,
  \end{equation}
where $u^N= \nabla^\perp (-\Delta)^{-(1+\eps)/2} \xi^N$. Theorem \ref{thm-existence} implies that, for all $N\in \N$, the above equation has a stationary white noise solution $\xi^N$ with trajectories in $C\big([0,T], H^{-1-}\big)$; more precisely, for all $\phi\in C^\infty(\T^2)$, $\P$-a.s. for all $t\in [0,T]$, it holds
  \begin{equation}\label{3-stoch-mSQG-Ito.1}
  \aligned
  \big\<\xi^N_t,\phi \big\> =&\ \big\<\xi^N_0,\phi \big\> + \int_0^t \big\<\xi^N_s\otimes \xi^N_s, H^\phi_\eps \big\>\,\d s + \int_0^t \big\<\xi^N_s,\Delta \phi \big\> \,\d s \\
  &\, +\frac{\sqrt 2}{\|\theta^N \|_{\ell^2}} \sum_{|k|\leq N} \frac1{|k|^\gamma} \int_0^t \big\<\xi^N_s,\sigma_k\cdot \nabla\phi \big\>\,\d W^k_s.
  \endaligned
  \end{equation}
Note that the solutions (and Brownian motions) may be defined on different probability spaces; however, for simplicity, we do not distinguish the notations $\Omega,\,\P$ etc. Our purpose in this section is to prove Theorem \ref{thm-scaling-limit}  following the ideas in \cite{FlaLuo-2}.

First, let $Q^N$ be the law of $\xi^N_\cdot,\, N\geq 1$; we want to prove that the family $\{Q^N \}_{N\geq 1}$ is tight on $C\big([0,T], H^{-1-}\big)$. This can be done in the same way as Corollary \ref{cor-tightness}, once we have the following analogs of Proposition \ref{2-prop-estimates}.

\begin{lemma}\label{3-lem-estimates}
For any $p\geq 2$ and $\delta>0$, there are constants $C_p>0$ and $C_{p,\delta}>0$, independent of $N\geq 1$ and $t\in [0,T]$, such that
\begin{itemize}
\item[\rm (i)] for any $f\in L^\infty(\T^2)$, one has
  $$\E\big| \big\<\xi^N_t,f \big\> \big|^p \leq C_p \|f\|_{L^\infty(\T^2)}^p; $$
\item[\rm (ii)] it holds
  $$\E \Big[\big\| \xi^N_t\big\|_{H^{-1-\delta}(\T^2)}^p \Big] \leq C_{p,\delta}; $$
\item[\rm (iii)] for all $\phi\in C^\infty(\T^2)$,
  $$\E\Big[ \big|\big\<\xi^N_t\otimes \xi^N_t,H_\eps^\phi \big\> \big|^{p}\Big] \leq C_p\big\|H_\eps^\phi \big\|_{L^2(\T^2\times \T^2)}^{p}.   $$
\end{itemize}
\end{lemma}

Since the processes $\xi^N_\cdot,\, N\geq 1$ are stationary with white noise marginal distribution, the first two estimates are well known, while the third one follows readily from the moment estimate \eqref{moment-estimate}. We omit the details here. Note that in the last estimate we only need the $L^2(\T^2\times \T^2)$-norm of $H_\eps^\phi$, in contrast to Proposition \ref{2-prop-estimates}(iii).

Using the above estimates and the equation \eqref{3-stoch-mSQG-Ito.1}, we can proceed as in Section \ref{subsec-limit} to show the tightness of $\big\{Q^N \big\}_{N\geq 1}$. Again by the Prohorov theorem, we can find a subsequence $\big\{Q^{N_i} \big\}_{i\geq 1}$ which converges weakly to some probability measure $Q$ supported by $C\big([0,T], H^{-1-}\big)$. Next, the Skorohod representation theorem implies that there exist a filtered probability space $\big(\tilde \Omega, \tilde{\mathcal F}, \tilde\P\big)$ and a sequence of processes $\tilde\xi^i_\cdot,\, i\geq 1$ and the limit process $\tilde \xi_\cdot$, such that
\begin{itemize}
\item[(a)] $\tilde \xi_\cdot \stackrel{d}{\sim} Q$ and $\tilde \xi^i_\cdot \stackrel{d}{\sim} Q^{N_i}$ for all $i\geq 1$;
\item[(b)] $\tilde\P$-a.s., $\tilde\xi^i_\cdot$ converges in $C\big([0,T], H^{-1-}\big)$ to $\tilde \xi_\cdot$ as $i\to \infty$.
\end{itemize}
Moreover, for any $i\geq 1$, since the process $\tilde\xi^i_\cdot$ has the same law $Q^{N_i}$ as $\xi^{N_i}_\cdot$ and the latter is stationary, it is easy to show that the limit process $\tilde \xi_\cdot$ is also stationary.

Corresponding to the sequence of processes $\tilde\xi^i_\cdot,\, i\geq 1$, we can also prove the existences of complex Brownian motions $\big\{\tilde W^{i,k}_\cdot \big\}_{k\in \Z^2_0},\, i\geq 1$, such that, for all $i\geq 1$ the pair $\big(\tilde\xi^i_\cdot, \{\tilde W^{i,k}_\cdot \}_{k\in \Z^2_0}\big)$ has the same law as $\big(\xi^{N_i}_\cdot, \{W^{k}_\cdot \}_{k\in \Z^2_0}\big)$. As the original process  $\xi^{N_i}_\cdot$ solves the equation \eqref{3-stoch-mSQG-Ito.1} with $N_i$ in place of $N$, we conclude that $\tilde\xi^i_\cdot$ satisfies the following equation: for any $\phi\in C^\infty(\T^2)$, $\tilde\P$-a.s. for all $t\in [0,T]$,
  \begin{equation}\label{3-stoch-mSQG-Ito.2}
  \aligned
  \big\<\tilde \xi^i_t,\phi \big\> =&\ \big\<\tilde \xi^i_0,\phi \big\> + \int_0^t \big\<\tilde \xi^i_s\otimes \tilde \xi^i_s, H^\phi_\eps \big\>\,\d s + \int_0^t \big\<\tilde \xi^i_s,\Delta \phi \big\> \,\d s \\
  &\, + \frac{\sqrt 2}{\|\theta^{N_i} \|_{\ell^2}} \sum_{|k|\leq N_i} \frac1{|k|^\gamma} \int_0^t \big\<\tilde \xi^i_s,\sigma_k\cdot \nabla\phi \big\>\,\d \tilde W^{i,k}_s.
  \endaligned
  \end{equation}
We are unable to take limit directly in the equation above, since we cannot prove the convergence of the martingale part; instead, we shall follow the approach of \cite{FlaLuo-2} to show the convergence of martingales.

We need some more notations. By $\Lambda\Subset\Z^2_0$ we mean that $\Lambda$ is a finite subset of $\Z^2_0$. Let $\mathcal{FC}$ be the collection of cylindrical functions $F:H^{-1-} \to \R$ of the form $F(\xi)= f(\<\xi, e_l\>; l\in \Lambda)$ for some $\Lambda\Subset \Z^2_0$ and $f\in C_b^\infty \big(\R^\Lambda \big)$, where $\R^\Lambda$ is the Euclidean space of dimension $\#\Lambda$. For simplicity, we shall often write $F=f\circ \Pi_\Lambda$ and $f_l(\xi)= (\partial_l f)(\Pi_\Lambda \xi)$, $f_{l,m}(\xi)= (\partial_l\partial_m f)(\Pi_\Lambda \xi)$, where $\Pi_\Lambda \xi= (\<\xi, e_l\>; l\in \Lambda)= \sum_{l\in \Lambda} \<\xi, e_l\> e_l \in \R^\Lambda$. Denote by $\mathcal L$ the generator of \eqref{mSQG-space-time-white-noise}:
  \begin{equation}\label{generator}
  \mathcal LF= \L_0F + \mathcal G F,
  \end{equation}
where, for any cylindrical function $F=f\circ \Pi_\Lambda$,
  $$\L_0F = 4\pi^2 \sum_{l\in \Lambda} |l|^2 \big[f_{l,-l}(\xi) - \<\xi, e_l\> f_l(\xi) \big] $$
and the ``drift'' operator $\mathcal G$ is given by
  $$\mathcal G F = -\<u\cdot\nabla \xi, DF\>= \sum_{l\in \Lambda} f_l(\xi) \big\< \xi\otimes \xi, H_\eps^{e_l} \big\>.$$
Recall also the Malliavin derivative
  $$(D_x F)(\xi)= \sum_{l\in \Lambda} f_l(\xi) e_l(x),\quad x\in \T^2,$$
and introduce the notation
  \begin{equation}\label{energy-op}
  \mathcal E(F)(\xi)= 2\int_{\T^2} \big|(-\Delta)^{1/2}_x (D_x F)(\xi)\big|^2\,\d x= 8\pi^2 \sum_{l\in \Lambda} |l|^2 f_l(\xi) f_{-l}(\xi).
  \end{equation}
Moreover, let
  $$C_{k,l}= \frac{a_k\cdot l}{|k|^\gamma}, \quad k,l\in \Z^2_0, $$
where $a_k$ is defined in Section \ref{subsec-prelim}. The next simple result was proved in \cite[Lemma 3.4]{FlaLuo-0}; we present the proof here for the reader's convenience.

\begin{lemma}\label{lem-identity}
For all $N\geq 1$; it holds that
  $$\sum_{|k|\leq N} C_{k,l}^2 = \frac12 |l|^2 \|\theta^N\|_{\ell^2}^2. $$
\end{lemma}

\begin{proof}
Let $D_{k,l}= \frac{k\cdot l}{|k|^{\gamma+1}},\, k,l\in \Z^2_0$; then
  $$C_{k,l}^2 + D_{k,l}^2= \frac{(a_k\cdot l)^2}{|k|^{2\gamma}}+ \frac{(k\cdot l)^2}{|k|^{2\gamma+2}}= \frac{1}{|k|^{2\gamma}} \bigg[\Big(\frac{k^\perp}{|k|} \cdot l\Big)^2 + \Big(\frac{k}{|k|} \cdot l\Big)^2 \bigg]= \frac{|l|^2}{|k|^{2\gamma}}, $$
where we have used the fact that $\big\{\frac{k^\perp}{|k|}, \frac{k}{|k|} \big\}$ is an orthonormal basis of $\R^2$ for any $k\in \Z^2_0$. Next, noting that the transformation $k\mapsto k^\perp$ is an isometry on $\Z^2_0$, hence
  $$\sum_{|k|\leq N} C_{k,l}^2 = \sum_{|k|\leq N} \frac{(k^\perp\cdot l)^2}{|k|^{2\gamma+2}}= \sum_{|k|\leq N} \frac{(k\cdot l)^2}{|k|^{2\gamma+2}} =  \sum_{|k|\leq N} D_{k,l}^2 . $$
Summarizing the above facts, we arrive at
  \[\sum_{|k|\leq N} C_{k,l}^2 = \frac12 \sum_{|k|\leq N}\big( C_{k,l}^2 + D_{k,l}^2\big) = \frac12 \sum_{|k|\leq N} \frac{|l|^2}{|k|^{2\gamma}} =\frac12 |l|^2 \|\theta^N\|_{\ell^2}^2. \qedhere \]
\end{proof}

Recall the limit process $\tilde \xi_\cdot$ on the new probability space $\big(\tilde \Omega, \tilde{\mathcal F}, \tilde\P\big)$; the following result is analogous to \cite[Proposition 2.9]{FlaLuo-2}, but here we also give the explicit formula of the quadratic variation of the martingale.

\begin{proposition}\label{prop-martingale-solution}
For any $F\in \mathcal{FC}$,
  $$\tilde M^F_t:= F\big(\tilde \xi_t \big)- F\big(\tilde \xi_0 \big) - \int_0^t \mathcal L F\big(\tilde \xi_s \big)\,\d s$$
is an $\tilde{\mathcal F}_t=\sigma\big(\tilde \xi_s; s\leq t \big)$-martingale, with the quadratic variation
  $$\big[\tilde M^F \big]_t= \int_0^t \mathcal E(F)\big(\tilde \xi_s \big) \,\d s.$$
\end{proposition}

\begin{proof}
For some $l\in \Z_0^2$, taking $\phi= e_l$ in \eqref{3-stoch-mSQG-Ito.2} gives us
  \begin{equation}\label{eq-seq-1}
  \aligned
  \d \big\< \tilde \xi^{i}_t, e_l\big\>&= \big\<\tilde \xi^{i}_t\otimes \tilde \xi^{i}_t, H_\eps^{e_l}\big\>\, \d t -4 \pi^2 |l|^2 \big\< \tilde \xi^{i}_t, e_l \big\>\,\d t \\
  &\hskip13pt + \frac{\sqrt 2}{\|\theta^{N_i} \|_{\ell^2}} \sum_{|k|\leq N_i} \frac1{|k|^{\gamma}} \big\< \tilde \xi^{i}_t,\sigma_k \cdot \nabla e_l \big\>\,\d \tilde W^{i, k}_t.
  \endaligned
  \end{equation}
Therefore, for $l,m \in \Z_0^2$, the quadratic covariation
  \begin{equation*}
  \d \Big[\big\< \tilde \xi^{i}_\cdot, e_l\big\>, \big\< \tilde \xi^{i}_\cdot, e_m\big\> \Big]_t = \frac{4}{\|\theta^{N_i} \|_{\ell^2}^2 } \sum_{|k|\leq N_i} \frac1{|k|^{2\gamma}} \big\< \tilde \xi^{i}_t,\sigma_k \cdot \nabla e_l \big\> \big\< \tilde \xi^{i}_t,\sigma_{-k} \cdot \nabla e_m \big\>\,\d t,
  \end{equation*}
where we have used \eqref{BMs}. Direct computation leads to $\sigma_k \cdot \nabla e_l = 2\pi{\rm i} (a_k\cdot l) e_{k+l}$ and $\sigma_{-k} \cdot \nabla e_m =2\pi{\rm i} (a_{-k}\cdot m) e_{-k+m}$; hence, by the definition of $C_{k,l}$,
  $$\aligned
  &\ \frac1{|k|^{2\gamma}} \big\< \tilde \xi^{i}_t,\sigma_k \cdot \nabla e_l \big\> \big\< \tilde \xi^{i}_t,\sigma_{-k} \cdot \nabla e_m \big\> \\
  = &\ -4\pi^2 C_{k,l}C_{k,m} \big\< \tilde \xi^{i}_t, e_{k+l}\big\> \big\< \tilde \xi^{i}_t,e_{-k +m} \big\> \\
  = &\ -4\pi^2 C_{k,l}C_{k,m} \Big[\big\< \tilde \xi^{i}_t, e_{k+l}\big\> \big\< \tilde \xi^{i}_t,e_{-k +m} \big\> -\delta_{l,-m} \Big] + 4\pi^2 \delta_{l,-m} C_{k,l}^2.
  \endaligned$$
As a result, by Lemma \ref{lem-identity},
  \begin{equation*}
  \aligned
  &\ \d \Big[\big\< \tilde \xi^{i}_\cdot, e_l\big\>, \big\< \tilde \xi^{i}_\cdot, e_m\big\> \Big]_t \\
  = &\ -\frac{16 \pi^2}{\|\theta^{N_i} \|_{\ell^2}^2} \sum_{|k|\leq N_i} C_{k,l}C_{k,m} \Big[\big\< \tilde \xi^{i}_t, e_{k+l}\big\> \big\< \tilde \xi^{i}_t,e_{-k +m} \big\> -\delta_{l,-m} \Big] \d t \\
  &\ +8\pi^2 \delta_{l,-m} |l|^2 \,\d t.
  \endaligned
  \end{equation*}
To simplify the notations, we denote by
  \begin{equation}\label{R-l-m}
  R_{l,m}\big(\tilde \xi^{i}_t\big) = -8 \pi^2 \sum_{|k|\leq N_i} C_{k,l}C_{k,m} \Big[\big\< \tilde \xi^{i}_t, e_{k+l}\big\> \big\< \tilde \xi^{i}_t,e_{-k+m} \big\> -\delta_{l,-m} \Big].
  \end{equation}
Recall that $\tilde \xi^{i}_t$ is a white noise on $\T^2$ for any $t\in [0,T]$, thus by the classical renormalization argument, the random variable $R_{l,m}\big(\tilde \xi^{i}_t\big)$ is bounded in any $L^p\big([0,T]\times \tilde\Omega\big),\, p>1$. Finally, we arrive at
  \begin{equation}\label{eq-seq-2}
  \aligned
  \d \Big[\big\< \tilde \xi^{i}_\cdot, e_l\big\>, \big\< \tilde \xi^{i}_\cdot, e_m\big\> \Big]_t &= \frac{2}{\|\theta^{N_i} \|_{\ell^2}^2} R_{l,m}\big(\tilde \xi^{i}_t\big) \,\d t + 8 \pi^2 \delta_{l,-m} |l|^2 \,\d t.
  \endaligned
  \end{equation}

By the It\^o formula and \eqref{eq-seq-1}, \eqref{eq-seq-2},
  \begin{equation*}
  \aligned
  \d F\big(\tilde \xi^{i}_t \big)= &\ \d f\big(\big\< \tilde \xi^{i}_t, e_l\big\>; l\in \Lambda\big)\\
  =&\ \sum_{l\in \Lambda} f_l\big(\tilde \xi^{i}_t \big) \Big[\big\<\tilde \xi^{i}_t\otimes \tilde \xi^{i}_t, H_\eps^{e_l} \big\> -4 \pi^2 |l|^2 \big\< \tilde \xi^{i}_t, e_l \big\>\Big]\,\d t \\
  &\ +\frac{\sqrt 2}{\|\theta^{N_i} \|_{\ell^2}} \sum_{l\in \Lambda} f_l\big(\tilde \xi^{i}_t \big) \sum_{|k|\leq N_i} \frac1{|k|^\gamma} \big\< \tilde \xi^{i}_t,\sigma_k \cdot \nabla e_l \big\>\,\d \tilde W^{i, k}_t \\
  &\ + \sum_{l, m\in \Lambda} f_{l,m}\big(\tilde \xi^{i}_t \big) \bigg[\frac{1}{\|\theta^{N_i} \|_{\ell^2}^2} R_{l,m}\big(\tilde \xi^{i}_t\big) + 4 \pi^2 \delta_{l,-m} |l|^2 \bigg] \,\d t.
  \endaligned
  \end{equation*}
Recalling the operator $\L$ defined in \eqref{generator}, the above formula can be rewritten as
  \begin{equation}\label{Ito-formula}
  \d F\big(\tilde \xi^{i}_t \big) = \L F\big(\tilde \xi^{i}_t \big) \,\d t + \frac{1}{\|\theta^{N_i} \|_{\ell^2}^2} \tilde \zeta^{(i)}_t \,\d t + \d \tilde M^{(i)}_t,
  \end{equation}
where
  $$\tilde \zeta^{(i)}_t = \sum_{l, m\in \Lambda} f_{l,m}\big(\tilde \xi^{i}_t \big) R_{l,m}\big(\tilde \xi^{i}_t\big)$$
is bounded in $L^p\big([0,T]\times \tilde\Omega \big)$ for any $p>1$, and the martingale part
  \begin{equation}\label{martingale-i}
  \d \tilde M^{(i)}_t= \frac{\sqrt{2}}{\|\theta^{N_i} \|_{\ell^2}} \sum_{l\in \Lambda} f_l\big(\tilde \xi^{i}_t \big) \sum_{|k|\leq N_i} \frac1{|k|^\gamma} \big\< \tilde \xi^{i}_t,\sigma_k \cdot \nabla e_l \big\>\,\d \tilde W^{{i}, k}_t.
  \end{equation}
Note that $\tilde M^{(i)}_t$ is a martingale with respect to the filtration
  $$\tilde{\mathcal F}^{(i)}_t = \sigma\big(\tilde \xi^{i}_s, \tilde W^{(i)}_s: s\leq t\big),$$
where we denote by $\tilde W^{(i)}_s = \big\{ \tilde W^{i,k}_s \big\}_{k\in \Z_0^2}$.

Next, we show that the formula \eqref{Ito-formula} converges as $i\to \infty$ in a suitable sense, by following the argument of \cite[p. 232]{DaPZ}. Fix any $0<s <t\leq T$. Take a real valued, bounded and continuous function $\varphi: C\big([0,s], H^{-1-} \times \mathbb C^{\Z_0^2} \big)\to \R$. By \eqref{Ito-formula}, we have
  \begin{equation}\label{martingale-0}
  \tilde \E\bigg[\bigg( F\big(\tilde \xi^{i}_t \big) -F\big(\tilde \xi^{i}_s \big) - \int_s^t\! \L F\big(\tilde \xi^{i}_r \big) \,\d r - \frac{1}{\|\theta^{N_i} \|_{\ell^2}^2} \! \int_s^t \tilde \zeta^{(i)}_r \,\d r \bigg) \varphi\big(\tilde \xi^{i}_\cdot, \tilde W^{(i)}_\cdot \big)\bigg] =0.
  \end{equation}
Since $F\in \mathcal{FC}$ and $\tilde \xi^{i}_t$ is a white noise, all the terms in the bracket belong to $L^p\big(\tilde \P\big)$ for any $p>1$. Recall that $\|\theta^{N_i} \|_{\ell^2} \to \infty$ as $i\to \infty$ and, $\tilde\P$-a.s., $\big(\tilde \xi^{i}_\cdot, \tilde W^{(i)}_\cdot \big)$ converges in $C\big([0,T], H^{-1-} \times \mathbb C^{\Z_0^2} \big)$ to $\big(\tilde \xi_\cdot, \tilde W_\cdot \big)$, where we write $\tilde W_\cdot$ for the family of Brownian motions $\big\{ \tilde W^{k}_\cdot \big\}_{k\in \Z_0^2}$. Thus, letting $i\to \infty$ in the above equality yields
  $$\tilde \E\bigg[\bigg( F(\tilde \xi_t ) -F(\tilde \xi_s ) - \int_s^t \L F(\tilde \xi_r ) \,\d r \bigg) \varphi\big(\tilde \xi_\cdot, \tilde W_\cdot \big)\bigg] =0.$$
The arbitrariness of $0<s<t$ and $\varphi: C\big([0,s], H^{-1-} \times \mathbb C^{\Z_0^2} \big)\to \R$ implies that $\tilde M^F_\cdot$ is a martingale with respect to the filtration $\tilde{\mathcal G}_t = \sigma\big(\tilde \xi_s, \tilde W_s: s\leq t\big),\, t\in [0,T]$. For any $0\leq s< t\leq T$, we have $\tilde{\mathcal F}_s\subset \tilde{\mathcal G}_s$, thus
  $$\tilde \E \big(\tilde M^F_t \big|\tilde{\mathcal F}_s \big)= \tilde \E \Big[\tilde \E \big(\tilde M^F_t \big| \tilde{\mathcal G}_s \big) \big| \tilde{\mathcal F}_s \Big] = \tilde \E \big[\tilde M^F_s \big| \tilde{\mathcal F}_s \big] =\tilde M^F_s ,$$
since $\tilde M^F_s$ is adapted to $\tilde{\mathcal F}_s $.

Finally we prove the formula for the quadratic variation. Recall the martingale in \eqref{martingale-i}; we have, by \eqref{BMs},
  $$\aligned
  \d\big[\tilde M^{(i)} \big]_t &= \frac{4}{\|\theta^{N_i} \|_{\ell^2}^2} \sum_{l,m\in \Lambda} f_l\big(\tilde \xi^{i}_t \big) f_m\big(\tilde \xi^{i}_t \big) \sum_{|k|\leq N_i} \frac1{|k|^{2\gamma}} \big\< \tilde \xi^{i}_t,\sigma_k \cdot \nabla e_l \big\> \big\< \tilde \xi^{i}_t,\sigma_{-k} \cdot \nabla e_m \big\>\,\d t \\
  &= \frac{-16\pi^2}{\|\theta^{N_i} \|_{\ell^2}^2} \sum_{l,m\in \Lambda} f_l\big(\tilde \xi^{i}_t \big) f_m\big(\tilde \xi^{i}_t \big) \sum_{|k|\leq N_i} C_{k,l}C_{k,m} \big\< \tilde \xi^{i}_t,e_{k+l} \big\> \big\< \tilde \xi^{i}_t,e_{-k+m} \big\>\,\d t.
  \endaligned $$
Using the notation $R_{l,m}\big(\tilde \xi^{i}_t \big)$ in \eqref{R-l-m}, we obtain
  $$\aligned
  \d\big[\tilde M^{(i)} \big]_t &= \frac{2}{\|\theta^{N_i} \|_{\ell^2}^2} \sum_{l,m\in \Lambda} f_l\big(\tilde \xi^{i}_t \big) f_m\big(\tilde \xi^{i}_t \big) \bigg[R_{l,m}\big(\tilde \xi^{i}_t \big) +8\pi^2 \delta_{l,-m} \sum_{|k|\leq N_i} C_{k,l}^2 \bigg]\,\d t \\
  &= \frac{2}{\|\theta^{N_i} \|_{\ell^2}^2} \sum_{l,m\in \Lambda} f_l\big(\tilde \xi^{i}_t \big) f_m\big(\tilde \xi^{i}_t \big) R_{l,m}\big(\tilde \xi^{i}_t \big)\,\d t + 8\pi^2 \sum_{l\in \Lambda} |l|^2 f_l\big(\tilde \xi_t \big) f_{-l}\big(\tilde \xi_t \big) \,\d t,
  \endaligned $$
where in the second step we have used Lemma \ref{lem-identity}. Similarly as the arguments above, since $\|\theta^{N_i} \|_{\ell^2}\to \infty$ as $i\to \infty$, the $\tilde\P$-a.s. convergence $\tilde \xi^{i}_\cdot \to \tilde \xi_\cdot$ and the boundedness in $L^p\big(\tilde\Omega, \tilde\P\big)\, (p>1)$ of $R_{l,m}\big(\tilde \xi^{i}_\cdot \big)$ give rise to
  \begin{equation}\label{limit-variation}
  \lim_{i\to\infty} \big[\tilde M^{(i)} \big]_t= 8\pi^2 \sum_{l\in \Lambda} |l|^2 \int_0^t f_l\big(\tilde \xi_s \big) f_{-l}\big(\tilde \xi_s \big) \,\d s = \int_0^t \mathcal E(F)\big(\tilde \xi^{i}_s \big) \,\d s,
  \end{equation}
which holds $\tilde\P$-a.s. and in $L^p\big(\tilde\Omega, \tilde\P\big)\, (p>1)$. Note that $\big(\tilde M^{(i)}_t \big)^2 - \big[\tilde M^{(i)} \big]_t$ is also a $\sigma\big(\tilde \xi^i_s, \tilde W^{(i)}_s: s\leq t\big)$-martingale, $ t\in [0,T]$; similar to \eqref{martingale-0}, we have, for $0<s<t\leq T$,
  $$\tilde\E \Big[\Big\{\big(\tilde M^{(i)}_t \big)^2 - \big[\tilde M^{(i)} \big]_t- \big(\tilde M^{(i)}_s \big)^2 + \big[\tilde M^{(i)} \big]_s \Big\} \varphi\big(\tilde \xi^{i}_\cdot, \tilde W^{(i)}_\cdot \big) \Big]=0. $$
Using \eqref{limit-variation} and the expression \eqref{Ito-formula} of $\tilde M^{(i)}_t$, we can let $i\to \infty$ to conclude
  $$\tilde\E \bigg[\bigg\{\big(\tilde M^F_t \big)^2- \big(\tilde M^F_s \big)^2 -\int_s^t \mathcal E(F) \big(\tilde \xi_r \big) \,\d r \bigg\} \varphi\big(\tilde \xi_\cdot, \tilde W_\cdot \big) \bigg]=0. $$
This implies that
  $$\big(\tilde M^F_t \big)^2- \int_0^t \mathcal E(F)\big(\tilde \xi_s \big) \,\d s$$
is a $\tilde{\mathcal G}_t = \sigma\big(\tilde \xi_s, \tilde W_s: s\leq t\big)$-martingale, $ t\in [0,T]$. A little more arguments as above yield that it is also a martingale with respect to $\tilde{\mathcal F}_t = \sigma\big(\tilde \xi_s: s\leq t\big)$. Thus we obtain the expression of the quadratic variation of $\tilde M^F_t$.
\end{proof}

As a consequence, we can prove the following result.

\begin{corollary}\label{cor-BM}
There exists a family of independent complex Brownian motions $\{\bar W^k_t; t\geq 0\}_{k\in \Z^2_0}$ such that the limit process $\tilde\xi_\cdot$ and the cylindrical Brownian motion $\zeta_\cdot := \sum_k \sigma_k \bar W^k_\cdot $ solve the equations \eqref{mSQG-space-time-white-noise}.
\end{corollary}

\begin{proof}
For $l\in \Z^2_0$, taking the cylinder function $F(\xi)= \<\xi, e_l\>$ in Proposition \ref{prop-martingale-solution} yields the martingale
  \begin{equation}\label{cor-BM.0}
  \tilde M^{(l)}_t= \big\< \tilde \xi_t, e_l\big\> - \big\< \tilde \xi_0, e_l\big\> + 4\pi^2|l|^2 \int_0^t \big\< \tilde \xi_s, e_l\big\>\,\d s - \int_0^t \big\< \tilde \xi_s\otimes \tilde\xi_s, H_\eps^{e_l} \big\>\,\d s .
  \end{equation}
Moreover, for $F(\xi)= \<\xi, e_l\>\<\xi, e_m\>,\, l,m\in \Z^2_0, m\neq -l$, we obtain the martingale
  $$\aligned
  \tilde M^{(l,m)}_t=&\ \big\< \tilde \xi_t, e_l\big\> \big\< \tilde \xi_t, e_m\big\>- \big\< \tilde \xi_0, e_l\big\> \big\< \tilde \xi_0, e_m\big\> + 4\pi^2 (|l|^2+|m|^2) \int_0^t \big\< \tilde \xi_s, e_l\big\> \big\< \tilde \xi_s, e_m\big\>\,\d s \\
  &- \int_0^t\Big( \big\< \tilde \xi_s, e_m\big\> \big\< \tilde \xi_s\otimes \tilde\xi_s, H_\eps^{e_l} \big\>+ \big\< \tilde \xi_s, e_l\big\> \big\< \tilde \xi_s\otimes \tilde\xi_s, H_\eps^{e_m} \big\> \Big)\,\d s.
  \endaligned $$
From these facts one can deduce
  \begin{equation}\label{cor-BM.1}
  \big[\tilde M^{(l)}, \tilde M^{(m)}\big]_t=0 \quad \mbox{for all } m\neq -l.
  \end{equation}

Next, taking $F(\xi)= \<\xi, e_l\>\<\xi, e_{-l}\>$ in Proposition \ref{prop-martingale-solution} gives us the martingale
  $$\aligned
  \tilde M^{(l,-l)}_t=&\ \big\< \tilde \xi_t, e_l\big\> \big\< \tilde \xi_t, e_{-l}\big\>- \big\< \tilde \xi_0, e_l\big\> \big\< \tilde \xi_0, e_{-l}\big\> - 8\pi^2 |l|^2 \int_0^t \big(1- \big\< \tilde \xi_s, e_l\big\> \big\< \tilde \xi_s, e_{-l}\big\> \big)\,\d s \\
  &- \int_0^t\Big( \big\< \tilde \xi_s, e_{-l}\big\> \big\< \tilde \xi_s\otimes \tilde\xi_s, H_\eps^{e_l} \big\>+ \big\< \tilde \xi_s, e_l\big\> \big\< \tilde \xi_s\otimes \tilde\xi_s, H_\eps^{e_m} \big\> \Big)\,\d s.
  \endaligned $$
Therefore, we have
  $$\big[\tilde M^{(l)}, \tilde M^{(-l)}\big]_t= 8\pi^2 |l|^2 t. $$
Combining this with \eqref{cor-BM.1}, and using L\'evy's characterization of Brownian motions, we conclude that $\big\{ \tilde M^{(l)}_t \big\}_{l\in \Z^2_0}$ are independent complex Brownian motions, with the quadratic variation process $4\pi^2 |l|^2 t$. Finally, if we define
  $$\bar W^l_t= \frac1{2\pi{\rm i} (a_l\cdot l^\perp)} \tilde M^{(-l)}_t, \quad l\in \Z^2_0, $$
then by \eqref{cor-BM.0}, we have
  $$\big\< \tilde \xi_t, e_l\big\> = \big\< \tilde \xi_0, e_l\big\> - 4\pi^2|l|^2 \int_0^t \big\< \tilde \xi_s, e_l\big\>\,\d s + \int_0^t \big\< \tilde \xi_s\otimes \tilde\xi_s, H_\eps^{e_l} \big\>\,\d s - 2\pi{\rm i} (a_l\cdot l^\perp) \bar W^{-l}_t, \quad l\in \Z^2_0.$$
These are the componentwise version of \eqref{mSQG-space-time-white-noise}.
\end{proof}

Finally we can prove the main result of this paper.

\begin{proof}[Proof of Theorem \ref{thm-scaling-limit}]
From Corollary \ref{cor-BM} we see that any weakly convergent subsequence $\big\{\tilde \xi^{N_i} \big\}_{i\geq 1}$ converges weakly to the stationary white noise solutions of \eqref{mSQG-space-time-white-noise}. Furthermore, by the following remark and Proposition \ref{prop-Ito-trick}, we know that the limit solves the cylinder martingale problem for the operator $\L$ with initial distribution $\mu$ in the sense of Definition \ref{def-cylinder-martingale}. Since the family $\{Q^N\}_{N\geq 1}$ is tight, by the weak uniqueness of martingale solutions to \eqref{mSQG-space-time-white-noise} in Theorem \ref{thm-uniqueness}, we conclude that the whole sequence $\big\{\xi^{N} \big\}_{N\geq 1}$ is weakly convergent.
\end{proof}

We want to prove another property of the limit process $\tilde\xi_\cdot$ which will be useful in the next section. Recall the equation \eqref{3-stoch-mSQG-Ito.1} at the beginning of this section. In view of Remark \ref{rem-time-reversal-1}, for fixed $T>0$, we can also consider the time-reversed process $\hat\xi^N_t := \xi^N_{T-t}$ which solves a similar equation, with a minus sign in front of the nonlinear part. The subsequent proofs work as well for these new processes and we can prove an analogue of Proposition \ref{prop-martingale-solution}. More precisely, for the limit process $\tilde\xi_t$ solving \eqref{3-stoch-mSQG-Ito.2}, we define the time reversal $\hat\xi_t:= \tilde\xi_{T-t}$, then
  $$\hat M^F_t:= F\big(\hat \xi_t \big)- F\big(\hat \xi_0 \big) - \int_0^t \hat{\mathcal L} F\big(\hat \xi_s \big)\,\d s$$
is an $\hat{\mathcal F}_t=\sigma\big(\hat \xi_s; s\leq t \big)$-martingale, with the quadratic variation
  $$\big[\hat M^F \big]_t= \int_0^t \mathcal E(F)\big(\hat \xi_s \big) \,\d s.$$
Here, $\hat{\mathcal L}$ is the adjoint operator of $\L$ defined in \eqref{generator}:
  $$\hat{\L}F=\L_0F -\mathcal GF = 4\pi^2 \sum_{l\in \Lambda} |l|^2 \big[f_{l,-l}(\xi) - \<\xi, e_l\> f_l(\xi) \big] + \<u\cdot\nabla \xi, DF\>.$$

To state the next result, we need a few notations. Let $\mathcal N$ be the number (or Ornstein-Uhlenbeck) operator on $L^2(H^{-1-},\mu)$. Given a function $w:\mathbb N\to \R_+$, the operator $w(\mathcal N)$ acts on $F\in L^2(H^{-1-},\mu)$ as a spectral multiplier; more precisely, let $F= \sum_{n\geq 0} F_n$ be the Wiener chaos expansion, then $w(\mathcal N)F= \sum_{n\geq 0} w(n) F_n$. It is self-adjoint. Moreover, let $c_p= \sqrt{p-1}$ for $p\geq 2$; we have the following well known hypercontractivity:
  $$\big\| |F|^{p/2} \big\|_{L^2(\mu)}^2 \leq \big\|c_p^{\mathcal N} F \big\|_{L^2(\mu)}^p. $$
With these preparations, we can use the It\^o trick in \cite{GJ} to prove

\begin{proposition}\label{prop-Ito-trick}
Let $F\in \mathcal{FC}$ have zero mean. For all $p\geq 2$,
  \begin{equation}\label{prop-Ito-trick.1}
  \tilde\E\bigg[ \sup_{0\leq t\leq T} \bigg| \int_0^t F\big( \tilde\xi_s\big) \,\d s \bigg|^p \bigg] \lesssim T^{p/2} \big\|c_{p}^{\mathcal N}(-\L_0)^{-1/2} F \big\|_{L^2(\mu)}^p .
  \end{equation}
\end{proposition}

\begin{proof}
The arguments above tell us that, for $0<t<T$,
  $$F\big( \hat\xi_T\big)= F\big( \hat\xi_{T-t}\big) + \int_{T-t}^T \hat{\mathcal L} F\big(\hat \xi_s \big)\,\d s + \hat M^F_T- \hat M^F_{T-t}; $$
equivalently,
  $$F\big( \tilde\xi_0\big)= F\big( \tilde\xi_t\big) + \int_0^t \L_0 F\big(\tilde\xi_s \big)\,\d s- \int_0^t \mathcal G F\big(\tilde\xi_s \big)\,\d s + \hat M^F_T- \hat M^F_{T-t}. $$
Combining this equality with Proposition \ref{prop-martingale-solution}, namely,
  $$F\big( \tilde\xi_t\big)= F\big( \tilde\xi_0\big) + \int_0^t \L_0 F\big(\tilde\xi_s \big)\,\d s + \int_0^t \mathcal G F\big(\tilde\xi_s \big)\,\d s + \tilde M^F_t, $$
we obtain
  $$0= 2\int_0^t \L_0 F\big(\tilde\xi_s \big)\,\d s+ \tilde M^F_t + \hat M^F_T- \hat M^F_{T-t}. $$

Therefore, recalling the quadratic variations of the martingales $\tilde M^F_t$ and $\hat M^F_t$, we have
  $$\aligned
  \tilde\E \bigg[ \sup_{t\in [0,T]} \bigg| \int_0^t \L_0 F\big(\tilde\xi_s \big) \,\d s \bigg|^p \bigg] &\lesssim \tilde\E \bigg[ \bigg( \int_0^T \mathcal E(F)\big(\tilde\xi_s \big) \,\d s \bigg)^{p/2} \bigg] \\
  &\leq T^{p/2-1} \int_0^T \tilde\E \big[\mathcal E(F)\big(\tilde\xi_s \big) \big]^{p/2} \,\d s \\
  &= T^{p/2} \big\|\mathcal E(F)^{p/2} \big\|_{L^1(\mu)} = T^{p/2} \big\|\mathcal E(F)^{p/4} \big\|_{L^2(\mu)}^2.
  \endaligned $$
By the hypercontractivity and \eqref{lem-martingale.1} below,
  $$\tilde\E \bigg[ \sup_{t\in [0,T]} \bigg| \int_0^t \L_0 F\big(\tilde\xi_s \big) \,\d s \bigg|^p \bigg] \lesssim T^{p/2} \big\| c_p^{\mathcal N} \mathcal E(F)^{1/2} \big\|_{L^2(\mu)}^p \lesssim T^{p/2} \big\| c_p^{\mathcal N} (-\L_0)^{1/2} F \big\|_{L^2(\mu)}^p . $$
Finally, let $\psi= (-\L_0)^{-1} F$, and use the above inequality, we have
  $$\aligned
  \tilde\E \bigg[ \sup_{t\in [0,T]} \bigg| \int_0^t F\big(\tilde\xi_s \big) \,\d s \bigg|^p \bigg] &= \tilde\E \bigg[ \sup_{t\in [0,T]} \bigg| \int_0^t (-\L_0)\psi\big(\tilde\xi_s \big) \,\d s \bigg|^p \bigg] \\
  &\lesssim T^{p/2} \big\| c_p^{\mathcal N} (-\L_0)^{1/2} \psi \big\|_{L^2(\mu)}^p \\ 
  &= T^{p/2}\big\|c_p^{\mathcal N}(-\L_0)^{-1/2} F \big\|_{L^2(\mu)}^p .
  \endaligned $$
The proof is complete.
\end{proof}

\section{Uniqueness of the martingale problem to \eqref{mSQG-space-time-white-noise}}

In this section, we follow the approach in \cite{GP-18, GT} to prove the uniqueness of the cylinder martingale problem associated to \eqref{mSQG-space-time-white-noise}.

\subsection{Galerkin approximation and a priori estimates}

We consider the Galerkin approximation of \eqref{mSQG-space-time-white-noise}: for $m\geq 1$,
  \begin{equation}\label{2D-NSEs-Galerkin}
  \partial_t \xi^m+ B_m(\xi^m) = \Delta \xi^m +  \nabla^\perp \cdot \dot \zeta,
  \end{equation}
where
  $$B_m(\xi)= \div \Pi_m \big((K_\eps\ast \Pi_m\xi) \Pi_m\xi \big).$$
In the above identity, $\Pi_m$ is the orthogonal projection to Fourier modes $0<|k|\leq m$, and $K_\eps$ is the kernel on $\T^2$ corresponding to the operator $\nabla^\perp (-\Delta)^{-(1+\eps)/2}$:
  $$K_\eps(x) = \frac{\rm i}{(2\pi)^\eps } \sum_{k\in \Z^2_0} \frac{k^\perp}{|k|^{1+\eps}} e_k(x).$$
where $e_k(x)= {\rm e}^{2\pi {\rm i} k\cdot x}$. Let $\rho^m(x) = \sum_{|k|\leq m} e_k(x)$, then
  $$\Pi_m f= f \ast \rho^m.$$
Note that
  $$K_\eps \ast \Pi_m\xi = K_\eps\ast (\rho^m \ast \xi) = (K_\eps\ast \rho^m)\ast \xi=: K^m_\eps \ast \xi,$$
where
  $$K^m_\eps(x) = \frac{\rm i}{(2\pi)^\eps } \sum_{|k|\leq m} \frac{k^\perp}{|k|^{1+\eps}} e_k(x). $$
Using these notations, we have
  \begin{equation}\label{drift}
  B_m(\xi)= \div\Big[ \rho^m \ast \big((K^m_\eps \ast \xi) \rho^m \ast \xi \big) \Big].
  \end{equation}

It is well known that the equation \eqref{2D-NSEs-Galerkin} has a unique strong solution $\xi \in C(\R_+, H^{-1-}(\T^2))$ for any deterministic initial condition $\xi_0\in H^{-1-}(\T^2)$. Let $\<\, ,\>$ be the duality between distributions and smooth functions on $\T^2$. For any $k\in \Z^2_0$, the component form of \eqref{2D-NSEs-Galerkin} is
  \begin{equation}\label{2D-NSEs-component}
  \d\<\xi^m, e_k\> = -\<B_m(\xi^m), e_k\>\,\d t -4\pi^2|k|^2 \<\xi^m,e_k\>\,\d t - 2\pi {\rm i}(a_k\cdot k^\perp) \,\d W^{-k}_t,
  \end{equation}
where
  $$a_k\cdot k^\perp= |k|\big({\bf 1}_{\Z^2_+}(k)- {\bf 1}_{\Z^2_-}(k) \big). $$
Let $F(\xi)= f(\<\xi,e_k\>;k\in \Lambda)=: f(\Pi_\Lambda \xi)$ be a cylinder functional, and define the operators
  \begin{equation}\label{operator-1}
  \mathcal L_0 F(\xi) = 4\pi^2\sum_{k\in \Lambda} |k|^2\Big[ \partial_k\partial_{-k} f(\Pi_\Lambda \xi) -  \<\xi,e_k\> \partial_k f(\Pi_\Lambda \xi) \Big] ,
  \end{equation}
  \begin{equation}\label{operator-2}
  \mathcal G^m F(\xi) = -\sum_{k\in \Lambda} \partial_k f(\Pi_\Lambda \xi) \<B_m(\xi), e_k\>= \sum_{k\in \Lambda} \partial_k f(\Pi_\Lambda \xi) \big\<\Pi_m \xi\otimes \Pi_m \xi, H^{e_k}_\eps \big\>.
  \end{equation}
Moreover, let $\mathcal L^m= \mathcal L_0 + \mathcal G^m$, then by It\^o's formula,
  $$\d F(\xi^m_t)= \mathcal L^m F(\xi^m_t) \,\d t + \sum_{k\in \Lambda} \partial_k f(\Pi_\Lambda \xi^m_t)\, \d M_t(e_k), $$
where $M_t(e_k)= 2\pi {\rm i}(a_k\cdot k^\perp) W^{-k}_t$.

Recall the definition of the operator $\mathcal G$ below \eqref{generator}: for any cylinder function $\varphi=\varphi\circ \Pi_\Lambda$,
  $$\mathcal{G}\varphi= -\<u\cdot\nabla \xi, D\varphi\>= \sum_{l\in \Lambda} \partial_l \varphi(\xi) \big\< \xi\otimes \xi, H_\eps^{e_l} \big\>.$$

\begin{lemma}\label{lem-limit-operator}
For any $\varphi\in \mathcal{FC}$, the following limit holds in $L^2(\mu)$:
  $$\mathcal G \varphi= \lim_{m\to \infty} \mathcal G^m \varphi.$$
\end{lemma}

\begin{proof}
It is sufficient to prove that for $l\in\Lambda$,
  \begin{equation}\label{e:1}
  \big\< \xi\otimes \xi, H_\eps^{e_l} \big\>=\lim_{m\to\infty}\big\< \xi^m\otimes \xi^m, H_\eps^{e_l} \big\> \quad \textrm{ holds in } L^2(\mu).
  \end{equation}
Here we write $ \xi^m=\Pi_m \xi= \sum_{0<|k|\leq m} \<\xi,e_k\> e_k$. The case for $\eps=1$ has been proved in \cite[Theorem A.9]{FlaLuo-2}; for $\eps\in (0,1)$ the proof is similar. We recall it for completeness. Define $\tilde{\xi}^m:=\sum_{|k|\leq m}\<\xi, e_k\> e_k$; then $\xi^m=\tilde{\xi}^m-\<\xi,1\>$. By the triangle inequality we have
  \begin{align*}
  &\ \E_\mu\big[\big( \big\<\xi\otimes \xi, H^{e_l}_\eps \big\>- \big\<\tilde{\xi}^m\otimes \tilde{\xi}^m, H^{e_l}_\eps \big\>\big)^2 \big]\\ \lesssim&\ \E_\mu\big[ \big\<\xi\otimes \xi, H^{e_l}_\eps -H^{e_l,n}_\eps \big\>^2 \big] +\E_\mu\big[ \big(\big\<\xi\otimes \xi, H^{e_l,n}_\eps \big\>- \big\<\tilde{\xi}^m\otimes \tilde{\xi}^m, H^{e_l,n}_\eps \big\>\big)^2\big]\\
  &\, +\E_\mu\big[ \big\<\tilde{\xi}^m\otimes \tilde{\xi}^m, H^{e_l}_\eps -H^{e_l,n}_\eps \big\>^2\big],
\end{align*}
where $H^{e_l,n}_\eps$ is the approximation to $H^{e_l}_\eps$  in Proposition  \ref{2-prop-2}. By using the same argument as in the proof of \cite[Proposition A.6]{FlaLuo-2} we obtain
\begin{align*}
&\limsup_{m\to\infty}\E_\mu\big[ \big( \big\<\xi\otimes \xi, H^{e_l}_\eps\big\> -\big\<\tilde{\xi}^m\otimes \tilde{\xi}^m, H^{e_l}_\eps \big\>\big)^2\big] \lesssim \|H^{e_l}_\eps -H^{e_l,n}_\eps \|_{L^2(\T^2\times \T^2)}^2.
\end{align*}
Letting $n\to\infty$ and using Proposition \ref{2-prop-2},
$$\big\< \xi\otimes \xi, H_\eps^{e_l} \big\>=\lim_{m\to\infty}\big\< \tilde{\xi}^m\otimes \tilde{\xi}^m, H_\eps^{e_l} \big\> \quad \textrm{ holds in } L^2(\mu).$$
By similar argument as the proof of \cite[Lemma A.7]{FlaLuo-2} we have
$$\big\< \tilde{\xi}^m\otimes \tilde{\xi}^m, H_\eps^{e_l} \big\>=\big\< {\xi}^m\otimes {\xi}^m, H_\eps^{e_l} \big\>,$$
which implies \eqref{e:1} and the result follows.
\end{proof}

We want to find the formula for the operator $\mathcal G^m$. Throughout this section we assume the Wiener functionals $\varphi\in L^2(\mu)$ have zero average. We will use the chaos expansion
  $$\varphi= \sum_{n\geq 1} W_n(\varphi_n),$$
where $W_n$ is the multiple Wiener-It\^o integral for symmetric functions $\varphi_n\in L^2_0((\T^2)^n)$. For simplicity, we write $\T^{2n}$ instead of $(\T^2)^n$. Moreover,
  $$\|\varphi \|_{L^2(\mu)}^2 = \sum_{n\geq 1} n! \, \|\varphi_n \|_{L^2(\T^{2n})}^2. $$
As in \cite{GP-18}, we denote $\int_{x}$ and $\int_{x,s}$ for integrals with respect to the variables $x\in \T^2$ and $x,s\in \T^2$. Recall the definitions of $\rho^m$ and $K^m_\eps$ at the beginning of this section. Since $\eps\in (0,1]$ is fixed, we shall omit it and simply write $K^m$ instead of $K^m_\eps$.

\begin{lemma}\label{lem-1}
Let $\varphi\in \mathcal{FC}$ have the chaos expansion $\varphi = \sum_{n\geq 0} W_n(\varphi_n)$. Then $\mathcal G^m= \mathcal G^m_+ + \mathcal G^m_-$, and writing $f_x= f(x-\cdot)$ and $K^m_s= K^m(s-\cdot)$, we have
  $$\aligned
  \mathcal G^m_+ W_n(\varphi_n) &= n W_{n+1} \bigg(\int_{x,s} \varphi_n(x,\cdot) \big(K^m_s \hat\otimes \rho^m_s \big) \cdot \nabla_x \rho^m_x(s) \bigg), \\
  \mathcal G^m_- W_n(\varphi_n) &= 2n(n-1) W_{n-1} \bigg(\int_{x,y,s} \varphi_n(x,y,\cdot) \big(K^m_s \hat\otimes \rho^m_s \big)(y,\cdot) \cdot \nabla_x \rho^m_x(s) \bigg),
  \endaligned$$
where $(f\hat\otimes g)(x,y)= \frac12 (f(x) g(y) + g(x) f(y))$. Moreover, for all $\varphi_n\in L^2_0(\T^{2n})$ and $\varphi_{n+1}\in L^2_0(\T^{2(n+1)})$, it holds
  $$\big\< W_{n+1}(\varphi_{n+1}), \mathcal G^m_+ W_n(\varphi_n) \big\>= - \big\< \mathcal G^m_- W_{n+1}(\varphi_{n+1}), W_n(\varphi_n) \big\>. $$
\end{lemma}

The proof is similar to that of \cite[Lemma 2.4]{GP-18} and will be given in the appendix. To derive the expressions of $\mathcal G^m_+$ and $ \mathcal G^m_-$, we will work on the Fock space $\mathcal H= \Gamma L^2_0(\T^2)= \otimes_{n=1}^\infty L^2_0(\T^{2n})$ equipped with the norm
  $$\|\varphi \|_{\mathcal H}^2 = \sum_{n\geq 1} n! \, \|\varphi_n \|_{L^2(\T^{2n})}^2 =  \sum_{n\geq 1} n! \sum_{k\in (\Z^2_0)^n} |\hat \varphi_n(k)|^2. $$
We denote with the same notations $\mathcal L_0, \mathcal G^m_+, \mathcal G^m_-$ the operators on the Fock space in such a way that, on smooth cylinder functions, we have
  $$\mathcal L_0 \bigg[\sum_{n\geq 1} W_n(\varphi_n) \bigg] = \sum_{n\geq 1} W_n((\mathcal L_0\varphi)_n), \quad \mathcal G^m_\pm \bigg[ \sum_{n\geq 1} W_n(\varphi_n) \bigg]= \sum_{n\geq 1} W_n((\mathcal G^m_\pm \varphi)_n). $$
In the following, for $m,n\in \Z_+,\, m\leq n$, $k_{m:n}$ denotes the vector $(k_m, k_{m+1}, \ldots, k_n)$ where $k_i\in \Z_0^2$. Similarly, $r_{m:n}= (r_m,\ldots, r_n)$ with $r_i\in \T^2$. The following results will be proved in the appendix.

\begin{lemma}\label{lem-drift}
For sufficiently nice function $\varphi\in \mathcal H$, the operators $\mathcal L_0, \mathcal G^m_+, \mathcal G^m_-$ are given in Fourier variables by
  \begin{equation*}
  \aligned
  \mathcal F(\mathcal L_0\varphi)_n(k_{1:n}) &= -4\pi^2 \big(|k_1|^2 + \cdots + |k_n|^2 \big) \hat \varphi_n(k_{1:n}), \\
  \mathcal F(\mathcal G^m_+ \varphi)_n(k_{1:n}) &= \frac{(2\pi)^{1-\eps}}2 (n-1) {\bf 1}_{|k_1|, |k_2|, |k_1+k_2|\leq m} \bigg(\frac{k_1^\perp\cdot k_2}{|k_1|^{1+\eps}} + \frac{k_2^\perp\cdot k_1}{|k_2|^{1+\eps}} \bigg)  \hat \varphi_{n-1}(k_1+k_2, k_{3:n}), \\
  \mathcal F(\mathcal G^m_- \varphi)_n(k_{1:n}) &= -\frac{(2\pi)^{1-\eps}}2 (n+1)n {\bf 1}_{|k_1|\leq m}\sum_{|p|,|q|\leq m, p+q=k_1} \bigg(\frac{p^\perp \cdot q}{|p|^{1+\eps}}+ \frac{q^\perp \cdot p}{|q|^{1+\eps}}\bigg) \hat \varphi_{n+1}(p,q,k_{2:n}) .
  \endaligned
  \end{equation*}
\end{lemma}

Next we want to obtain some a priori estimates on the drift $\mathcal G^m$. Following \cite{GP-18, GT}, we introduce the notion of weight.

\begin{definition}
We call $w: \N\to \R_+$ a weight function if $w$ is increasing and there exists $C>0$ such that $w(x)\leq Cw(x-1)$ for all $x\geq 1$. We write $|w|$ for the smallest such constant $C>0$.
\end{definition}

We will prove the following key estimates; $\|\cdot \|$ without subscripts means the norm in $L^2(\mu)$.

\begin{lemma}\label{lem-apriori}
Let $w: \N\to \R_+$ and $\varphi\in \mathcal{FC}$. Then we have, uniformly in $m$,
  \begin{equation}\label{lem-apriori.1}
  \big\|w(\mathcal N) (-\mathcal L_0)^{-\gamma} \mathcal G^m_+ \varphi \big\| \lesssim \big\|w(1+\mathcal N) (1+\mathcal N) (-\mathcal L_0)^{1-\gamma -\eps/2} \varphi \big\|
  \end{equation}
for all $\gamma>(1-\eps)/2 $, and
  \begin{equation}\label{lem-apriori.2}
  \big\|w(\mathcal N) (-\mathcal L_0)^{-\gamma} \mathcal G^m_- \varphi \big\|\lesssim \big\|w(\mathcal N-1) \mathcal N (-\mathcal L_0)^{1-\gamma -\eps/4} \varphi \big\|
  \end{equation}
for $\gamma \leq(2-\eps)/4$. We also have the $m$-dependent estimate:
  \begin{equation}\label{lem-apriori.3}
  \big\| w(\mathcal N) \mathcal G^m \varphi \big\| \lesssim m \big\| \big(w(\mathcal N+1)+ w(\mathcal N-1)\big) (\mathcal N+1) (-\mathcal L_0)^{1/2} \varphi \big\|.
  \end{equation}
\end{lemma}

We postpone the proof of Lemma \ref{lem-apriori} to the appendix. By \eqref{lem-apriori.3}, it is natural to identify a dense domain $\mathcal D (\mathcal L^m)$ for $\mathcal L^m$ as
  $$\mathcal D (\mathcal L^m)= \big\{\varphi\in \mathcal H: \|(1+\mathcal N) \mathcal L_0 \varphi\|<\infty \big\}= (1+\mathcal N)^{-1} (-\mathcal L_0)^{-1} (\mathcal H). $$
We have $\<\psi, (\mathcal L_0+ \mathcal G^m) \varphi \>= \<(\mathcal L_0- \mathcal G^m) \psi, \varphi \>$ for any $\psi,\varphi\in \mathcal D (\mathcal L^m)$. Since $\mathcal L_0$ is dissipative, for all $\varphi\in \mathcal D (\mathcal L^m)$, one has
  $$\<\varphi, (\mathcal L_0+ \mathcal G^m) \varphi \>= \<\varphi, \mathcal L_0 \varphi \>= -\|(-\mathcal L_0 )^{-1/2} \varphi\|^2 \leq 0. $$

We shall consider the Galerkin approximations with ``near-stationary'' fixed-time marginal.

\begin{definition}
A stochastic process $(\xi_t)_{t\geq 0}$ with values in $C^\infty(\T^2)'$ is ($L^2$-)incompressible if for all $T>0$, there exists a constant $C(T)$ such that
  $$\sup_{0\leq t\leq T} \E |\varphi(\xi_t)| \leq C(T) \|\varphi \|, \quad \varphi\in \mathcal{FC}. $$
\end{definition}

By density argument, for an incompressible process $(\xi_t)_{t\geq 0}$ and for all $\varphi\in \mathcal H$, we can define $s\mapsto \varphi(\xi_s)$ as a stochastic process continuous in $L^1$. Similarly to \cite[Lemma 5]{GT}, we have

\begin{lemma}\label{lem-incompress}
For $\eta\in L^2(\mu)$, let $\E_{\eta\d \mu}$ be the expectation with respect to the law of the solution $\xi^m$ to the Galerkin approximation \eqref{2D-NSEs-Galerkin} with the initial condition $\xi^m_0 \sim \eta\d \mu$. Then, for any $\Psi\in C(\R_+, C^\infty(\T^2)')\to \R$,
  $$\E_{\eta\d \mu} |\Psi(\xi^m )| \leq \|\eta\| \,(\E_\mu \Psi(\xi^m )^2)^{1/2}. $$
In particular, any such process is incompressible uniformly in $m$.
\end{lemma}

Next we prove an analogue of \cite[Lemma 6]{GT}; recall that $D_x$ is the Malliavin derivative.

\begin{lemma}\label{lem-martingale}
Let $\eta\in L^2(\mu)$ and $\xi^m$ be a solution to the Galerkin approximation \eqref{2D-NSEs-Galerkin} with the initial condition $\xi^m_0 \sim \eta\d \mu$. Then this solution is incompressible and, for any $\varphi\in \mathcal D(\mathcal L^m)$, the process
  $$M^{m,\varphi}_t = \varphi(\xi^m_t)- \varphi(\xi^m_0)- \int_0^t \mathcal L^m \varphi(\xi^m_s)\,\d s$$
is a continuous martingale with quadratic variation
  \begin{equation}\label{lem-martingale.0}
  \< M^{m,\varphi}\>_t = \int_0^t \mathcal E(\varphi)(\xi^m_s)\,\d s ,\quad \mbox{with } \mathcal E(\varphi) = 2\int_{\T^2} \big| (-\Delta)^{1/2} D_x\varphi \big|^2\,\d x.
  \end{equation}
For any weight $w$, we have
  \begin{equation}\label{lem-martingale.1}
  \big\| w(\mathcal N) (\mathcal E(\varphi))^{1/2} \big\| \lesssim \big\| w(\mathcal N-1) (-\mathcal L_0)^{1/2} \varphi\big\|.
  \end{equation}
Moreover, for all $p\geq 1$,
  \begin{equation}\label{lem-martingale.2}
  \E \sup_{0\leq t\leq T} \bigg| \int_0^t \varphi(\xi^m_s) \,\d s \bigg|^p \lesssim T^{p/2} \big\|c_{2p}^{\mathcal N}(-\L_0)^{-1/2}\varphi \big\|_{L^2(\mu)}^p .
  \end{equation}
\end{lemma}

\begin{proof}
For a cylinder function $\varphi(\xi)= \varphi(\<\xi,e_{k_j}\>; 1\leq j\leq n)$, applying the It\^o formula yields
  $$M^{m,\varphi}_t = -2\pi {\rm i} \sum_{j=1}^n (a_{k_j}\cdot k_j^\perp) \int_0^t (\partial_{k_j} \varphi)(\xi^m_s)\,\d W^{-k_j}_s. $$
It is clear that
  $$\< M^{m,\varphi}\>_t = \int_0^t \mathcal E(\varphi)(\xi^m_s)\,\d s. $$
Therefore, by Doob's inequality and Lemma \ref{lem-incompress},
  $$\E\bigg[\sup_{t\in [0,T]} |M^{m,\varphi}_t|\bigg] \lesssim \E \big(\<M^{m,\varphi}\>_T^{1/2} \big) \leq \|\eta \| \big(\E_\mu \<M^{m,\varphi}\>_T \big)^{1/2}. $$
Note that the process $\xi^m$ has the invariant law $\mu$, we have
  $$\E_\mu \<M^{m,\varphi}\>_T = \E_\mu \int_0^T \mathcal E(\varphi)(\xi^m_s)\,\d s= T \|\mathcal E(\varphi) \|_{L^1(\mu)}= T \big\|\mathcal E(\varphi)^{1/2} \big\|_{L^2(\mu)}^2.  $$
Combining the above two results gives us
  $$\E\bigg[\sup_{t\in [0,T]} |M^{m,\varphi}_t|\bigg] \lesssim \|\eta \| \sqrt{T}\, \big\|\mathcal E(\varphi)^{1/2} \big\|_{L^2(\mu)}.  $$

Next, we prove \eqref{lem-martingale.1}. Since $(D_x \varphi)_n = n \varphi_n(x,\cdot)$,
  $$\aligned
  \big\|w(\mathcal N) \mathcal E(\varphi)^{1/2} \big\|_{L^2(\mu)}^2&= 2\int_{\T^2} \big\|w(\mathcal N) (-\Delta)^{1/2} D_x\varphi \big\|_{L^2(\mu)}^2\,\d x \\
  &=  2\int_{\T^2} \bigg(\sum_{n\geq 1} (n-1)!\, w(n-1)^2 n^2 \big\| (-\Delta)^{1/2} \varphi_n(x,\cdot) \big\|_{L^2(\T^{2(n-1)})}^2 \bigg) \,\d x \\
  &\backsimeq \sum_{n\geq 1} n!\, w(n-1)^2 n \sum_{k_{1:n}} 4\pi^2|k_1|^2 |\hat\varphi_n(k_{1:n})|^2 \\
  &\backsimeq \sum_{n\geq 1} n!\, w(n-1)^2 \sum_{k_{1:n}} \big(|k_1|^2+ \cdots +|k_n|^2\big) |\hat\varphi_n(k_{1:n})|^2 \\
  &\backsimeq \big\|w(\mathcal N-1) (-\mathcal L_0)^{1/2}\varphi \big\|_{L^2(\mu)}^2,
  \endaligned$$
which gives us \eqref{lem-martingale.1}. Using the bounds \eqref{lem-apriori.3} and \eqref{lem-martingale.1} and by a density argument, we can extend the last two formulae in Lemma \ref{lem-drift} to all functions in $\mathcal D(\mathcal L^m)$.

The proof of the last estimate is similar to that of Proposition \ref{prop-Ito-trick}; this is because if the process $\xi^m$ starts from the stationary distribution $\mu$, then the reversed process $(\tilde \xi^m= \xi^m_{T-t})_{t\in [0,T]}$ is also stationary with the generator $\tilde{\mathcal L}^m = \L_0 -\mathcal G^m$. Thus we have
  $$\E_\mu \bigg[ \sup_{0\leq t\leq T} \bigg| \int_0^t \varphi(\xi^m_s) \,\d s \bigg|^{p} \bigg] \lesssim T^{p/2} \big\|c_{p}^{\mathcal N} (-\L_0)^{-1/2}\varphi \big\|_{L^2(\mu)}^{p}. $$
This is an estimate uniform in $m$. Finally, by Lemma \ref{lem-incompress},
  $$\aligned
  \E \bigg[ \sup_{0\leq t\leq T} \bigg| \int_0^t \varphi(\xi^m_s) \,\d s \bigg|^p \bigg] &\lesssim \bigg[\E_\mu \bigg(\sup_{0\leq t\leq T} \bigg| \int_0^t \varphi(\xi^m_s) \,\d s \bigg|^{2p} \bigg) \bigg]^{1/2} \\
  &\lesssim \Big[T^{p}\big\|c_{2p}^{\mathcal N}(-\L_0)^{-1/2}\varphi \big\|_{L^2(\mu)}^{2p} \Big]^{1/2}\\
  &= T^{p/2} \big\|c_{2p}^{\mathcal N}(-\L_0)^{-1/2}\varphi \big\|_{L^2(\mu)}^{p},
  \endaligned $$
which completes the proof.
\end{proof}

\subsection{The cylinder martingale problem}

We shall take limit of Galerkin approximations and characterize the limit dynamics. The estimate \eqref{lem-martingale.2} suggests us that any limit process $(\xi_t)_{t\geq 0}$ should satisfy
  \begin{equation}\label{limit-a-priori}
  \E \bigg[ \sup_{0\leq t\leq T} \bigg| \int_0^t \varphi(\xi_s) \,\d s \bigg|^p \bigg] \lesssim T^{p/2} \big\|c_{2p}^{\mathcal N}(-\L_0)^{-1/2}\varphi \big\|_{L^2(\mu)}^{p}
  \end{equation}
for all $p\geq 1$ and all cylinder functions $\varphi\in \mathcal{FC}$. Following \cite[Definition 3]{GT}, we introduce the notion of cylinder martingale problem with respect to the operator $\L$.

\begin{definition}\label{def-cylinder-martingale}
A process $(\xi_t)_{t\geq 0}$ with trajectories in $C\big(\R_+; C^\infty(\T^2)' \big)$ solves the cylinder martingale problem for the operator $\L$ with initial distribution $\nu$ if $\xi_0 \sim \nu$ and the following conditions hold.
\begin{itemize}
\item[\rm(i)] $(\xi_t)_{t\geq 0}$ is incompressible;
\item[\rm(ii)] the It\^o trick works: for all cylinder functions $\varphi$ and all $p\geq 1$, \eqref{limit-a-priori} holds;
\item[\rm(iii)] for any $\varphi\in \mathcal{FC}$, the process
  \begin{equation}\label{def-cylinder-martingale.1}
  M^\varphi_t= \varphi(\xi_t)- \varphi(\xi_0) - \int_0^t (\L \varphi)(\xi_s) \,\d s
  \end{equation}
is a continuous martingale with quadratic variation $\<M^\varphi \>_t = \int_0^t \mathcal E(\varphi)(\xi_s) \,\d s$.
\end{itemize}
\end{definition}

By Propositions \ref{prop-martingale-solution} and \ref{prop-Ito-trick}, the weak limit process $\tilde\xi_\cdot$ obtained in Section 3 is a solution to the cylinder martingale problem for $\L$; moreover, it is stationary with the marginal law $\mu$. Similarly to \cite[Theorem 1]{GT}, we can prove

\begin{theorem}
Let $\eta\in L^2(\mu)$ and for each $m\geq 1$, let $(\xi^m_t)_{t\geq 0}$ be the solution to \eqref{2D-NSEs-Galerkin} with $\xi^m_0 \sim \eta\,\d\mu$. Then the family $(\xi^m)_{m\geq 1}$ is tight in $C\big(\R_+; C^\infty(\T^2)' \big)$ and any weak limit $\xi$ solves the cylinder martingale problem for $\mathcal L$ with initial distribution $\eta\,\d\mu$ according to Definition \ref{def-cylinder-martingale} and we have
  $$\E[|\varphi(\xi_t)- \varphi(\xi_s)|^p]\lesssim |t-s|^{p/2} \big\| c_{4p}^{\mathcal N}(-\L_0)^{1/2} \varphi \big\|^p $$
for all $p\geq 2$ and $\varphi\in \mathcal{FC}$.
\end{theorem}

\begin{proof}
The proof is similar to that of \cite[Theorem 1]{GT}.
\end{proof}

Next, as shown in Theorem \ref{thm-Kolmogorov} below, we can find a dense domain $\mathcal D(\L)\subset \mathcal H$ for $\L$ such that the Kolmogorov equation
  $$\partial_t \varphi(t) = \L\varphi(t)$$
has a unique solution in $C(\R_+; \mathcal D(\L)) \cap C^1(\R_+; \mathcal H)$ for any initial condition in a dense set $\mathcal U\subset \mathcal H$. Then by duality arguments, one can prove Theorem \ref{thm-uniqueness}.

\begin{proof}[Proof of Theorem \ref{thm-uniqueness}] It is sufficient to prove that for $\varphi\in  \mathcal D(\L)$,
$\varphi(\xi_t)-\varphi(\xi_0)-\int_0^t\L\varphi(\xi_s)\d s$ is a martingale. The rest are the same as those of \cite[Theorem 2]{GT}. Define $\varphi^M$ as the projection of $\varphi$ onto the chaos components of order $\leq M$, and in each chaos we project onto the Fourier modes $|k|_\infty\leq M$. Then $\varphi^M\in \mathcal{FC}$. By construction
$$\lim_{M\to\infty} \E\big|\varphi(\xi_t)-\varphi(\xi_0)-\varphi^M(\xi_t)+\varphi^M(\xi_0) \big|=0.$$
We only have to show that
$$\lim_{M\to\infty}\E\left|\int_0^t\L\varphi^M(\xi_s)\, \d s-\int_0^t\L\varphi(\xi_s)\, \d s\right|=0.$$
By using \eqref{limit-a-priori} and the fact that $c_2=1$, we have
  $$\aligned
  \E\left|\int_0^t\L\varphi^M(\xi_s)\, \d s-\int_0^t\L\varphi(\xi_s)\, \d s\right| & \lesssim \|(-\L_0)^{-1/2}\L(\varphi^M-\varphi)\|\\
  & \lesssim\|(-\L_0)^{1/2}(\varphi^M-\varphi)\|+\|(-\L_0)^{-1/2}\mathcal{G}(\varphi^M-\varphi)\| \\
  & \lesssim \|(1+\mathcal{N})^{\alpha(0)} (-\L_0)^{1/2}(\varphi^M-\varphi)\|,
  \endaligned$$
where in the last step we have used Propositions \ref{prop-control-struc} and \ref{prop-control-struc-2} below (these results hold uniformly in $m$ and hold also for the limit operator). The dominated convergence theorem implies the result.
\end{proof}

\subsection{The Kolmogorov equation}

We want to determine a suitable domain for $\L$ and solve the Kolmogorov backward equation
  $$\partial_t\varphi(t)= \L\varphi(t)$$
for a sufficiently large class of initial data.

\subsubsection{A priori estimates}

The next result is similar to \cite[Lemma 10]{GT}.

\begin{lemma}\label{lem-apriori-Kol}
For any $\varphi_0\in \mathcal V:= (1+\mathcal N)^{-2} (-\L_0)^{-1} (\mathcal H)$, there exists a solution
  $$\varphi^m\in C(\R_+,\mathcal D(\L^m)) \cap C^1(\R_+, \mathcal H)$$
to the backward Kolmogorov equation
  $$\partial_t \varphi^m(t)= \L^m \varphi^m(t)$$
with $\varphi^m(0)= \varphi_0$, satisfying the estimates
  $$\|(1+\mathcal N)^p \varphi^m(t)\|^2 + \int_0^t e^{-C(t-s)} \|(1+\mathcal N)^p (-\L_0)^{1/2} \varphi^m(s)\|^2\,\d s \lesssim_p e^{Ct} \|(1+\mathcal N)^p \varphi_0\|^2,$$
  $$\|(1+\mathcal N)^p \L_0 \varphi^m(t)\| \lesssim_{t,m,p} \|(1+\mathcal N)^{p+1} \L_0 \varphi_0 \|$$
for all $t\geq 0$ and $p\geq 1$.
\end{lemma}

\begin{proof}
The proof is similar to that of Lemma \cite[Lemma 10]{GT} by using a suitable approximation to $\mathcal{G}^m$: $\mathcal{G}^{m,h} =J_h \mathcal{G}^m J_h$ with $J_h=e^{-h(\mathcal{N}-\L_0)}$. We only give a sketch of the proof to save space.  Let $\varphi^{m,h}$ be the solution to $\partial_t\varphi^{m,h}=(\L_0+\mathcal{G}^{m,h})\varphi^{m,h}$.  To obtain the desired estimate we calculate
\begin{equation}\label{zzz}\partial_t\frac{1}{2}\|(1+\mathcal{N})^p\varphi^{m,h}(t)\|^2
=\<(1+\mathcal{N})^{2p}\varphi^{m,h}(t),(\L_0+\mathcal{G}^{m,h})\varphi^{m,h}(t)\>.\end{equation}
The second identity in Lemma \ref{lem-drift} leads to
  $$\mathcal G^{m,h}_+ w(\mathcal N+1)= w(\mathcal N)\mathcal G^{m,h}_+. $$
Using the last equality in Lemma \ref{lem-1} and the uniform estimates in Lemma \ref{lem-apriori}, we can prove
  $$|\<(1+\mathcal{N})^{2p}\varphi^{m,h}(t),\mathcal{G}^{m,h}\varphi^{m,h}(t)\>|\leq C_\delta\|(1+\mathcal{N})^{p}\varphi^{m,h}(t)\|^2+\delta\|(1+\mathcal{N})^{p}(-\L_0)^{1/2}\varphi^{m,h}(t)\|^2.$$
  Substituting this into \eqref{zzz} and taking integration, we deduce the first uniform in $m, h$ estimate for $\varphi^{m,h}$. The rest is the same as  the proof of Lemma \cite[Lemma 10]{GT}.
We only mention one difference: we have, for the Galerkin projectors,
  $$ -\L_0 \Pi_m \lesssim |m|^2 (1+\mathcal N) \Pi_m,$$
and thus
  $$\aligned
  \|(1+\mathcal N)^p \L_0 \mathcal G^{m,h}\varphi^{m,h}(s)\| &\leq C(m) \|(1+\mathcal N)^{p+1} \mathcal G^{m,h}\varphi^{m,h}(s)\| \\
  &\lesssim_m \|(1+\mathcal N)^{p+1} (-\L_0)^{1/2} \varphi^{m,h}(s)\|,
  \endaligned $$
where the second step follows from the non-uniform bounds in Lemma \ref{lem-apriori}.
\end{proof}

Next, let $T^m$ be the semigroup generated by the Galerkin approximation $\xi^m$. By Lemma \ref{lem-apriori-Kol} and the approximation to the functions in $\mathcal{D}(\L^m)$ from cylinder functions,  we can prove the following result. We omit the proof; for more details we refer to \cite[Lemma 11]{GT}.

\begin{lemma}
$(\L^m, \mathcal D(\L^m))$ is closable and its closure is the generator $\hat\L^m$. In particular, if $\varphi\in \mathcal V$, then $\varphi^m(t) =T^m_t \varphi$ solves
  $$\partial_t \varphi^m(t) = \L^m \varphi^m(t),$$
and it holds
  $$\L^m T^m_t\varphi= T^m_t\L^m \varphi. $$
\end{lemma}

The identity $\L^m T^m_t\varphi= T^m_t\L^m \varphi$ enables us to get better estimates uniform in $m$; see the proof of \cite[Corollary 1]{GT}.

\begin{corollary}\label{cor-evolution}
For all $\varphi_0\in \mathcal V$ and for all $\alpha\geq 1$, we have
  $$\|(1+\mathcal N)^\alpha \partial_t\varphi^m(t)\|^2 = \|(1+\mathcal N)^\alpha \L^m\varphi^m(t)\|^2 \lesssim e^{Ct} \|(1+\mathcal N)^\alpha \L^m\varphi_0 \|^2,$$
and
  $$\|(1+\mathcal N)^\alpha (-\L_0)^{1/2} \varphi^m(t)\|^2\lesssim t e^{Ct} \|(1+\mathcal N)^\alpha \L^m \varphi_0\|^2 + \|(1+\mathcal N)^\alpha (-\L_0)^{1/2} \varphi_0 \|^2. $$
\end{corollary}

\subsubsection{Controlled structures} \label{sec-control-struc}

Next we deal with the limit operator $\L$ and define a domain $\mathcal D(\L)$. We will decompose the term $\mathcal G$ in $\L$ by means of a cut-off function $\mathcal M= M(\mathcal N)$. First, for the approximating operators, we define
  $$\mathcal G^m= {\bf 1}_{|\L_0|\geq \mathcal M} \mathcal G^m + {\bf 1}_{|\L_0|< \mathcal M} \mathcal G^m=: \mathcal G^{m,\succ} + \mathcal G^{m,\prec}. $$

\begin{proposition}\label{prop-control-struc}
Let $w$ be a weight, $L\geq 1$, $M(n)= L(n+1)^{4/\eps}$. Then we have
  \begin{equation}\label{prop-control-struc.1}
  \|w(\mathcal N) (-\L_0)^{-1/2} \mathcal G^{m,\succ}\psi \| \lesssim |w| L^{-\eps/4} \|w(\mathcal N) (-\L_0)^{1/2} \psi \|.
  \end{equation}
Hence, there exists $L_0= L_0(|w|,\eps)$ such that, for all $L\geq L_0$ and all $\varphi^\sharp\in w(\mathcal N)^{-1} (-\L_0)^{-1/2} (\mathcal H)$, there is a unique $\varphi^m= \mathcal K^m\varphi^\sharp$ such that
  $$\varphi^m= (-\L_0)^{-1} \mathcal G^{m,\succ} \varphi^m + \varphi^\sharp \in w(\mathcal N)^{-1} (-\L_0)^{-1/2} (\mathcal H),$$
satisfying the bound
  \begin{equation}\label{prop-control-struc.1.5}
  \|w(\mathcal N) (-\L_0)^{1/2} \mathcal K^m\varphi^\sharp \| + |w|^{-1} L^{\eps/4} \|w(\mathcal N) (-\L_0)^{1/2} (\mathcal K^m\varphi^\sharp- \varphi^\sharp) \| \lesssim \|w(\mathcal N) (-\L_0)^{1/2} \varphi^\sharp \|.
  \end{equation}
All the estimates are uniform in $m$ and true in the limit $m\to \infty$. We shall denote $\mathcal K= \mathcal K^\infty$.
\end{proposition}

\begin{proof}
We follow the ideas of \cite[Lemma 2.14]{GP-18}, \cite[Lemma 12]{GT} and start with the estimate on $\mathcal G^{m,\succ}_+$:
  $$\aligned
  \|w(\mathcal N) (-\L_0)^{-1/2} \mathcal G^{m,\succ}_+\psi \|&= \|w(\mathcal N) (-\L_0)^{-1/2} {\bf 1}_{|\L_0|\geq M(\mathcal N)} \mathcal G^m_+\psi \| \\
  &\lesssim \|w(\mathcal N) M(\mathcal N)^{-\eps/4} (-\L_0)^{-1/2+\eps/4} {\bf 1}_{|\L_0|\geq M(\mathcal N)} \mathcal G^m_+\psi \| \\
  &\leq \|w(\mathcal N) M(\mathcal N)^{-\eps/4} (-\L_0)^{-1/2+\eps/4} \mathcal G^m_+\psi \|\\
  &\lesssim \|w(\mathcal N+1) M(\mathcal N+1)^{-\eps/4} (\mathcal N+1) (-\L_0)^{1/2-\eps/4} \psi \|,
  \endaligned $$
where in the last step we have used \eqref{lem-apriori.1}. In the same way,
  $$\aligned
  \|w(\mathcal N) (-\L_0)^{-1/2} \mathcal G^{m,\succ}_-\psi \|&= \|w(\mathcal N) (-\L_0)^{-1/2} {\bf 1}_{|\L_0|\geq M(\mathcal N)} \mathcal G^m_-\psi \| \\
  &\lesssim \|w(\mathcal N) M(\mathcal N)^{-\eps/4} (-\L_0)^{-1/2+\eps/4} {\bf 1}_{|\L_0|\geq M(\mathcal N)} \mathcal G^m_-\psi \| \\
  &\leq \|w(\mathcal N) M(\mathcal N)^{-\eps/4} (-\L_0)^{-1/2+\eps/4} \mathcal G^m_-\psi \|\\
  &\lesssim \|w(\mathcal N-1) M(\mathcal N-1)^{-\eps/4} \mathcal N (-\L_0)^{1/2} \psi \|,
  \endaligned $$
where the last step follows from \eqref{lem-apriori.2}. Using the definition of $M(n)$ and the fact that $w(n\pm 1) \leq |w| w(n)$, we obtain
  \begin{equation}\label{prop-control-struc.2}
  \aligned
  \|w(\mathcal N) (-\L_0)^{-1/2} \mathcal G^{m,\succ} \psi \|&\lesssim |w| L^{-\eps/4} \|w(\mathcal N) (-\L_0)^{1/2} \psi \|.
  \endaligned
  \end{equation}

Let $\varphi^\sharp\in w(\mathcal N)^{-1} (-\L_0)^{-1/2} (\mathcal H)$ and define the map
  $$\Psi^m: w(\mathcal N)^{-1} (-\L_0)^{-1/2} (\mathcal H) \to w(\mathcal N)^{-1} (-\L_0)^{-1/2} (\mathcal H), $$
  $$\psi\mapsto \Psi^m(\psi):= (-\L_0)^{-1}\mathcal G^{m,\succ} \psi + \varphi^\sharp. $$
Then by \eqref{prop-control-struc.2}, we have
  $$\aligned
  \|w(\mathcal N) (-\L_0)^{1/2} \Psi^m(\psi) \| & \leq \|w(\mathcal N) (-\L_0)^{-1/2} \mathcal G^{m,\succ} \psi \| + \|w(\mathcal N) (-\L_0)^{1/2} \varphi^\sharp \| \\
  &\leq C|w| L^{-\eps/4} \|w(\mathcal N) (-\L_0)^{1/2} \psi \| + \|w(\mathcal N) (-\L_0)^{1/2} \varphi^\sharp \|
  \endaligned $$
for some constant $C>0$.
From this we conclude that, for $L$ big enough, the map $\Psi^m$ is a contraction leaving the ball of radius $2\|w(\mathcal N) (-\L_0)^{1/2} \varphi^\sharp \|$ invariant. Thus it has a unique fixed point $\mathcal K^m\varphi^\sharp$ such that
  \begin{equation}\label{prop-control-struc.3}
  \|w(\mathcal N) (-\L_0)^{1/2} \mathcal K^m\varphi^\sharp \|\leq 2 \|w(\mathcal N) (-\L_0)^{1/2} \varphi^\sharp \|.
  \end{equation}
Combining this with \eqref{prop-control-struc.2} yields
  $$\aligned
  \|w(\mathcal N) (-\L_0)^{1/2} (\mathcal K^m\varphi^\sharp- \varphi^\sharp) \| &= \|w(\mathcal N) (-\L_0)^{-1/2} \mathcal G^{m,\succ} \mathcal K^m\varphi^\sharp \| \\
  &\lesssim |w| L^{-\eps/4} \|w(\mathcal N) (-\L_0)^{1/2}\mathcal K^m\varphi^\sharp \| \\
  &\lesssim \|w(\mathcal N) (-\L_0)^{1/2} \varphi^\sharp \|.
  \endaligned $$
The proof is complete.
\end{proof}

In the above result, the coefficient $L$ in the cut-off $M(n)$ depends on $|w|$ of the weight $w$. In the following we will only make use of polynomial weights: $w(n)= (n+1)^\alpha$ with $|\alpha|\leq K$ for some fixed $K$; then $|w|$ is uniformly bounded and thus we can choose a cut-off which is adapted to all those weights. This will be fixed once and for all and will not be mentioned again.

\begin{proposition}\label{prop-control-struc-2}
Let $w$ be a polynomial weight, $\gamma\geq 0$ and
  $$\alpha(\gamma)= \frac4\eps \bigg(\gamma+ \frac12 \bigg)+1. $$
Let
  $$\varphi^\sharp \in w(\mathcal N)^{-1} (-\L_0)^{-1}(\mathcal H) \cap w(\mathcal N)^{-1} (\mathcal N+1)^{-\alpha(\gamma)} (-\L_0)^{-1/2}(\mathcal H),$$
and set $\varphi^m= \mathcal K^m\varphi^\sharp$. Then $\L^m \varphi^m$ is a well-defined operator and the following bound holds:
  $$ \|w(\mathcal N) (-\L_0)^\gamma \mathcal G^{m,\prec}\varphi^m \| \lesssim \|w(\mathcal N) (\mathcal N+1)^{\alpha(\gamma)} (-\L_0)^{1/2} \varphi^\sharp \|.$$
\end{proposition}

\begin{proof}
By Proposition \ref{prop-control-struc}, we have
  $$\varphi^m = (-\L_0)^{-1} \mathcal G^{m,\succ} \varphi^m+ \varphi^\sharp,$$
then,
  $$ \L^m \varphi^m= \L_0 \varphi^m+ \mathcal G^m \varphi^m= - \mathcal G^{m,\succ} \varphi^m + \L_0\varphi^\sharp+ \mathcal G^m \varphi^m= \L_0\varphi^\sharp + \mathcal G^{m,\prec} \varphi^m.$$
Hence it is sufficient to estimate $\mathcal G^{m,\prec} \varphi^m$.

First, we have
  $$\aligned
  \| w(\mathcal N) (-\L_0)^\gamma \mathcal G^{m,\prec}_+ \varphi^m \| &= \|w(\mathcal N) (-\L_0)^{\gamma +1/2} (-\L_0)^{-1/2} {\bf 1}_{|\L_0|< M(\mathcal N)} \mathcal G^{m}_+ \varphi^m\| \\
  &\leq \|w(\mathcal N) M(\mathcal N)^{\gamma +1/2} (-\L_0)^{-1/2} \mathcal G^{m}_+ \varphi^m\| \\
  &\lesssim \|w(\mathcal N+1) M(\mathcal N+1)^{\gamma +1/2} (\mathcal N+1) (-\L_0)^{(1-\eps)/2} \varphi^m\|,
  \endaligned $$
where in the last step we have used \eqref{lem-apriori.1}. Similarly, by \eqref{lem-apriori.2},
  $$\aligned
  \| w(\mathcal N) (-\L_0)^\gamma \mathcal G^{m,\prec}_- \varphi^m \| &= \|w(\mathcal N) (-\L_0)^{\gamma +1/2-\eps/4} (-\L_0)^{-1/2 +\eps/4} {\bf 1}_{|\L_0|< M(\mathcal N)} \mathcal G^{m}_- \varphi^m\| \\
  &\leq \|w(\mathcal N) M(\mathcal N)^{\gamma +1/2 -\eps/4} (-\L_0)^{-1/2+\eps/4} \mathcal G^{m}_- \varphi^m\| \\
  &\lesssim \|w(\mathcal N-1) M(\mathcal N-1)^{\gamma +1/2 -\eps/4} \mathcal N (-\L_0)^{1/2} \varphi^m\|.
  \endaligned $$
Recall the function $M(n)$ defined in Proposition \ref{prop-control-struc}. Summarizing the above two estimates and  using the definitions of $|w|$, we obtain
  $$\aligned
  \| w(\mathcal N) (-\L_0)^\gamma \mathcal G^{m,\prec} \varphi^m \| & \lesssim |w| L^{\gamma +1/2} \|w(\mathcal N) (\mathcal N+1)^{\alpha(\gamma)} (-\L_0)^{1/2} \varphi^m\|
  \endaligned $$
which, combined with \eqref{prop-control-struc.3}, completes the proof.
\end{proof}

\subsubsection{Limit operator and its domain}

Based on the estimates in the previous subsection, we will find a suitable domain for the limit operator $\L$. The following lemma is the analogue of \cite[Lemma 2.19]{GP-18} and \cite[Lemma 13]{GT}.

\begin{lemma}\label{lem-domain}
Let $w$ be a weight and $M(n)$ the cut-off function in Proposition \ref{prop-control-struc}; take $\gamma=0$ in Proposition \ref{prop-control-struc-2}. Set
  $$\mathcal D_w(\L)= \big\{ \mathcal K\varphi^\sharp: \varphi^\sharp\in w(\mathcal N)^{-1} (-\L_0)^{-1} (\mathcal H) \cap w(\mathcal N)^{-1} (\mathcal N+1)^{-\alpha(0)} (-\L_0)^{-1/2} (\mathcal H) \big\}. $$
Then $\mathcal D_w(\L)$ is dense in $w(\mathcal N)^{-1} (\mathcal H)$. We simply write $\mathcal D(\L)$ if $w\equiv 1$.
\end{lemma}

\begin{proof}
We know that $w(\mathcal N)^{-1} (-\L_0)^{-1} (\mathcal H) \cap w(\mathcal N)^{-1} (\mathcal N+1)^{-\alpha(0)} (-\L_0)^{-1/2} (\mathcal H)$ is dense in $w(\mathcal N)^{-1} (\mathcal H)$. Thus it is enough to show that, for any $\psi\in w(\mathcal N)^{-1} (-\L_0)^{-1} (\mathcal H) \cap w(\mathcal N)^{-1} (\mathcal N+1)^{-\alpha(0)} (-\L_0)^{-1/2} (\mathcal H)$ and for all $\nu\geq 1$, there exists a $\varphi^\nu\in \mathcal D_w(\L)$ such that
  \begin{eqnarray}
  \|w(\mathcal N) (-\L_0)^{1/2} \varphi^\nu \| &\lesssim& \|w(\mathcal N) (-\L_0)^{1/2} \psi \|, \label{lem-domain.1} \\
  \|w(\mathcal N) (-\L_0)^{1/2} (\varphi^\nu-\psi) \| &\lesssim& \nu^{-\eps/4} \|w(\mathcal N) (-\L_0)^{1/2} \psi \|, \label{lem-domain.2}
  \end{eqnarray}
and moreover,
  \begin{equation} \label{lem-domain.3}
  \|w(\mathcal N) \L \varphi^\nu \| \lesssim \nu^{1/2 -\eps/4} \big(\|w(\mathcal N) \L_0 \psi \| + \| w(\mathcal N) (\mathcal N+1)^{\alpha(0)} (-\L_0)^{1/2} \psi\| \big).
  \end{equation}

First, as the results in Proposition \ref{prop-control-struc} hold uniformly in $m\geq 1$, there exists $\varphi^\nu\in w(\mathcal N)^{-1} (\mathcal H)$ such that
  $$\varphi^\nu = {\bf 1}_{\nu M(\mathcal N)\leq |\L_0|} (-\L_0)^{-1}\mathcal G\varphi^\nu +\psi $$
and it enjoys the estimates \eqref{lem-domain.1} and \eqref{lem-domain.2}. It remains to show that $\varphi^\nu\in \mathcal D_w(\L)$ and \eqref{lem-domain.3} holds. We first give some heuristic discussions. Note that
  $$\varphi^\nu = (-\L_0)^{-1}\mathcal G^\succ \varphi^\nu +\varphi^{\nu,\sharp} ,$$
where
  $$\varphi^{\nu,\sharp}= \psi - {\bf 1}_{M(\mathcal N) \leq |\L_0|<\nu M(\mathcal N)} (-\L_0)^{-1}\mathcal G \varphi^\nu. $$
This means that $\varphi^\nu = \mathcal K \varphi^{\nu,\sharp}$ and thus, if we can show that $\varphi^{\nu,\sharp}\in w(\mathcal N)^{-1} (-\L_0)^{-1} (\mathcal H) \cap w(\mathcal N)^{-1} (\mathcal N+1)^{-\alpha(0)} (-\L_0)^{-1/2} (\mathcal H)$, then $\varphi^\nu\in \mathcal D_w(\L)$. Moreover, since
  $$\aligned
  \L \varphi^\nu &= \L_0 \varphi^\nu + \mathcal G\varphi^\nu= -\mathcal G^\succ \varphi^\nu+ \L_0 \varphi^{\nu,\sharp} + \mathcal G\varphi^\nu  = \L_0 \varphi^{\nu,\sharp} + \mathcal G^\prec \varphi^\nu,
  \endaligned $$
we have
  $$w(\mathcal N) \L \varphi^\nu = w(\mathcal N) \L_0 \varphi^{\nu,\sharp} + w(\mathcal N) \mathcal G^\prec \varphi^\nu .$$
Applying Proposition \ref{prop-control-struc-2} with $\gamma=0$, we have
  $$\|w(\mathcal N) \mathcal G^\prec \varphi^\nu\| \lesssim \|w(\mathcal N) (\mathcal N+1)^{\alpha(0)} (-\L_0)^{1/2} \varphi^{\nu,\sharp} \|. $$
Thus, to prove \eqref{lem-domain.3}, it suffices to estimate $\|w(\mathcal N)\L_0 \varphi^{\nu,\sharp} \|$ and $\|w(\mathcal N) (\mathcal N+1)^{\alpha(0)} (-\L_0)^{1/2} \varphi^{\nu,\sharp} \|$. Summarizing these discussions, we need to prove
  $$\varphi^{\nu,\sharp}\in w(\mathcal N)^{-1} (-\L_0)^{-1} (\mathcal H) \cap w(\mathcal N)^{-1} (\mathcal N+1)^{-\alpha(0)} (-\L_0)^{-1/2} (\mathcal H) $$
and estimate the norms.

We denote by
  $$ \psi^\nu= {\bf 1}_{M(\mathcal N) \leq |\L_0|<\nu M(\mathcal N)} (-\L_0)^{-1}\mathcal G \varphi^\nu, $$
and hence $\varphi^{\nu,\sharp}= \psi - \psi^\nu$. It is enough to show that both $\psi$ and $\psi^\nu$ enjoy the above properties. By assumption, $\psi$ satisfies the desired bounds, thus it remains to prove
  \begin{eqnarray}
  \| w(\mathcal N) \L_0 \psi^\nu \| &\lesssim& \nu^{1/2-\eps/4} \| w(\mathcal N) (\mathcal N+1)^{\alpha(0)} (-\L_0)^{1/2} \psi\|, \label{lem-domain.4} \\
  \| w(\mathcal N) (\mathcal N+1)^{\alpha(0)} (-\L_0)^{1/2} \psi^\nu \| &\lesssim & \| w(\mathcal N) (\mathcal N+1)^{\alpha(0)} (-\L_0)^{1/2} \psi\|. \label{lem-domain.5}
  \end{eqnarray}
Note that $w(\mathcal N) \L_0 \psi^\nu= -w(\mathcal N) {\bf 1}_{M(\mathcal N) \leq |\L_0|<\nu M(\mathcal N)} \mathcal G \varphi^\nu$, hence \eqref{lem-domain.4} can be proved by using the uniform estimates in Lemma \ref{lem-apriori} (these estimates hold also for $m=+\infty$). First, as in the proof of Proposition \ref{prop-control-struc}, we have, for the operator $\mathcal G_+$,
  $$\aligned
  &\, \|w(\mathcal N) {\bf 1}_{M(\mathcal N) \leq |\L_0|<\nu M(\mathcal N)} \mathcal G_+ \varphi^\nu \| \\
  = &\, \|w(\mathcal N) (-\L_0)^{1/2 -\eps/4 -1/2 +\eps/4 } {\bf 1}_{M(\mathcal N) \leq |\L_0|<\nu M(\mathcal N)} \mathcal G_+ \varphi^\nu \| \\
  \leq &\, \nu^{1/2 -\eps/4} \|w(\mathcal N) M(\mathcal N)^{1/2 -\eps/4} (-\L_0)^{-1/2 +\eps/4 } \mathcal G_+ \varphi^\nu \| \\
  \lesssim &\, \nu^{1/2 -\eps/4} \|w(\mathcal N+1) M(\mathcal N+1)^{1/2 -\eps/4} (\mathcal N+1) (-\L_0)^{1/2 -\eps/4 } \varphi^\nu \|,
  \endaligned $$
where in the last step we have used \eqref{lem-apriori.1} (with $m=+\infty$). In the same way, for the operator $\mathcal G_-$, by \eqref{lem-apriori.2}
  $$\aligned
  &\, \|w(\mathcal N) {\bf 1}_{M(\mathcal N) \leq |\L_0|<\nu M(\mathcal N)} \mathcal G_- \varphi^\nu \| \\
  = &\, \|w(\mathcal N) (-\L_0)^{1/2 -\eps/4 -1/2 +\eps/4 } {\bf 1}_{M(\mathcal N) \leq |\L_0|<\nu M(\mathcal N)} \mathcal G_- \varphi^\nu \| \\
  \leq &\, \nu^{1/2 -\eps/4} \|w(\mathcal N) M(\mathcal N)^{1/2 -\eps/4} (-\L_0)^{-1/2 +\eps/4 } \mathcal G_- \varphi^\nu \| \\
  \lesssim &\, \nu^{1/2 -\eps/4} \|w(\mathcal N-1) M(\mathcal N-1)^{1/2 -\eps/4} \mathcal N (-\L_0)^{1/2 } \varphi^\nu \|.
  \endaligned $$
Combining these results and using the definition of $M(n)$ in Proposition \ref{prop-control-struc}, we obtain \eqref{lem-domain.4} as follows:
  $$\aligned
  \| w(\mathcal N) \L_0 \psi^\nu \| &= \|w(\mathcal N) {\bf 1}_{M(\mathcal N) \leq |\L_0|<\nu M(\mathcal N)} \mathcal G \varphi^\nu \| \\
  &\lesssim \nu^{1/2 -\eps/4} |w| L^{1/2-\eps/4} \|w(\mathcal N) (\mathcal N+1)^{\alpha(0)} (-\L_0)^{1/2 } \varphi^\nu \| \\
  &\lesssim \nu^{1/2 -\eps/4} \|w(\mathcal N) (\mathcal N+1)^{\alpha(0)} (-\L_0)^{1/2 } \psi \|,
  \endaligned $$
where in the last step we have used \eqref{lem-domain.1} with $w(\mathcal N) (\mathcal N+1)^{\alpha(0)}$ in place of $w(\mathcal N)$.

Finally, we prove \eqref{lem-domain.5}. By the definition of $\psi^\nu$,
  $$\aligned
  &\, \| w(\mathcal N) (\mathcal N+1)^{\alpha(0)} (-\L_0)^{1/2} \psi^\nu \| \\
  =&\, \| w(\mathcal N) (\mathcal N+1)^{\alpha(0)}{\bf 1}_{M(\mathcal N) \leq |\L_0|<\nu M(\mathcal N)} (-\L_0)^{-1/2} \mathcal G \varphi^\nu \| \\
  \leq &\, \| w(\mathcal N) (\mathcal N+1)^{\alpha(0)}{\bf 1}_{M(\mathcal N) \leq |\L_0|<\nu M(\mathcal N)} (-\L_0)^{-1/2} \mathcal G_+ \varphi^\nu \| \\
  & + \| w(\mathcal N) (\mathcal N+1)^{\alpha(0)}{\bf 1}_{M(\mathcal N) \leq |\L_0|<\nu M(\mathcal N)} (-\L_0)^{-1/2} \mathcal G_- \varphi^\nu \| .
  \endaligned $$
We have
  $$\aligned
  &\, \| w(\mathcal N) (\mathcal N+1)^{\alpha(0)}{\bf 1}_{M(\mathcal N) \leq |\L_0|<\nu M(\mathcal N)} (-\L_0)^{-1/2} \mathcal G_+ \varphi^\nu \| \\
  \leq &\, \| w(\mathcal N) (\mathcal N+1)^{\alpha(0)} M(\mathcal N)^{-\eps/4} (-\L_0)^{-1/2+\eps/4} \mathcal G_+ \varphi^\nu \|\\
  \lesssim &\, \| w(\mathcal N+1) (\mathcal N+2)^{\alpha(0)} M(\mathcal N+1)^{-\eps/4} (\mathcal N+1) (-\L_0)^{1/2-\eps/4} \varphi^\nu \| \\
  \lesssim &\, \| w(\mathcal N) (\mathcal N+1)^{\alpha(0)} (-\L_0)^{1/2} \varphi^\nu \|,
  \endaligned $$
where the second step follows from \eqref{lem-apriori.1} and the last step from the definition of $M(n)$. Similarly,
  $$\aligned
  &\, \| w(\mathcal N) (\mathcal N+1)^{\alpha(0)}{\bf 1}_{M(\mathcal N) \leq |\L_0|<\nu M(\mathcal N)} (-\L_0)^{-1/2} \mathcal G_- \varphi^\nu \| \\
  \leq &\, \| w(\mathcal N) (\mathcal N+1)^{\alpha(0)} M(\mathcal N)^{-\eps/4} (-\L_0)^{-1/2+\eps/4} \mathcal G_- \varphi^\nu \|\\
  \lesssim &\, \| w(\mathcal N-1) \mathcal N^{\alpha(0)} M(\mathcal N-1)^{-\eps/4} \mathcal N (-\L_0)^{1/2} \varphi^\nu \| \\
  \lesssim &\, \| w(\mathcal N) (\mathcal N+1)^{\alpha(0)} (-\L_0)^{1/2} \varphi^\nu \|,
  \endaligned $$
where we have used \eqref{lem-apriori.2} in the second step. Combining the above estimates we obtain
  $$\| w(\mathcal N) (\mathcal N+1)^{\alpha(0)} (-\L_0)^{1/2} \psi^\nu \| \lesssim \| w(\mathcal N) (\mathcal N+1)^{\alpha(0)} (-\L_0)^{1/2} \varphi^\nu \|,$$
which, together with \eqref{lem-domain.1}, yields \eqref{lem-domain.5}. The proof is complete.
\end{proof}

The following result gives the dissipativity of the limit operator $\L$ and it is the same as \cite[Lemma 2.22]{GP-18}, see also \cite[Lemma 14]{GT}.

\begin{corollary}\label{cor-dissipa}
For any $\varphi\in \mathcal D(\L)$, we have $\<\varphi, \L\varphi\>\leq 0$. In particular, the operator $(\L, \mathcal D(\L))$ is dissipative.
\end{corollary}

\begin{proof}
Since $\varphi \in \mathcal D(\L)$, we can find some $\varphi^\sharp \in (-\L_0)^{-1} (\mathcal H) \cap (\mathcal N+1)^{-\alpha(0)} (-\L_0)^{-1/2} (\mathcal H)$ such that $\varphi= \mathcal K \varphi^\sharp$. By Proposition \ref{prop-control-struc} (see in particular the last assertion), we have $\varphi\in (-\L_0)^{-1/2} (\mathcal H)$ and thus $\L_0\varphi \in (-\L_0)^{1/2} (\mathcal H)$. Moreover, by Proposition \ref{prop-control-struc-2}, $ \mathcal G\varphi\in (-\L_0)^{1/2} (\mathcal H)$. Next, taking $w(n)= (n+1)^{\alpha(0)}$ in Proposition \ref{prop-control-struc}, we deduce that $\varphi\in (\mathcal N+1)^{-1} (-\L_0)^{-1/2} (\mathcal H)$. These regularities are enough to proceed by approximation and prove that
  $$\<\varphi, \L\varphi\>= -\<\varphi, -\L_0 \varphi\> + \<\varphi, \mathcal G\varphi\>= - \|(-\L_0)^{1/2} \varphi \|^2\leq 0,$$
where we have used the antisymmetry of the form associated to $\mathcal G$, i.e. $\<\varphi, \mathcal G\varphi\>=0$.
\end{proof}

\subsubsection{Existence and uniqueness for the Kolmogorov equation}

In this part, we intend to study the Kolmogorov equation $\partial_t \varphi= \L\varphi$. First, let $\varphi^m= \varphi^m(t)$ be the solution to the equation:
  $$\partial_t \varphi^m= \L^m \varphi^m= \L_0\varphi^m + \mathcal G^m\varphi^m,\quad \varphi^m(0)= \varphi^m_0. $$
As in Section \ref{sec-control-struc}, set
  \begin{equation}\label{eq-identity.0}
  \varphi^{m,\sharp}= \varphi^m- (-\L_0)^{-1} \mathcal G^{m,\succ}\varphi^m .
    \end{equation}
Then $ \L_0\varphi^{m,\sharp} = \L_0 \varphi^m + \mathcal G^{m,\succ}\varphi^m $ and thus
  \begin{equation}\label{eq-identity}
  \aligned
  \partial_t \varphi^m =\L^m\varphi^m &= \L_0 \varphi^{m,\sharp}- \mathcal G^{m,\succ}\varphi^m + \mathcal G^m \varphi^m \\
  &= \L_0 \varphi^{m,\sharp}+ \mathcal G^{m,\prec} \varphi^m.
  \endaligned
  \end{equation}
Now we want to find an equation for $\varphi^{m,\sharp}$: by \eqref{eq-identity.0},
  $$\aligned
  \partial_t \varphi^{m,\sharp} -\L_0\varphi^{m,\sharp} &= \partial_t \varphi^m - (-\L_0)^{-1} \mathcal G^{m,\succ} \partial_t \varphi^m  -\L_0 \varphi^{m,\sharp} \\
  &= \mathcal G^{m,\prec} \varphi^m - (-\L_0)^{-1} \mathcal G^{m,\succ} \partial_t \varphi^m ,
  \endaligned $$
where the second step follows from the last equality in \eqref{eq-identity}. Using again these equalities, we obtain
  \begin{equation}\label{eq-identity.1}
  \aligned
  \partial_t \varphi^{m,\sharp} -\L_0\varphi^{m,\sharp}  &= \mathcal G^{m,\prec} \varphi^m - (-\L_0)^{-1} \mathcal G^{m,\succ} \big( \mathcal G^{m,\prec} \varphi^m +\L_0 \varphi^{m,\sharp} \big) =: \Phi^{m,\sharp}.
  \endaligned
  \end{equation}
We want to find a bound for each term of $\Phi^{m,\sharp}$ in terms of $\varphi^{m,\sharp}_0$. To this end, we first prove the following result.

\begin{lemma}\label{lem-auxili-1}
It holds that
  \begin{equation}\label{lem-auxili-1.1}
  \aligned
  &\, \|(1+\mathcal N)^p(-\L_0)^{1/2} \varphi^{m,\sharp}(t) \| \\
  \lesssim&\, \big(t e^{Ct}+1 \big)^{1/2} \big( \|(1+\mathcal N)^p \L_0 \varphi^{m,\sharp}_0 \| +  \|(1+\mathcal N)^{p+\alpha(0)} (-\L_0)^{1/2} \varphi^{m,\sharp}_0 \|\big).
  \endaligned
  \end{equation}
\end{lemma}

\begin{proof}
We follow the idea of \cite[Lemma 3.5]{GP-18} and \cite[Lemma 15]{GT}. By the definition of $\varphi^{m,\sharp}$,
  $$\aligned
  &\, \|(1+\mathcal N)^p(-\L_0)^{1/2} \varphi^{m,\sharp}(t) \|\\
  \leq&\, \|(1+\mathcal N)^p(-\L_0)^{1/2} \varphi^{m}(t) \| + \|(1+\mathcal N)^p(-\L_0)^{-1/2} \mathcal G^{m,\succ} \varphi^{m}(t) \| \\
  \lesssim &\, \|(1+\mathcal N)^p(-\L_0)^{1/2} \varphi^{m}(t) \|,
  \endaligned $$
where in the second step we have used \eqref{prop-control-struc.1}. Next, by Corollary \ref{cor-evolution},
  $$\aligned
  \|(1+\mathcal N)^p(-\L_0)^{1/2} \varphi^{m}(t) \| &\lesssim \big(te^{Ct} \big)^{1/2} \|(1+\mathcal N)^p \L^m \varphi^m_0 \| + \|(1+\mathcal N)^p(-\L_0)^{1/2} \varphi^{m}_0 \| \\
  &\lesssim \big(te^{Ct} \big)^{1/2} \|(1+\mathcal N)^p \L^m \varphi^m_0 \| +\|(1+\mathcal N)^p(-\L_0)^{1/2} \varphi^{m,\sharp}_0 \|,
  \endaligned $$
where we have used \eqref{prop-control-struc.1.5} for the second term on the right-hand side. It remains to estimate the first one: by \eqref{eq-identity},
  $$\|(1+\mathcal N)^p \L^m \varphi^m_0 \| \leq \|(1+\mathcal N)^p \L_0 \varphi^{m,\sharp}_0 \|+ \|(1+\mathcal N)^p \mathcal G^{m,\prec} \varphi^m_0 \|. $$
Using Proposition \ref{prop-control-struc-2} with $\gamma=0$, we get
  $$\|(1+\mathcal N)^p \mathcal G^{m,\prec} \varphi^m_0 \| \lesssim \|(1+\mathcal N)^{p+\alpha(0)} (-\L_0)^{1/2} \varphi^{m,\sharp}_0 \|. $$
Summarizing these estimates, we finish the proof.
\end{proof}

We can prove the following bound on $\Phi^{m,\sharp}$.

\begin{lemma}\label{lem-Phi-m}
For any $p>0$ and $\gamma\in (3/2 -\eps/2, 3/2-\eps/4)$, there exists $q=q(p,\gamma)$ and small $\delta\in (0,1)$ such that
  $$\aligned
  &\, \sup_{0\leq t\leq T} \|(1+\mathcal N)^p (-\L_0)^\gamma \Phi^{m,\sharp}(t) \| \\
  \lesssim_T &\ \|(1+\mathcal N)^q \L_0 \varphi^{m,\sharp}_0 \| + \|(1+\mathcal N)^{q+\alpha(0)} (-\L_0)^{1/2} \varphi^{m,\sharp}_0 \|\\
  &\, + \delta\sup_{0\leq t\leq T} \|(1+\mathcal N)^p (-\L_0)^{1+\gamma -\eps/8} \varphi^{m,\sharp}(t) \|.
  \endaligned $$
\end{lemma}

\begin{proof}
By the definition \eqref{eq-identity.1} of $\Phi^{m,\sharp}(t)$, we have
  \begin{equation}\label{lem-Phi-m.1}
  \aligned
  \|(1+\mathcal N)^p (-\L_0)^\gamma \Phi^{m,\sharp}(t) \|
  \leq &\  \|(1+\mathcal N)^p (-\L_0)^\gamma \mathcal G^{m,\prec} \varphi^m (t) \| \\
  & + \|(1+\mathcal N)^p (-\L_0)^{\gamma-1} \mathcal G^{m,\succ} \mathcal G^{m,\prec} \varphi^m (t) \| \\
  &+ \|(1+\mathcal N)^p (-\L_0)^{\gamma-1} \mathcal G^{m,\succ} \L_0\varphi^{m,\sharp} (t) \| .
  \endaligned
  \end{equation}
We start with the estimate of the first term on the right-hand side: by Proposition \ref{prop-control-struc-2} and Lemma \ref{lem-auxili-1},
  $$\aligned
  &\, \sup_{0\leq t\leq T}\|(1+\mathcal N)^p (-\L_0)^\gamma \mathcal G^{m,\prec} \varphi^m(t) \|\\
  \lesssim &\ \|(1+\mathcal N)^{p+\alpha(\gamma)} (-\L_0)^{1/2} \varphi^{m,\sharp}(t) \| \\
  \lesssim &\ \|(1+\mathcal N)^{p+\alpha(\gamma)} \L_0 \varphi^{m,\sharp}_0 \| + \|(1+\mathcal N)^{p+\alpha(\gamma)+\alpha(0)} (-\L_0)^{1/2} \varphi^{m,\sharp}_0 \|.
  \endaligned $$
Next, since $\gamma-1\in (1/2 -\eps/2, 1/2-\eps/4)$, the estimates \eqref{lem-apriori.1} and \eqref{lem-apriori.2} hold also for $\mathcal G^{m,\succ}_\pm$, thus
  $$\aligned
  \|(1+\mathcal N)^p (-\L_0)^{\gamma-1} \mathcal G^{m,\succ} \mathcal G^{m,\prec} \varphi^m (t) \| & \lesssim \|(1+\mathcal N)^{p+1} (-\L_0)^{\gamma -\eps/4} \mathcal G^{m,\prec} \varphi^m (t) \|\\
  &\lesssim \|(1+\mathcal N)^{p+1+ \alpha(\gamma -\eps/4)} (-\L_0)^{1/2} \varphi^{m,\sharp} (t) \|\\
  &= \|(1+\mathcal N)^{p+ \alpha(\gamma)} (-\L_0)^{1/2} \varphi^{m,\sharp} (t) \|,
  \endaligned $$
where in the last two steps we have used Proposition \ref{prop-control-struc-2} and the fact that $1+ \alpha(\gamma -\eps/4)= \alpha(\gamma)$, respectively. Moreover, by Lemma \ref{lem-auxili-1}, we obtain
  $$\aligned
  &\, \sup_{0\leq t\leq T} \|(1+\mathcal N)^p (-\L_0)^{\gamma-1} \mathcal G^{m,\succ} \mathcal G^{m,\prec} \varphi^m (t) \|\\
  \lesssim &\ \|(1+\mathcal N)^{p+ \alpha(\gamma)} \L_0 \varphi^{m,\sharp}_0 \| + \|(1+\mathcal N)^{p+ \alpha(\gamma)+\alpha(0)} (-\L_0)^{1/2} \varphi^{m,\sharp}_0 \|.
  \endaligned $$
Using again \eqref{lem-apriori.1} and \eqref{lem-apriori.2}, we have
  $$\aligned
  \|(1+\mathcal N)^p (-\L_0)^{\gamma-1} \mathcal G^{m,\succ} \L_0\varphi^{m,\sharp} (t) \| & \lesssim \|(1+\mathcal N)^{p+1} (-\L_0)^{\gamma -\eps/4} \L_0\varphi^{m,\sharp} (t) \| \\
  &= \|(1+\mathcal N)^{p+1} (-\L_0)^{\gamma +1 -\eps/4} \varphi^{m,\sharp} (t) \|.
  \endaligned $$
By interpolation, for any $\delta\in (0,1)$ there is a $C_\delta>0$ such that
  $$\aligned
  &\ \|(1+\mathcal N)^p (-\L_0)^{\gamma-1} \mathcal G^{m,\succ} \L_0\varphi^{m,\sharp} (t) \| \\
  \lesssim &\ C_\delta \|(1+\mathcal N)^q (-\L_0)^{1/2} \varphi^{m,\sharp} (t) \| + \delta \|(1+\mathcal N)^p (-\L_0)^{\gamma +1 -\eps/8} \varphi^{m,\sharp} (t) \|,
  \endaligned $$
where $q=q(p,\gamma)$ is some parameter. Again by Lemma \ref{lem-auxili-1},
  $$\aligned
  &\, \sup_{0\leq t\leq T} \|(1+\mathcal N)^p (-\L_0)^{\gamma-1} \mathcal G^{m,\succ} \L_0\varphi^{m,\sharp} (t) \| \\
  \lesssim &\ \|(1+\mathcal N)^q \L_0 \varphi^{m,\sharp}_0 \| + \|(1+\mathcal N)^{q+\alpha(0)} (-\L_0)^{1/2} \varphi^{m,\sharp}_0 \| \\
  &\, + \delta \sup_{0\leq t\leq T} \|(1+\mathcal N)^p (-\L_0)^{\gamma +1 -\eps/8} \varphi^{m,\sharp} (t) \|.
  \endaligned $$
Summarizing all the above estimates, we obtain the desired result.
\end{proof}

Now we are ready to prove the following key result.

\begin{proposition}\label{prop-key-result}
For any $p>0$ and $\gamma\in (3/2 -5\eps/8, 3/2-3\eps/8)$, there exists $q=q(p,\gamma)$ and $\kappa\in (0,1)$ such that
  $$\|(1+\mathcal N)^p (-\L_0)^{1+\gamma/2} (\varphi^{m,\sharp}(t)- \varphi^{m,\sharp}(s))\|\lesssim |t-s|^\kappa \|(1+\mathcal N)^q (-\L_0)^{1+\gamma} \varphi^{m,\sharp}_0 \|. $$
\end{proposition}

\begin{proof}
Recall the equation \eqref{eq-identity.1}; let $S_t$ be the semigroup associated to the operator $\L_0$, then we have
  $$\varphi^{m,\sharp}(t)= S_t \varphi^{m,\sharp}_0 + \int_0^t S_{t-s} \Phi^{m,\sharp}(s)\,\d s. $$
Therefore,
  \begin{equation}\label{prop-key-result.0}
  \aligned
  \|(1+\mathcal N)^p (-\L_0)^{1+\gamma} \varphi^{m,\sharp}(t) \| \leq&\ \|(1+\mathcal N)^p (-\L_0)^{1+\gamma} S_t \varphi^{m,\sharp}_0 \| \\
  &\, + \int_0^t \|(1+\mathcal N)^p (-\L_0)^{1+\gamma} S_{t-s} \Phi^{m,\sharp}(s) \|\,\d s .
  \endaligned
  \end{equation}
One has
  \begin{equation}\label{prop-key-result.1}
  \|(1+\mathcal N)^p (-\L_0)^{1+\gamma} S_t \varphi^{m,\sharp}_0 \| \lesssim \|(1+\mathcal N)^p (-\L_0)^{1+\gamma} \varphi^{m,\sharp}_0 \|.
  \end{equation}
Next, we need the Schauder estimate: for $\alpha,\beta>0$,
  $$\|(1+\mathcal N)^\alpha (-\L_0)^\beta S_t\psi \| \lesssim t^{-\beta} \|(1+\mathcal N)^\alpha\psi \|, \quad t>0. $$
Therefore,
  $$\aligned
  \|(1+\mathcal N)^p (-\L_0)^{1+\gamma} S_{t-s} \Phi^{m,\sharp}(s) \|
  =&\ \|(1+\mathcal N)^p (-\L_0)^{\gamma+\eps/8} (-\L_0)^{1-\eps/8} S_{t-s} \Phi^{m,\sharp}(s) \| \\
  \lesssim&\ (t-s)^{-1+\eps/8} \|(1+\mathcal N)^p (-\L_0)^{\gamma+\eps/8} \Phi^{m,\sharp}(s) \|.
  \endaligned $$
As $\gamma\in (3/2 -5\eps/8, 3/2-3\eps/8)$, one has $\gamma+\eps/8\in (3/2 -\eps/2, 3/2-\eps/4)$, thus, applying Lemma \ref{lem-Phi-m} with $\gamma+\eps/8$ in place of $\gamma$, we have (choose a new $q=q(p,\gamma)$)
  $$\aligned
  &\ \|(1+\mathcal N)^p (-\L_0)^{1+\gamma} S_{t-s} \Phi^{m,\sharp}(s) \| \\
  \lesssim &\ (t-s)^{-1+\eps/8} \Big( \|(1+\mathcal N)^q \L_0 \varphi^{m,\sharp}_0 \| + \|(1+\mathcal N)^{q+\alpha(0)} (-\L_0)^{1/2} \varphi^{m,\sharp}_0 \| \Big)\\
  &\, + \delta (t-s)^{-1+\eps/8}\sup_{0\leq t\leq T} \|(1+\mathcal N)^p (-\L_0)^{1+\gamma} \varphi^{m,\sharp}(t) \|.
  \endaligned $$
As a result,
  $$\aligned
  &\ \int_0^t \|(1+\mathcal N)^p (-\L_0)^{1+\gamma} S_{t-s} \Phi^{m,\sharp}(s) \|\,\d s \\
  \lesssim &\ T^{\eps/8} \Big( \|(1+\mathcal N)^q \L_0 \varphi^{m,\sharp}_0 \| + \|(1+\mathcal N)^{q+\alpha(0)} (-\L_0)^{1/2} \varphi^{m,\sharp}_0 \| \Big) \\
  &\, + \delta T^{\eps/8} \sup_{0\leq t\leq T} \|(1+\mathcal N)^p (-\L_0)^{1+\gamma} \varphi^{m,\sharp}(t) \|.
  \endaligned $$
Take $\delta$ small enough such that $\delta T^{\eps/8} \leq 1/2$. Combining this estimate with \eqref{prop-key-result.0} and \eqref{prop-key-result.1}, we obtain
  \begin{equation}\label{prop-key-result.2}
  \aligned
  &\, \sup_{0\leq t\leq T} \|(1+\mathcal N)^p (-\L_0)^{1+\gamma} \varphi^{m,\sharp}(t) \|\\
  \lesssim &\ \|(1+\mathcal N)^q (-\L_0)^{1+\gamma} \varphi^{m,\sharp}_0 \| + \|(1+\mathcal N)^{q+\alpha(0)} (-\L_0)^{1/2} \varphi^{m,\sharp}_0 \| \\
  \lesssim &\ \|(1+\mathcal N)^{\bar q} (-\L_0)^{1+\gamma} \varphi^{m,\sharp}_0 \|
  \endaligned
  \end{equation}
for some $\bar q= \bar q(p,\gamma)$.

Next, using again the equation \eqref{eq-identity.1}, we have
  $$\aligned
  &\, \sup_{0\leq t\leq T} \|(1+\mathcal N)^p (-\L_0)^{\gamma} \partial_t \varphi^{m,\sharp}(t) \|\\
  \lesssim_T &\, \sup_{0\leq t\leq T} \|(1+\mathcal N)^p (-\L_0)^{1+\gamma} \varphi^{m,\sharp}(t) \| + \sup_{0\leq t\leq T} \|(1+\mathcal N)^p (-\L_0)^{\gamma} \Phi^{m,\sharp}(t) \| \\
  \lesssim_T &\ \|(1+\mathcal N)^{\bar q} (-\L_0)^{1+\gamma} \varphi^{m,\sharp}_0 \|,
  \endaligned $$
where in the last step we have used Lemma \ref{lem-Phi-m} and the inequality \eqref{prop-key-result.2}. Combining this result with \eqref{prop-key-result.2}, we obtain the desired inequality by interpolation for the time variable.
\end{proof}

Now for $p>0$, we introduce the set
  $$\mathcal U_p= \bigcup_{\gamma\in \big(\frac32- \frac58 \eps, \frac32- \frac38\eps \big)} \mathcal K(1+\mathcal N)^{\bar q(p,\gamma)} (-\L_0)^{-1-\gamma} (\mathcal H) $$
and $\mathcal U=\bigcup_{p>\alpha(0)} \mathcal U_p$. Then we have

\begin{theorem}\label{thm-Kolmogorov}
Let $p>0$ and $\varphi_0\in \mathcal U_p$. Then there exists a solution
  $$\varphi\in \bigcup_{\delta>0} C\big(\R_+; (1+\mathcal N)^{-p+\delta} (-\L_0)^{-1} (\mathcal H) \big)$$
to the backward Kolmogorov equation $\partial_t\varphi= \L\varphi$ with initial condition $\varphi(0)= \varphi_0$. For $p>\alpha(0)$, we have $\varphi\in C(\R_+; \mathcal D(\L)) \cap C^1(\R_+; \mathcal H)$ and, by dissipativity of $\L$, this solution is unique.
\end{theorem}

\begin{proof}Using the above key estimate Proposition \ref{prop-key-result} we could obtain the relative compactness of $(\varphi^{m,\sharp})_m$ in $C(\R^+;(1+\mathcal{N})^{-p+\delta}(-\L_0)^{-1}\mathcal{H})$, which allows to find the solution.
For more details we refer to the proof of \cite[Theorem 3]{GT}.
\end{proof}

\section{Appendix: proofs of Lemmas \ref{lem-1}, \ref{lem-drift} and \ref{lem-apriori}}

First we prove Lemma \ref{lem-1} by following the ides of \cite[Lemma 2.4]{GP}.

\begin{proof}[Proof of Lemma \ref{lem-1}]
We have
  $$B_m(\xi)(x) = W_2\bigg(\int_s \big(K^m_s\hat\otimes \rho^m_s \big) \cdot \nabla_x \rho^m_x(s) \bigg) + \int_s \bigg(\int_x K^m_s(x) \rho^m_s(x) \bigg) \cdot \nabla_x \rho^m_x(s), $$
Note that $\nabla_x \rho^m_x(s)= - \nabla_s \rho^m_s(x)$, integrating by parts with respect to $s$ gives us
  $$\aligned
  \int_s \bigg(\int_x K^m_s(x) \rho^m_s(x) \bigg) \cdot \nabla_x \rho^m_x(s) &=-\int_s \bigg(\int_x K^m_s(x) \rho^m_s(x) \bigg) \cdot \nabla_s \rho^m_s(x)\\
  &= \int_s \div_s\bigg(\int_x K^m_s(x) \rho^m_s(x) \bigg) \rho^m_s(x) =0,
  \endaligned $$
where the last step is due to the fact that $\int_x K^m_s(x) \rho^m_s(x)$ is independent of $s\in \T^2$. Therefore,
  $$B_m(\xi)(x) = W_2\bigg(\int_s \big(K^m_s\hat\otimes \rho^m_s \big) \cdot \nabla_x \rho^m_x(s) \bigg). $$
Recall that $D_x W_n(\varphi_n)= n W_{n-1}(\varphi_n(x,\cdot))$, the contraction property of multiple Wiener-It\^o integrals leads to
  \begin{equation}\label{lem-1.1}
  \aligned
  &\, \int_x B_m(\xi)(x) D_x W_n(\varphi_n)\\
  =&\, \int_x W_2\bigg(\int_s \big(K^m_s\hat\otimes \rho^m_s \big) \cdot \nabla_x \rho^m_x(s) \bigg) n W_{n-1}(\varphi_n(x,\cdot))\\
  =&\, n\int_x W_{n+1} \bigg(\varphi_n(x,\cdot) \int_s \big(K^m_s\hat\otimes \rho^m_s \big) \cdot \nabla_x \rho^m_x(s) \bigg) \\
  &\, + 2n(n-1) \int_x W_{n-1} \bigg(\int_y \varphi_n(x,y,\cdot) \int_s \big(K^m_s \hat\otimes \rho^m_s \big)(y,\cdot) \cdot \nabla_x \rho^m_x(s) \bigg) \\
  &\, + n2! {\binom{n-1} 2} \int_x W_{n-3} \bigg(\int_{y,z} \varphi_n(x,y,z,\cdot) \int_s \big(K^m_s \hat\otimes \rho^m_s \big)(y,z) \cdot \nabla_x \rho^m_x(s) \bigg) .
  \endaligned
  \end{equation}

First, we show that the last term in the above formula vanishes. We have
  $$\aligned
  &\, \int_{x,y,z,s} \varphi_n(x,y,z,\cdot) \big(K^m_s \hat\otimes \rho^m_s \big)(y,z) \cdot \nabla_x \rho^m_x(s) \\
  =&\, -\frac12 \int_{x,y,z,s} \varphi_n(x,y,z,\cdot) \big(K^m_s(y) \rho^m_s(z) + K^m_s(z) \rho^m_s(y)\big) \cdot \nabla_s \rho^m_s(x).
  \endaligned $$
By the symmetry of $\varphi_n(x,y,z,\cdot)$ with respect to $x,z$, we obtain
  $$\aligned
  &\, \int_{x,y,z,s} \varphi_n(x,y,z,\cdot) K^m_s(y) \rho^m_s(z) \cdot \nabla_s \rho^m_s(x) \\
  =&\, \frac12 \int_{x,y,z} \varphi_n(x,y,z,\cdot) \int_s K^m_s(y)  \cdot \nabla_s \big( \rho^m_s(z)\rho^m_s(x) \big) \\
  =&\, - \frac12 \int_{x,y,z} \varphi_n(x,y,z,\cdot) \int_s \div_s (K^m_s(y)) \rho^m_s(z)\rho^m_s(x) =0,
  \endaligned $$
where in the second step we have used integration by parts with respect to $s$. In the same way,
  $$\int_{x,y,z,s} \varphi_n(x,y,z,\cdot) K^m_s(z) \rho^m_s(y) \cdot \nabla_s \rho^m_s(x) =0.$$
Summarizing these facts we conclude that the last term in \eqref{lem-1.1} vanishes, thus we obtain the two formulae for $\mathcal G^m_+ W_n(\varphi_n)$ and $\mathcal G^m_- W_n(\varphi_n)$.

Now we prove the last identity. We have
  $$\aligned
  &\, \big\< W_{n+1}(\varphi_{n+1}), \mathcal G^m_+ W_n(\varphi_n) \big\> \\
  =&\, (n+1)! \int_{r_{1:(n+1)}} \varphi_{n+1}(r_{1:(n+1)})\, n \int_{x,s} \varphi_n(x,r_{3:(n+1)}) \big(K^m_s \hat\otimes \rho^m_s \big)(r_1,r_2) \cdot \nabla_x \rho^m_x(s) \\
  =&\, (n+1)!n \int_{r_{1:(n+1)}} \varphi_{n+1}(r_{1:(n+1)}) \int_{x,s} \varphi_n(x,r_{3:(n+1)})  \rho^m_s(x) \div_s\big(K^m_s \hat\otimes \rho^m_s \big)(r_1,r_2),
  \endaligned $$
where we have used $\nabla_x \rho^m_x(s) = -\nabla_s \rho^m_s(x)$ and integration by parts. Since $\div_s K^m_s=0$, one has
  $$\aligned
  &\, \big\< W_{n+1}(\varphi_{n+1}), \mathcal G^m_+ W_n(\varphi_n) \big\> \\
  =&\, (n+1)! n \int_{r_{1:(n+1)},x,s} \varphi_{n+1}(r_{1:(n+1)}) \varphi_n(x,r_{3:(n+1)})  \rho^m_x(s) \big(K^m_s\hat\otimes \nabla_s\rho^m_s \big)(r_1,r_2) .
  \endaligned $$
We change variables as follows: $r_1\leftrightarrow x, r_2\rightarrow y, r_i\to r_{i-1}$ for $i\geq 3$; then
  $$\aligned
  &\, \big\< W_{n+1}(\varphi_{n+1}), \mathcal G^m_+ W_n(\varphi_n) \big\> \\
  =&\, (n+1)! n \int_{x, y, r_{2:n},r_1,s} \varphi_{n+1}(x,y, r_{2:n}) \varphi_n(r_1,r_{2:n})  \rho^m_{r_1}(s) \big(K^m_s\hat\otimes \nabla_s\rho^m_s \big)(x,y) \\
  =&\, (n+1)! n \int_{x, y, r_{1:n},s} \varphi_{n+1}(x,y, r_{2:n}) \varphi_n(r_{1:n})  \rho^m_{ s}(r_1) \big(K^m_s\hat\otimes \nabla_s\rho^m_s \big)(x,y) .
  \endaligned $$
Note that
  $$\aligned
  \big(K^m_s\hat\otimes \nabla_s\rho^m_s \big)(x,y)&= \frac12 \big(K^m_s(x)\cdot \nabla_s\rho^m_s(y) + K^m_s(y)\cdot \nabla_s\rho^m_s(x)\big) \\
  &= -\frac12 \big(K^m_s(x)\cdot \nabla_y\rho^m_y(s) + K^m_s(y)\cdot \nabla_x\rho^m_x(s)\big),
  \endaligned $$
we arrive at
  \begin{equation}\label{lem-1.2}
  \aligned
  &\, \big\< W_{n+1}(\varphi_{n+1}), \mathcal G^m_+ W_n(\varphi_n) \big\> \\
  =&\, -(n+1)! \frac n2 \int_{x, y, r_{1:n},s} \varphi_{n+1}(x,y, r_{2:n}) \varphi_n(r_{1:n}) \rho^m_{ s}(r_1) K^m_s(x)\cdot \nabla_y\rho^m_y(s) \\
  &\, -(n+1)! \frac n2 \int_{x, y, r_{1:n},s} \varphi_{n+1}(x,y, r_{2:n}) \varphi_n(r_{1:n}) \rho^m_{ s}(r_1) K^m_s(y)\cdot \nabla_x\rho^m_x(s) .
  \endaligned
  \end{equation}

Now by the second formula, we have
  $$\mathcal G^m_- W_{n+1}(\varphi_{n+1}) = 2(n+1)n W_{n} \bigg(\int_{x,y,s} \varphi_{n+1}(x,y,\cdot) \big(K^m_s\hat\otimes \rho^m_s \big)(y,\cdot) \cdot \nabla_x \rho^m_x(s) \bigg),$$
therefore
  $$\aligned
  &\, \big\< \mathcal G^m_- W_{n+1}(\varphi_{n+1}), W_n(\varphi_n) \big\> \\
  =&\, n! \int_{r_{1:n}} \varphi_{n}(r_{1:n})\, 2(n+1)n \int_{x,y,s} \varphi_{n+1}(x,y,r_{2:n}) \big(K^m_s\hat\otimes \rho^m_s \big)(y,r_1) \cdot \nabla_x \rho^m_x(s) \\
  =&\, (n+1)! n \int_{r_{1:n}, x,y,s} \varphi_{n}(r_{1:n}) \varphi_{n+1}(x,y,r_{2:n}) \big(K^m_s(y) \rho^m_s(r_1) + K^m_s(r_1) \rho^m_s(y) \big)  \cdot \nabla_x \rho^m_x(s) .
  \endaligned $$
Note that $\nabla_x \rho^m_x(s)= -\nabla_s \rho^m_s(x)$, one has, by the symmetry of $\varphi_{n+1}(x,y,r_{2:n})$ with respect to $x,y$,
  $$\aligned
  &\, \int_{r_{1:n}, x,y,s} \varphi_{n}(r_{1:n}) \varphi_{n+1}(x,y,r_{2:n}) K^m_s(r_1) \rho^m_s(y) \cdot \nabla_x \rho^m_x(s) \\
  =&\, - \int_{r_{1:n}, x,y,s} \varphi_{n}(r_{1:n}) \varphi_{n+1}(x,y,r_{2:n}) K^m_s(r_1) \rho^m_s(y) \cdot \nabla_s \rho^m_s(x) \\
  =& \, - \frac12 \int_{r_{1:n}, x,y,s} \varphi_{n}(r_{1:n}) \varphi_{n+1}(x,y,r_{2:n}) K^m_s(r_1) \cdot \nabla_s \big( \rho^m_s(y) \rho^m_s(x) \big) \\
  =&\, \frac12 \int_{r_{1:n}, x,y,s} \varphi_{n}(r_{1:n}) \varphi_{n+1}(x,y,r_{2:n}) \div_s( K^m_s(r_1)) \big( \rho^m_s(y) \rho^m_s(x) \big) =0,
  \endaligned $$
where we have used the integration by parts formula and the fact $\div_s( K^m_s(r_1))=0$. Again by the symmetry of $\varphi_{n+1}(x,y,r_{2:n})$ with respect to $x,y$,
  $$\aligned
  &\, \big\< \mathcal G^m_- W_{n+1}(\varphi_{n+1}), W_n(\varphi_n) \big\> \\
  =&\, (n+1)! n \int_{r_{1:n}, x,y,s} \varphi_{n}(r_{1:n}) \varphi_{n+1}(x,y,r_{2:n}) \rho^m_s(r_1) K^m_s(y) \cdot \nabla_x \rho^m_x(s) \\
  =&\, (n+1)! \frac n2 \int_{r_{1:n}, x,y,s} \varphi_{n}(r_{1:n}) \varphi_{n+1}(x,y,r_{2:n}) \rho^m_s(r_1) \big( K^m_s(y) \cdot \nabla_x \rho^m_x(s) + K^m_s(x) \cdot \nabla_y \rho^m_y(s) \big) .
  \endaligned $$
Combining this equality with \eqref{lem-1.2}, we complete the proof.
\end{proof}

Now we prove the expressions of the drift operators.

\begin{proof}[Proof of Lemma \ref{lem-drift}]
The first formula is obvious. We begin with the proof of the second one. By Lemma \ref{lem-1}, the kernel for $(\mathcal G^m_+ \varphi)_n$ has the Fourier transform
  $$\aligned
  \mathcal F(\mathcal G^m_+ \varphi)_n(k_{1:n}) =&\, (n-1) \int_{r_{1:n}} {\rm e}^{-2\pi {\rm i} k_{1:n} \cdot r_{1:n}} \int_{x,s} \varphi_{n-1}(x,r_{3:n}) \big(K^m_s\hat\otimes \rho^m_s \big)(r_1, r_2) \cdot \nabla_x \rho^m_x(s) \\
  =&\, \frac12 (n-1) \int_{r_{1:n},x,s} {\rm e}^{-2\pi {\rm i} k_{1:n} \cdot r_{1:n}} \varphi_{n-1}(x,r_{3:n}) \rho^m_s (r_2) K^m_s(r_1) \cdot \nabla_x \rho^m_x(s) \\
  &\, + \frac12 (n-1) \int_{r_{1:n},x,s} {\rm e}^{-2\pi {\rm i} k_{1:n} \cdot r_{1:n}} \varphi_{n-1}(x,r_{3:n}) \rho^m_s (r_1) K^m_s(r_2) \cdot \nabla_x \rho^m_x(s)\\
  =: & \, I_1 + I_2.
  \endaligned $$
We first compute $I_1$. Recall that $\rho^m_s (r_2)= \rho^m(s-r_2) = \sum_{|l|\leq m} {\rm e}^{2\pi {\rm i} l \cdot (s-r_2)}$ and $K^m_s(r_1)= K^m_\eps(s-r_1)= \frac{\rm i}{(2\pi)^\eps} \sum_{|l|\leq m} \frac{l^\perp}{|l|^{1+\eps}} {\rm e}^{2\pi {\rm i} l \cdot (s-r_1)}$, we have
  $$\aligned
  &\, \int_{r_{1:2}} {\rm e}^{-2\pi {\rm i} k_{1:2} \cdot r_{1:2}} K^m_s(r_1) \rho^m_s (r_2) = {\bf 1}_{|k_2|\leq m} {\rm e}^{-2\pi {\rm i} k_2 \cdot s} \int_{r_1} {\rm e}^{-2\pi {\rm i} k_{1} \cdot r_{1}} K^m(s-r_1)\\
  =&\, {\bf 1}_{|k_2|\leq m} {\rm e}^{-2\pi {\rm i} k_2 \cdot s} \frac{\rm i}{(2\pi)^\eps} \sum_{|l|\leq m} \frac{l^\perp}{|l|^{1+\eps}} \int_{r_1} {\rm e}^{-2\pi {\rm i} k_{1} \cdot r_{1}} {\rm e}^{2\pi {\rm i} l \cdot (s-r_1)}\\
  =&\, \frac{-\rm i}{(2\pi)^\eps} {\bf 1}_{|k_1|, |k_2|\leq m} \frac{k_1^\perp}{|k_1|^{1+\eps}} {\rm e}^{-2\pi {\rm i} (k_1+k_2) \cdot s} .
  \endaligned $$
As a result,
  $$\aligned
  &\, \int_{r_{1:2}, s} {\rm e}^{-2\pi {\rm i} k_{1:2} \cdot r_{1:2}} \rho^m_s (r_2) K^m_s(r_1)\cdot \nabla_x \rho^m_x(s) \\
  =&\, \frac{-\rm i}{(2\pi)^\eps} {\bf 1}_{|k_1|, |k_2|\leq m} \frac{k_1^\perp}{|k_1|^{1+\eps}} \cdot \int_s {\rm e}^{-2\pi {\rm i} (k_1+k_2) \cdot s} \sum_{|l|\leq m} 2\pi {\rm i} l {\rm e}^{ 2\pi {\rm i} l \cdot (x-s)} \\
  =&\, (2\pi)^{1-\eps} {\bf 1}_{|k_1|, |k_2|, |k_1+k_2|\leq m} \frac{k_1^\perp\cdot (k_1+k_2)}{|k_1|^{1+\eps}} {\rm e}^{-2\pi {\rm i} (k_1+k_2) \cdot x} .
  \endaligned $$
Substituting this identity into the expression of $I_1$ yields
  $$\aligned
  I_1= &\, \frac12 (2\pi)^{1-\eps} (n-1)  {\bf 1}_{|k_1|, |k_2|, |k_1+k_2|\leq m} \frac{k_1^\perp\cdot (k_1+k_2)}{|k_1|^{1+\eps}} \\
  & \times \int_{r_{3:n},x} {\rm e}^{-2\pi {\rm i} k_{3:n} \cdot r_{3:n}} \varphi_{n-1}(x,r_{3:n}) {\rm e}^{-2\pi {\rm i} (k_1+k_2) \cdot x} \\
  = &\, \frac12 (2\pi)^{1-\eps} (n-1)  {\bf 1}_{|k_1|, |k_2|, |k_1+k_2|\leq m} \frac{k_1^\perp\cdot (k_1+k_2)}{|k_1|^{1+\eps}} \hat \varphi_{n-1} (k_1+k_2, k_{3:n}).
  \endaligned $$
The computation of $I_2$ is the same as above, thus we obtain
  $$I_2= \frac12 (2\pi)^{1-\eps} (n-1)  {\bf 1}_{|k_1|, |k_2|, |k_1+k_2|\leq m} \frac{k_2^\perp\cdot (k_1+k_2)}{|k_2|^{1+\eps}} \hat \varphi_{n-1} (k_1+k_2, k_{3:n}),$$
which, together with $I_1$, gives us the formula of $\mathcal F(\mathcal G^m_+ \varphi)_n(k_{1:n})$.

Next we prove the last formula:
  $$\aligned
  \mathcal F(\mathcal G^m_- \varphi)_n(k_{1:n}) =&\, 2(n+1)n \int_{r_{1:n}} {\rm e}^{-2\pi {\rm i} k_{1:n} \cdot r_{1:n}} \int_{x,y,s} \varphi_{n+1}(x,y,r_{2:n}) \big(K^m_s\hat\otimes \rho^m_s \big)(y,r_1) \cdot \nabla_x \rho^m_x(s) \\
  =&\, (n+1)n \int_{r_{1:n,x,y,s}} {\rm e}^{-2\pi {\rm i} k_{1:n} \cdot r_{1:n}} \varphi_{n+1}(x,y,r_{2:n}) \rho^m_s(r_1) K^m_s(y) \cdot \nabla_x \rho^m_x(s) \\
  &\, + (n+1)n \int_{r_{1:n},x,y,s} {\rm e}^{-2\pi {\rm i} k_{1:n} \cdot r_{1:n}} \varphi_{n+1}(x,y,r_{2:n}) \rho^m_s(y) K^m_s(r_1) \cdot \nabla_x \rho^m_x(s)\\
  =:&\, J_1+ J_2.
  \endaligned $$
We start with $J_1$. Note that
  $$\int_{r_{1}} {\rm e}^{-2\pi {\rm i} k_{1} \cdot r_{1}} \rho^m_s(r_1) = {\bf 1}_{|k_1|\leq m} {\rm e}^{-2\pi {\rm i} k_{1} \cdot s} , $$
we have
  $$J_1= (n+1)n {\bf 1}_{|k_1|\leq m} \int_{r_{2:n,x,y,s}} {\rm e}^{-2\pi {\rm i} k_{2:n} \cdot r_{2:n}}  {\rm e}^{-2\pi {\rm i} k_{1} \cdot s} \varphi_{n+1}(x,y,r_{2:n}) K^m_s(y) \cdot \nabla_x \rho^m_x(s). $$
Moreover,
  $$\aligned
  \int_s {\rm e}^{-2\pi {\rm i} k_{1} \cdot s} K^m_s(y) \cdot \nabla_x \rho^m_x(s) &= \sum_{|p|,|q|\leq m} \int_s {\rm e}^{-2\pi {\rm i} k_{1} \cdot s} \frac{\rm i}{(2\pi)^\eps} \frac{p^\perp}{|p|^{1+\eps}} {\rm e}^{2\pi {\rm i} p \cdot (s-y)} \cdot (2\pi {\rm i} q) {\rm e}^{2\pi {\rm i} q \cdot (x-s)} \\
  &= -(2\pi)^{1-\eps} \sum_{|p|,|q|\leq m} \frac{p^\perp \cdot q}{|p|^{1+\eps}} \int_s {\rm e}^{-2\pi {\rm i} k_{1} \cdot s} {\rm e}^{2\pi {\rm i} p \cdot (s-y)}  {\rm e}^{2\pi {\rm i} q \cdot (s-x)} \\
  &= - (2\pi)^{1-\eps} \sum_{|p|,|q|\leq m, p+q=k_1} \frac{p^\perp \cdot q}{|p|^{1+\eps}} {\rm e}^{-2\pi {\rm i} (p \cdot y + q\cdot x)} ,
  \endaligned $$
therefore
  $$\aligned
  J_1=&\, -(2\pi)^{1-\eps} (n+1)n {\bf 1}_{|k_1|\leq m}\sum_{|p|,|q|\leq m, p+q=k_1} \frac{p^\perp \cdot q}{|p|^{1+\eps}} \\
  &\, \times \int_{r_{2:n,x,y}} {\rm e}^{-2\pi {\rm i} k_{2:n} \cdot r_{2:n}} {\rm e}^{-2\pi {\rm i} (p \cdot y + q\cdot x)} \varphi_{n+1}(x,y,r_{2:n}) \\
  =&\, -(2\pi)^{1-\eps} (n+1)n {\bf 1}_{|k_1|\leq m}\sum_{|p|,|q|\leq m, p+q=k_1} \frac{p^\perp \cdot q}{|p|^{1+\eps}} \hat \varphi_{n+1}(q,p,k_{2:n}).
  \endaligned $$
Using the symmetry of $\varphi_{n+1}$, we have $\hat \varphi_{n+1}(q,p,k_{2:n}) =\hat \varphi_{n+1}(p,q,k_{2:n})$, hence
  $$J_1= -\frac12(2\pi)^{1-\eps} (n+1)n {\bf 1}_{|k_1|\leq m}\sum_{|p|,|q|\leq m, p+q=k_1} \bigg(\frac{p^\perp \cdot q}{|p|^{1+\eps}}+ \frac{q^\perp \cdot p}{|q|^{1+\eps}}\bigg) \hat \varphi_{n+1}(p,q,k_{2:n}).$$

It remains to compute $J_2$. First,
  $$\aligned
  &\, \int_{r_1,s} {\rm e}^{-2\pi {\rm i} k_{1} \cdot r_{1}} \rho^m_s(y) K^m_s(r_1) \cdot \nabla_x \rho^m_x(s) = \frac{- \rm i}{(2\pi)^\eps } {\bf 1}_{|k_1|\leq m} \frac{k_1^\perp}{|k_1|^{1+\eps}} \cdot \int_s {\rm e}^{-2\pi {\rm i} k_{1} \cdot s}\rho^m_s(y) \nabla_x \rho^m_x(s) \\
  =&\, \frac{- \rm i}{(2\pi)^\eps } {\bf 1}_{|k_1|\leq m} \frac{k_1^\perp}{|k_1|^{1+\eps}} \cdot \sum_{|p|, |q|\leq m} \int_s {\rm e}^{-2\pi {\rm i} k_{1} \cdot s} {\rm e}^{2\pi {\rm i} p \cdot (s-y)} (2\pi{\rm i} q) {\rm e}^{2\pi {\rm i} q \cdot (x-s)}  \\
  =&\, (2\pi)^{1-\eps} {\bf 1}_{|k_1|\leq m} \sum_{|p|, |q|\leq m} \frac{k_1^\perp \cdot q}{|k_1|^{1+\eps}} \int_s {\rm e}^{-2\pi {\rm i} k_{1} \cdot s} {\rm e}^{2\pi {\rm i} p \cdot (s-y)}  {\rm e}^{2\pi {\rm i} q \cdot (s-x)} \\
  =&\, (2\pi)^{1-\eps} {\bf 1}_{|k_1|\leq m} \sum_{|p|, |q|\leq m, p+q =k_1} \frac{k_1^\perp \cdot q}{|k_1|^{1+\eps}} {\rm e}^{-2\pi {\rm i} (p \cdot y+ q \cdot x)} .
  \endaligned $$
Thus,
  $$\aligned
  J_2= &\, (2\pi)^{1-\eps} (n+1)n {\bf 1}_{|k_1|\leq m} \sum_{|p|, |q|\leq m, p+q =k_1} \frac{k_1^\perp \cdot q}{|k_1|^{1+\eps}} \\
  &\times \int_{r_{2:n},x,y} {\rm e}^{-2\pi {\rm i} k_{2:n} \cdot r_{2:n}} \varphi_{n+1}(x,y,r_{2:n}) {\rm e}^{-2\pi {\rm i} (p \cdot y+ q \cdot x)} \\
  = &\, (2\pi)^{1-\eps} (n+1)n {\bf 1}_{|k_1|\leq m} \sum_{|p|, |q|\leq m, p+q =k_1} \frac{k_1^\perp \cdot q}{|k_1|^{1+\eps}} \hat \varphi_{n+1}(q,p,k_{2:n}).
  \endaligned $$
By the symmetry of $\hat \varphi_{n+1}(q,p,k_{2:n})$ with respect to $p,q$, we obtain
  $$\aligned
  J_2&= \frac12 (2\pi)^{1-\eps} (n+1)n {\bf 1}_{|k_1|\leq m} \sum_{|p|, |q|\leq m, p+q =k_1} \bigg( \frac{k_1^\perp \cdot q}{|k_1|^{1+\eps}} + \frac{k_1^\perp \cdot p}{|k_1|^{1+\eps}} \bigg) \hat \varphi_{n+1}(q,p,k_{2:n}) \\
  &= \frac12 (2\pi)^{1-\eps} (n+1)n {\bf 1}_{|k_1|\leq m} \sum_{|p|, |q|\leq m, p+q =k_1} \frac{k_1^\perp \cdot (q+p)}{|k_1|^{1+\eps}} \hat \varphi_{n+1}(q,p,k_{2:n}) =0
  \endaligned $$
since $k_1^\perp \cdot (q+p) = k_1^\perp \cdot k_1=0$. Combining this fact with $J_1$ we finally obtain the formula of $\mathcal F(\mathcal G^m_- \varphi)_n(k_{1:n})$.
\end{proof}

For the proof of Lemma \ref{lem-apriori}, we need the following preparation which is slightly different from \cite[Lemma 16]{GT} since we allow $\beta<0$.

\begin{lemma}\label{lem-auxiliary}
Let $C\geq 0,\,\theta>0, \,\beta>-d$, and $\alpha >(d+\beta)/\theta$. Then
  $$\sum_{l\in \Z^d_0} \frac{|l|^\beta}{\big(|l|^\theta + |k-l|^\theta + C\big)^\alpha} \lesssim (|k|^\theta +C)^{(\beta+d)/\theta -\alpha}. $$
\end{lemma}

\begin{proof}
We have $|l|^\theta + |k-l|^\theta \gtrsim |l|^\theta + |k|^\theta$, thus
  $$\sum_{l\in \Z^d_0} \frac{|l|^\beta}{\big(|l|^\theta + |k-l|^\theta + C\big)^\alpha} \lesssim \sum_{l\in \Z^d_0} \frac{|l|^\beta}{\big(|l|^\theta + |k|^\theta + C\big)^\alpha} = \sum_{l\in \Z^d_0} \frac{|l|^\beta}{\big(|l|^\theta + a^\theta \big)^\alpha}, $$
where we set $a=(|k|^\theta + C)^{1/\theta}$. Transforming into integral and using spherical coordinates, we obtain
  $$\aligned
  \sum_{l\in \Z^d_0} \frac{|l|^\beta}{\big(|l|^\theta + a^\theta \big)^\alpha} &\lesssim \int_{|x|\geq 1} \frac{|x|^\beta}{\big( |x|^\theta + a^\theta\big)^\alpha}\,\d x \lesssim  \int_1^\infty \frac{r^{\beta+d-1}}{(r^\theta+ a^\theta )^\alpha} \,\d r \\
  &= \bigg(\int_1^a + \int_a^\infty \bigg) \frac{r^{\beta+d-1}}{(r^\theta+ a^\theta )^\alpha} \,\d r =: I_1 + I_2.
  \endaligned $$
For $r\in [1,a]$, we have $a^\theta \leq r^\theta+ a^\theta\leq 2a^\theta $, thus
  $$I_1\leq \int_1^a  \frac{r^{\beta+d-1}}{a^{\theta \alpha} } \,\d r = \frac1{(\beta+d) a^{\theta \alpha}} (a^{\beta+d}-1) \lesssim a^{\beta+d- \theta \alpha}.$$
Next,
  $$I_2\leq \int_a^\infty \frac{r^{\beta+d-1}}{r^{\theta\alpha}} \,\d r = \int_a^\infty r^{\beta+d- \theta\alpha -1} \,\d r= \frac1{(\theta\alpha-\beta-d) a^{\theta\alpha-\beta-d}} \lesssim  a^{\beta+d- \theta \alpha}. $$
Combining these estimates with the definition of $a$, we finish the proof.
\end{proof}

Now we are ready to provide the

\begin{proof}[Proof of Lemma \ref{lem-apriori}]
We write $|k_{1:n}|^2= |k_1|^2 + \cdots + |k_n|^2$. By Lemma \ref{lem-drift},
  $$\aligned
  \big\|w(\mathcal N) (-\mathcal L_0)^{-\gamma} \mathcal G^m_+ \varphi \big\|^2
  =&\, \sum_{n\geq 0} n!\, w(n)^2 \sum_{k_{1:n}} \big( |k_{1:n}|^2 \big)^{-2\gamma} \big|\mathcal F(\mathcal G^m_+ \varphi)_n(k_{1:n})\big|^2 \\
  = &\, \sum_{n\geq 0} n!\, w(n)^2 (n-1)^2 \sum_{k_{1:n}} \big( |k_{1:n}|^2 \big)^{-2\gamma} C_\eps^2 {\bf 1}_{|k_1|, |k_2|, |k_1+k_2|\leq m} \\
  &\, \times \bigg| \bigg(\frac{k_1^\perp}{|k_1|^{1+\eps}}+  \frac{k_2^\perp}{|k_2|^{1+\eps}} \bigg)\cdot (k_1+k_2) \bigg|^2 |\hat \varphi_{n-1}(k_1+k_2, k_{3:n})|^2.
  \endaligned $$
Note that
  $$\aligned
  &\, \sum_{k_{1:n}} \big( |k_{1:n}|^2 \big)^{-2\gamma} {\bf 1}_{|k_1|, |k_2|, |k_1+k_2|\leq m} \bigg| \bigg(\frac{k_1^\perp}{|k_1|^{1+\eps}}+  \frac{k_2^\perp}{|k_2|^{1+\eps}} \bigg)\cdot (k_1+k_2) \bigg|^2 |\hat \varphi_{n-1}(k_1+k_2, k_{3:n})|^2 \\
  \leq &\, \sum_{k_{1:n}} \big( |k_{1:n}|^2 \big)^{-2\gamma} {\bf 1}_{|k_1|, |k_2|, |k_1+k_2|\leq m} \bigg|\frac{k_1^\perp }{|k_1|^{1+\eps}}+  \frac{k_2^\perp }{|k_2|^{1+\eps}} \bigg|^2 |k_1+k_2|^2 |\hat \varphi_{n-1}(k_1+k_2, k_{3:n})|^2 \\
  \leq &\, \sum_{l, k_{3:n}} |l|^2 |\hat \varphi_{n-1}(l, k_{3:n})|^2 \sum_{k_1+k_2=l} \big( |k_{1:n}|^2 \big)^{-2\gamma} \bigg| \frac{k_1^\perp}{|k_1|^{1+\eps}}+  \frac{k_2^\perp}{|k_2|^{1+\eps}} \bigg|^2 .
  \endaligned $$
We have the following estimate:
  $$\aligned
  \sum_{k_1+k_2=l} \big( |k_{1:n}|^2 \big)^{-2\gamma} \bigg| \frac{k_1^\perp}{|k_1|^{1+\eps}}+  \frac{k_2^\perp}{|k_2|^{1+\eps}} \bigg|^2
  \leq &\, 2 \sum_{k_1+k_2=l} \big( |k_{1:n}|^2 \big)^{-2\gamma} \bigg( \frac{1}{|k_1|^{2\eps}}+ \frac{1}{|k_2|^{2\eps}} \bigg) \\
  \lesssim &\, \sum_{k_1+k_2=l} \frac{|k_1|^{-2\eps}}{\big( |k_{1:n}|^2 \big)^{2\gamma}}\\
  =& \sum_{k_1\in \Z^2_0} \frac{|k_1|^{-2\eps}}{\big( |k_1|^2 + |l-k_1|^2+|k_{3:n}|^2 \big)^{2\gamma}}.
  \endaligned $$
For $\gamma>(1-\eps)/2$, applying Lemma \ref{lem-auxiliary} with $\theta=2,\, \alpha=2\gamma$ and $\beta=-2\eps, d=2$, we obtain
  $$\sum_{k_1+k_2=l} \big( |k_{1:n}|^2 \big)^{-2\gamma} \bigg| \frac{k_1^\perp}{|k_1|^{1+\eps}}+  \frac{k_2^\perp}{|k_2|^{1+\eps}} \bigg|^2 \lesssim \frac1{\big( |l|^2+|k_{3:n}|^2 \big)^{2\gamma+\eps -1}}. $$
As a result,
  $$\aligned
  &\, \sum_{k_{1:n}} \big( |k_{1:n}|^2 \big)^{-2\gamma} {\bf 1}_{|k_1|, |k_2|, |k_1+k_2|\leq m} \bigg| \bigg(\frac{k_1^\perp}{|k_1|^{1+\eps}}+  \frac{k_2^\perp}{|k_2|^{1+\eps}} \bigg)\cdot (k_1+k_2) \bigg|^2 |\hat \varphi_{n-1}(k_1+k_2, k_{3:n})|^2 \\
  \lesssim &\, \sum_{l,k_{3:n}} |l|^2 |\hat \varphi_{n-1}(l, k_{3:n})|^2 \frac1{\big( |l|^2+|k_{3:n}|^2 \big)^{2\gamma+\eps -1}} \\
  \leq&\, \frac1{n-1} \sum_{l,k_{3:n}} \big( |l|^2+|k_{3:n}|^2 \big) |\hat \varphi_{n-1}(l, k_{3:n})|^2 \frac1{\big( |l|^2+|k_{3:n}|^2 \big)^{2\gamma+\eps -1}},
  \endaligned $$
which is equal to
  $$\aligned  \frac1{n-1} \sum_{l,k_{3:n}} \big( |l|^2+|k_{3:n}|^2 \big)^{2-2\gamma-\eps} |\hat \varphi_{n-1}(l, k_{3:n})|^2
  = \frac1{n-1} \sum_{k_{1:n-1}} |k_{1:n-1}|^{2(2-2\gamma -\eps)} |\hat \varphi_{n-1}(k_{1:n-1})|^2.
  \endaligned $$
Consequently,
  $$\aligned
  \big\|w(\mathcal N) (-\mathcal L_0)^{-\gamma} \mathcal G^m_+ \varphi \big\|^2 &\lesssim \sum_{n\geq 0} n!\, w(n)^2 n \sum_{k_{1:n-1}} |k_{1:n-1}|^{2(2-2\gamma -\eps)} |\hat \varphi_{n-1}(k_{1:n-1})|^2 \\
  &= \sum_{n\geq 1} (n-1)!\, w(n)^2 n^2 \sum_{k_{1:n-1}} |k_{1:n-1}|^{2(2-2\gamma -\eps)} |\hat \varphi_{n-1}(k_{1:n-1})|^2 \\
  &= \big\|w(1+\mathcal N) (1+\mathcal N) (-\mathcal L_0)^{1-\gamma -\eps/2} \varphi \big\|^2 .
  \endaligned $$

Next we prove the second inequality. First,
  $$\big\|w(\mathcal N) (-\mathcal L_0)^{-\gamma} \mathcal G^m_- \varphi \big\|^2 \backsimeq \sum_{n\geq 0} n!\, w(n)^2 \sum_{k_{1:n}} \big(|k_{1:n}|^2 \big)^{-2\gamma} \big|\mathcal F(\mathcal G^m_- \varphi)_n(k_{1:n})\big|^2. $$
By Lemma \ref{lem-drift},
  $$\aligned
  \big|\mathcal F(\mathcal G^m_- \varphi)_n(k_{1:n})\big|^2 & \lesssim n^4\bigg| \sum_{p+q=k_1} {\bf 1}_{|k_1|,|p|,|q|\leq m} \bigg(\frac{p^\perp \cdot q}{|p|^{1+\eps}}+ \frac{q^\perp \cdot p}{|q|^{1+\eps}}\bigg) \hat \varphi_{n+1}(p,q,k_{2:n}) \bigg|^2 \\
  &= n^4\bigg| \sum_{p+q=k_1} {\bf 1}_{|k_1|,|p|,|q|\leq m} \bigg(\frac{p^\perp}{|p|^{1+\eps}}+ \frac{q^\perp}{|q|^{1+\eps}}\bigg)\cdot k_1\, \hat \varphi_{n+1}(p,q,k_{2:n}) \bigg|^2 \\
  &\leq n^4 |k_1|^2 \bigg[ \sum_{p+q=k_1} {\bf 1}_{|k_1|,|p|,|q|\leq m} \bigg(\frac1{|p|^\eps} + \frac1{|q|^\eps}\bigg) |\hat \varphi_{n+1}(p,q,k_{2:n})| \bigg]^2 .
  \endaligned$$
We have, by Cauchy's inequality,
  $$\aligned
  &\, \bigg[ \sum_{p+q=k_1} {\bf 1}_{|k_1|,|p|,|q|\leq m} \bigg(\frac1{|p|^\eps} + \frac1{|q|^\eps}\bigg) |\hat \varphi_{n+1}(p,q,k_{2:n})| \bigg]^2 \\
  =&\, \bigg[ \sum_{p+q=k_1} {\bf 1}_{|k_1|,|p|,|q|\leq m} \bigg(\frac1{|p|^\eps} + \frac1{|q|^\eps}\bigg) \frac1{(|p|^2 +|q|^2)^{1 -\gamma -\eps/4}} (|p|^2 +|q|^2)^{1 -\gamma -\eps/4} |\hat \varphi_{n+1}(p,q,k_{2:n})| \bigg]^2 \\
  \leq &\, \bigg[\sum_{p+q=k_1} \bigg(\frac1{|p|^\eps} + \frac1{|q|^\eps}\bigg)^2 \frac1{(|p|^2 +|q|^2)^{2 -2\gamma -\eps/2}} \bigg] \sum_{p+q=k_1} (|p|^2 +|q|^2)^{2 -2\gamma -\eps/2} |\hat \varphi_{n+1}(p,q,k_{2:n})|^2 \\
  \leq &\, \frac1{|k_1|^{2(1 -2\gamma) +\eps}} \sum_{p+q=k_1} (|p|^2 +|q|^2)^{2 -2\gamma -\eps/2} |\hat \varphi_{n+1}(p,q,k_{2:n})|^2,
  \endaligned $$
where the last step follows by applying Lemma \ref{lem-auxiliary}, since $\gamma <(1+\eps/2)/2$. Substituting this estimate into the one above, we get
  $$\aligned
  \big|\mathcal F(\mathcal G^m_- \varphi)_n(k_{1:n})\big|^2 &\lesssim \frac{n^4 }{|k_1|^{-4\gamma + \eps}} \sum_{p+q=k_1} (|p|^2 +|q|^2)^{2 -2\gamma -\eps/2} |\hat \varphi_{n+1}(p,q,k_{2:n})|^2 .
  \endaligned $$
Therefore,
  $$\aligned
  &\, \sum_{k_{1:n}} \big(|k_{1:n}|^2 \big)^{-2\gamma} \big|\mathcal F(\mathcal G^m_- \varphi)_n(k_{1:n})\big|^2 \\
  \lesssim &\, \sum_{k_{1:n}} \big(|k_{1:n}|^2 \big)^{-2\gamma} \frac{n^4 }{|k_1|^{-4\gamma +\eps}} \sum_{p+q=k_1} (|p|^2 +|q|^2)^{2 -2\gamma -\eps/2} |\hat \varphi_{n+1}(p,q,k_{2:n})|^2 \\
  \leq &\, n^4  \sum_{k_{1:n}} \frac{|k_1|^{4\gamma -\eps}}{\big(|k_{1:n}|^2 \big)^{2\gamma}} \sum_{p+q=k_1} (|p|^2 +|q|^2)^{2 -2\gamma -\eps/2} |\hat \varphi_{n+1}(p,q,k_{2:n})|^2.
  \endaligned $$
As a consequence,
  $$\aligned
  \sum_{k_{1:n}} \big(|k_{1:n}|^2 \big)^{-2\gamma} \big|\mathcal F(\mathcal G^m_- \varphi)_n(k_{1:n})\big|^2 \leq &\, n^4  \sum_{k_{1:n}} \sum_{p+q=k_1} (|p|^2 +|q|^2)^{2 -2\gamma -\eps/2} |\hat \varphi_{n+1}(p,q,k_{2:n})|^2 \\
  \leq &\, n^4  \sum_{k_{1:n+1}} (|k_1|^2 +|k_2|^2)^{2 -2\gamma -\eps/2} |\hat \varphi_{n+1}(k_{1:n+1})|^2 \\
  \lesssim &\, n^3  \sum_{k_{1:n+1}} (|k_{1:n+1}|^2 )^{2 -2\gamma -\eps/2} |\hat \varphi_{n+1}(k_{1:n+1})|^2 ,
  \endaligned $$
where in the last step we used the symmetry of $\hat \varphi_{n+1}$ in the variables $k_{1:n+1}$ and that $2 -2\gamma -\eps/2\geq 1$. Thus,
  $$\aligned
  \big\|w(\mathcal N) (-\mathcal L_0)^{-\gamma} \mathcal G^m_- \varphi \big\|^2 &\lesssim \sum_{n\geq 0} n!\, w(n)^2 n^3  \sum_{k_{1:n+1}} (|k_{1:n+1}|^2 )^{2 -2\gamma -\eps/2} |\hat \varphi_{n+1}(k_{1:n+1})|^2 \\
  &\lesssim \sum_{n\geq 1} (n+1)!\, w(n)^2 n^2 \big\|(-\mathcal L_0)^{1-\gamma -\eps/4} \varphi_{n+1} \big\|_{L^2(\T^{2(n+1)})}^2 \\
  &\lesssim \big\|w(\mathcal N-1) \mathcal N (-\mathcal L_0)^{1-\gamma -\eps/4} \varphi \big\|^2 .
  \endaligned $$

Finally, we derive the $m$-dependent estimate on $\mathcal G^m$. We have
  $$\aligned
  &\, \sum_{k_{1:n}} \big|\mathcal F(\mathcal G^m_+ \varphi)_n(k_{1:n})\big|^2 \\
  \lesssim &\, (n-1)^2 \sum_{k_{1:n}} {\bf 1}_{|k_1|, |k_2|, |k_1+k_2|\leq m} \bigg| \bigg(\frac{k_1^\perp}{|k_1|^{1+\eps}}+  \frac{k_2^\perp}{|k_2|^{1+\eps}} \bigg)\cdot (k_1+k_2) \bigg|^2 |\hat \varphi_{n-1}(k_1+k_2, k_{3:n})|^2 \\
  \lesssim &\, n^2 \sum_{k_{1:n}} {\bf 1}_{|k_1|, |k_2|, |k_1+k_2|\leq m} \bigg(\frac{1}{|k_1|^{\eps}} +\frac{1}{|k_2|^{\eps}} \bigg)^2 |k_1+ k_2|^2 |\hat \varphi_{n-1}(k_1+k_2, k_{3:n})|^2 \\
  \leq &\, n^2 \sum_{k_{1:n}} {\bf 1}_{|k_1|, |k_2|, |k_1+k_2|\leq m} |k_1+ k_2|^2 |\hat \varphi_{n-1}(k_1+k_2, k_{3:n})|^2 .
  \endaligned $$
Thus,
  $$\aligned
  \sum_{k_{1:n}} \big|\mathcal F(\mathcal G^m_+ \varphi)_n(k_{1:n})\big|^2 &\lesssim n^2 \sum_{l,k_{3:n}} \sum_{k_1+k_2=l} {\bf 1}_{|k_1|, |k_2|, |k_1+k_2|\leq m} |k_1+ k_2|^2 |\hat \varphi_{n-1}(k_1+k_2, k_{3:n})|^2 \\
  &= n^2 \sum_{l,k_{3:n}} {\bf 1}_{|l|\leq m} |l|^2 |\hat \varphi_{n-1}(l, k_{3:n})|^2 \sum_{k_1+k_2=l} {\bf 1}_{|k_1|, |k_2|\leq m} \\
  &\lesssim n^2 m^2 \sum_{l_{1:n-1}} {\bf 1}_{|l_1|\leq m} |l_1|^2 |\hat \varphi_{n-1}(l_{1:n-1})|^2 \\
  &\lesssim n m^2 \sum_{l_{1:n-1}} \big(|l_1|^2+\cdots + |l_{n-1}|^2 \big) |\hat \varphi_{n-1}(l_{1:n-1})|^2 .
  \endaligned $$
Therefore,
  $$ \aligned
  \big\| w(\mathcal N) \mathcal G^m_+ \varphi \big\|^2 & = \sum_{n\geq 0} n!\, w(n)^2 \sum_{k_{1:n}} \big|\mathcal F(\mathcal G^m_+ \varphi)_n(k_{1:n})\big|^2 \\
  &\lesssim \sum_{n\geq 1} n!\, w(n)^2 n m^2 \sum_{l_{1:n-1}} |l_{1:n-1}|^2 |\hat \varphi_{n-1}(l_{1:n-1})|^2 \\
  &\lesssim m^2 \sum_{n\geq 1} (n-1)!\, w(n)^2 n^2 \sum_{l_{1:n-1}} |l_{1:n-1}|^2 |\hat \varphi_{n-1}(l_{1:n-1})|^2 .
  \endaligned $$
From this we obtain
  \begin{equation}\label{lem-apriori.4}
  \aligned
  \big\| w(\mathcal N) \mathcal G^m_+ \varphi \big\|^2 &\leq m^2 \sum_{n\geq 0} n!\, w(n+1)^2 (n+1)^2 \sum_{l_{1:n-1}} |l_{1:n}|^2 |\hat \varphi_{n}(l_{1:n})|^2 \\
  &\leq m^2 \big\| w(\mathcal N+1) (\mathcal N+1) (-\mathcal L_0)^{1/2} \varphi \big\|^2.
  \endaligned
  \end{equation}
Next, in a similar way,
  $$\aligned
  \sum_{k_{1:n}} \big|\mathcal F(\mathcal G^m_- \varphi)_n(k_{1:n})\big|^2 &\lesssim n^4 \sum_{k_{1:n}} \bigg| \sum_{p+q=k_1} {\bf 1}_{|k_1|,|p|,|q|\leq m} \bigg(\frac{p^\perp}{|p|^{1+\eps}}+ \frac{q^\perp}{|q|^{1+\eps}}\bigg)\cdot k_1\, \hat \varphi_{n+1}(p,q,k_{2:n}) \bigg|^2 \\
  &\leq n^4 \sum_{k_{1:n}} |k_1|^2 \bigg[ \sum_{p+q=k_1} {\bf 1}_{|k_1|,|p|,|q|\leq m} \bigg(\frac{1}{|p|^{\eps}}+ \frac{1}{|q|^{\eps}} \bigg) |\hat \varphi_{n+1}(p,q,k_{2:n})| \bigg]^2 \\
  &\leq n^4 \sum_{k_{1:n}} |k_1|^2 \bigg[ \sum_{p+q=k_1} {\bf 1}_{|k_1|,|p|,|q|\leq m} |\hat \varphi_{n+1}(p,q,k_{2:n})| \bigg]^2.
  \endaligned $$
By Cauchy's inequality,
  $$\aligned
  \sum_{k_{1:n}} \big|\mathcal F(\mathcal G^m_- \varphi)_n(k_{1:n})\big|^2 &\lesssim n^4 \sum_{k_{1:n}} |k_1|^2 \bigg[ \sum_{p+q=k_1} {\bf 1}_{|k_1|,|p|,|q|\leq m} \bigg] \sum_{p+q=k_1} |\hat \varphi_{n+1}(p,q,k_{2:n})|^2 \\
  &\leq n^4 m^2 \sum_{k_{1:n}} |k_1|^2 \sum_{p+q=k_1} |\hat \varphi_{n+1}(p,q,k_{2:n})|^2.
  \endaligned $$
Now we have
  $$\aligned
  \sum_{k_{1:n}} \big|\mathcal F(\mathcal G^m_- \varphi)_n(k_{1:n})\big|^2 &\lesssim n^4 m^2 \sum_{k_{1:n}} \sum_{p+q=k_1} \big(|p|^2+ |q|^2 \big) |\hat \varphi_{n+1}(p,q,k_{2:n})|^2 \\
  &\leq n^4 m^2 \sum_{k_{1:n+1}} \big(|k_1|^2+ |k_2|^2 \big) |\hat \varphi_{n+1}(k_{1:n+1})|^2 \\
  &\lesssim n^3 m^2 \sum_{k_{1:n+1}} \big(|k_1|^2+\cdots + |k_{n+1}|^2 \big) |\hat \varphi_{n+1}(k_{1:n+1})|^2 .
  \endaligned $$
Consequently,
  $$\aligned
  \big\| w(\mathcal N) \mathcal G^m_- \varphi \big\|^2&= \sum_{n\geq 0} n!\, w(n)^2 \sum_{k_{1:n}} \big|\mathcal F(\mathcal G^m_- \varphi)_n(k_{1:n})\big|^2 \\
  &\lesssim \sum_{n\geq 0} n!\, w(n)^2 n^3 m^2 \sum_{k_{1:n+1}} \big(|k_1|^2+\cdots + |k_{n+1}|^2 \big) |\hat \varphi_{n+1}(k_{1:n+1})|^2 \\
  &\lesssim m^2 \sum_{n\geq 0} (n+1)!\, w(n)^2 n^2 \sum_{k_{1:n+1}} |k_{1:n+1}|^2 |\hat \varphi_{n+1}(k_{1:n+1})|^2 \\
  &\lesssim m^2 \big\| w(\mathcal N-1) \mathcal N (-\mathcal L_0)^{1/2} \varphi \big\|^2.
  \endaligned $$
Combining this with \eqref{lem-apriori.4} gives us the last estimate.
\end{proof}

\medskip

\noindent \textbf{Acknowledgements.} The first named author is grateful to the financial supports of the National Natural Science Foundation of China (Nos. 11688101, 11931004), and the Youth Innovation Promotion Association, CAS (2017003). The second named author is grateful to the financial supports of the National Natural Science Foundation of China (Nos. 11671035, 11922103) and the financial support by the DFG through the CRC
1283 ``Taming uncertainty and profiting from randomness and low regularity in analysis, stochastics and their
applications''.

\end{document}